\documentclass[opre,nonblindrev]%{informs3opt}%
{informs3hack2} 
\usepackage{comment}
\OneAndAHalfSpacedXI % current default line spacing
\usepackage{endnotes}
\let\footnote=\endnote

\usepackage{natbib}
\usepackage{natbib}
\bibpunct[, ]{(}{)}{,}{a}{}{,}%
\def\bibfont{\small}%
\usepackage{titlesec,soul}
\usepackage[titletoc,toc,title]{appendix}
\usepackage{graphicx}
\usepackage{subfigure}
\usepackage{multirow}
\usepackage[hypertexnames=false,colorlinks=true,breaklinks=true,bookmarks=true,urlcolor=blue,citecolor=black,linkcolor=blue,bookmarksopen=false,draft=false]{hyperref}
\usepackage{todonotes,bm}
\usepackage{tikz}
\usepackage{url}
\usepackage{amssymb}
\usepackage{psfrag}
\usepackage{amsmath}
\usepackage{setspace}
\newcommand{\exclude}[1]{}
\usepackage{bm}
\usepackage[flushleft]{threeparttable}
\usepackage{algorithm}
\usepackage{algpseudocode}
\algdef{SE}[DOWHILE]{Do}{doWhile}{\algorithmicdo}[1]{\algorithmicwhile\ #1}%
\algnewcommand{\Or}{\textbf{or}}
\algnewcommand{\And}{\textbf{and}}

\usepackage{tikz}

\usetikzlibrary{shapes,decorations,arrows,calc,arrows.meta,fit,positioning}
\tikzset{
-Latex,auto,node distance =1 cm and 1 cm,semithick,
state/.style ={ellipse, draw, minimum width = 0.7 cm},
point/.style = {circle, draw, inner sep=0.04cm,fill,node contents={}},
bidirected/.style={Latex-Latex,dashed},
el/.style = {inner sep=2pt, align=left, sloped}
}

\usepackage{thmtools} 
\usepackage{thm-restate}

\usepackage{cleveref}

\declaretheorem[name=Theorem]{theorem}
\declaretheorem[name=Proposition]{proposition}
\declaretheorem[name=Lemma]{lemma}
\declaretheorem[name=Claim]{claim}
\declaretheorem[name=Corollary]{corollary}

\declaretheorem[name=Definition]{definition}
\declaretheorem[name=Example]{example}

\def\k{\tilde k}

\def\Re{\mathbb{R}}

\def\hat{\widehat}

\def \V{\mathbf{ V}}

\def\O{{\mathcal O}}

\def\H{{\mathcal H}}

\def\J{{\mathcal J}}

\def\X{{\mathcal D}}
\def\Y{{\mathcal Y}}

\def\Re{{\mathbb R}}

\def\D{{\mathcal D}}

\def\Q{{\mathcal Q}}
\def\O{{\mathcal O}}

\newcommand{\e}{\mathbf{e}}

\newcommand{\T}{\mathcal{T}}

\DeclareMathOperator{\Proj}{proj}

\DeclareMathOperator{\aff}{aff}

\DeclareMathOperator{\diag}{diag}
\DeclareMathOperator{\Diag}{Diag}
\DeclareMathOperator{\conv}{conv}
\DeclareMathOperator{\tr}{tr}

\DeclareMathOperator{\rank}{rank}

\DeclareMathOperator{\sign}{sign}

\DeclareMathOperator{\rel}{rel}
\DeclareMathOperator{\opt}{opt}

\renewcommand{\S}{\mathcal{S}}
\renewcommand*{\qed}{\hfill\ensuremath{\square}}
\newcommand*{\qedA}{\hfill\ensuremath{\diamond}}

\usepackage{enumerate}

\usepackage{natbib}
 \bibpunct[, ]{(}{)}{,}{a}{}{,}%
 \def\bibfont{\small}%

\allowdisplaybreaks

\begin{document}

\RUNAUTHOR{Y. Li and W. Xie}

%\RUNTITLE{Bundling Information Goods of Decreasing Value}
\RUNTITLE{On the  Partial Convexification of the Low-Rank Spectral Optimization}

% \TITLE{Bundling Information Goods of Decreasing Value}
\TITLE{On the  Partial Convexification of the Low-Rank Spectral Optimization: Rank Bounds and Algorithms}

\ARTICLEAUTHORS{%
\AUTHOR{Yongchun Li}
\AFF{Stewart School of Industrial \& Systems Engineering, Georgia Institute of Technology, Atlanta, GA 30332, \EMAIL{ycli@gatech.edu}} %, \URL{}}
\AUTHOR{Weijun Xie}
\AFF{Stewart School of Industrial \& Systems Engineering, Georgia Institute of Technology, Atlanta, GA 30332, \EMAIL{wxie@gatech.edu}}
} % end of the block

\ABSTRACT{
A Low-rank Spectral Optimization Problem (LSOP) minimizes a linear objective function subject to multiple two-sided linear inequalities intersected with a low-rank and spectral constrained domain set. Although solving LSOP is, in general, NP-hard, its partial convexification (i.e., replacing the domain set by its convex hull) termed ``LSOP-R", is often tractable and yields a high-quality solution. This motivates us to study the strength of LSOP-R. Specifically, we derive rank bounds for any extreme point of the feasible set of LSOP-R with different matrix spaces and prove their tightness.
%for the domain sets with 
The proposed rank bounds recover two well-known results in the literature from a fresh angle and  allow us to derive sufficient conditions under which the relaxation LSOP-R is equivalent to the original LSOP. To effectively solve LSOP-R, we develop a column generation algorithm with a vector-based convex pricing oracle, coupled with a rank-reduction algorithm, which ensures that the output solution always satisfies the theoretical rank bound. Finally, we numerically verify the strength of the LSOP-R and the efficacy of the proposed algorithms.
% In a Low-rank Spectral Optimization Problem (LSOP), it minimizes a linear objective subject to multiple two-sided linear matrix inequalities intersected with a low-rank and  spectral constrained domain set. 
% % Many important problems like matrix completion, fair PCA, and sparse ridge regression involve the proposed LSOP. 
% Although solving LSOP is, in general, NP-hard, its partial convexification (i.e., replacing the domain set by its convex hull) termed ``LSOP-R," is often tractable and yields a high-quality solution.
% This motivates us to study the strength of LSOP-R. Specifically, we derive rank bounds for any extreme point  of the feasible set of LSOP-R and prove their tightness for the domain sets with different matrix spaces. The proposed rank bounds recover two well-known results in literature from a fresh angle and also %. Letting the rank bound be no larger than the low-rank requirement %provides us a sufficient condition
% allow us to derive sufficient conditions 
% under which the  relaxation LSOP-R
% % (termed ``LSOP-R")
% % partial convexification 
% is equivalent to 
% the original LSOP. To effectively solve LSOP-R, we develop a column generation algorithm with a vector-based convex pricing oracle, coupled with a rank-reduction algorithm, which ensures the output solution at least satisifes the theorectical rank bounds. Finally, we numerically verify the strength of the LSOP-R and  the efficiency of proposed algorithms.
}


\maketitle
% \newpage

\section{Introduction}\label{sec1}
This paper studies the Low-rank Spectral Optimization Problem (LSOP) of the form:
\begin{align} \label{eq_rank}
\V_{\opt} :=\min_{{\bm X \in \X}}\left\{\langle\bm A_0, \bm X\rangle: b_i^l
\le \langle \bm A_i, \bm X\rangle \le  b_i^u, \forall i \in [m] \right\}, \tag{LSOP} 
\end{align}
where the decision $\bm X$ is a matrix variable with domain $\X$, $\langle\cdot, \cdot \rangle$ denotes the Frobenius inner product of two equal-sized matrices,  
% the lower or upper bounds of the $i$th two-sided linear inequality can be negative infinite or positive infinite, respectively (i.e., 
we have that $-\infty\leq b_i^l\leq b_i^u\leq +\infty$ for each $i\in [m]$, and  matrices $\bm A_0$ and $\{\bm A_i\}_{i\in [m]}$  are symmetric if the matrix variable $\bm X$ is symmetric. Throughout, we let $\tilde m$ denote the number of linearly independent matrices in the set $\{\bm A_i\}_{i\in [m]}$.
%, e.g., by 

%\begin{align*}
%\V_{\opt} :=\min_{\bm X} \ \ &\langle \bm A_0, \bm X\rangle\\
%&  b_i^l
%\le \langle \bm A_i, \bm X\rangle \le  b_i^u, \forall i \in [m],\\
%& {\bm X \in \X}. 
%\end{align*}
%
%\begin{align*}
%\V_{\rel} :=\min_{\bm X} \ \ &\langle \bm A_0, \bm X\rangle\\
%&  b_i^l
%\le \langle \bm A_i, \bm X\rangle \le  b_i^u, \forall i \in [m],\\
%& {\bm X \in \conv(\X)}. 
%\end{align*}

Specifically, the domain set $\X$ in \ref{eq_rank} is defined as below 
\begin{align}\label{eq_set}
\X:=\left \{\bm X\in \Q:   \rank(\bm X)\le k, F(\bm X):=f(\bm \lambda(\bm X))  \le 0\right\} ,
\end{align}
which 
%is closed and 
consists of a low-rank and a closed convex  spectral constraint.
Note that (i) the matrix space $\Q$ can denote positive semidefinite matrix space $\S_+^n$, non-symmetric matrix space $\Re^{n \times p}$, symmetric indefinite matrix space $\S^n$,  or  diagonal matrix space with  $k\le n\le p$ being positive integers; (ii) function $F(\bm X): \Q \to \Re$ is continuous, convex, and spectral which only depends on the eigenvalue or singular value vector $\bm \lambda(\bm X)
\in \Re^n$ of matrix $\bm X$. Thus, we can rewrite it as $F(\bm X):= f(\bm \lambda(\bm X)) : \Re^n \to \Re$; and (iii)  when there are multiple convex  spectral constraints, i.e., $F_j(\bm X) \le 0, \forall j \in J$,  we can integrate them into a single convex spectral constraint by defining a function $F(\bm X):= \max_{j\in J} F_j(\bm X)$. Hence, 
the domain set $\X$ readily covers  multiple  spectral constraints. Such a  set $\X$ naturally appears  in many machine learning and optimization problems with a  low-rank constraint (see Subsection \ref{sub:app}). %presents several interesting examples  of our proposed \ref{eq_rank}.

The low-rank constraint dramatically complicates \ref{eq_rank}, which  often turns out to be an intractable nonconvex bilinear program. Thus, we leverage the  convex hull of domain set $\X$, denoted by $\conv(\X)$, to obtain a partial convexification over $\X$ for \ref{eq_rank}, termed LSOP-R throughout:
\begin{align} \label{eq_rel_rank}
 \V_{\rel}:=\min_{{\bm X \in \conv(\X)}}\left\{\langle\bm A_0, \bm X \rangle: 
b_i^l \le  \langle \bm A_i, \bm X\rangle \le b_i^u, \forall i \in [m] \right\} \le \V_{\opt}, \tag{LSOP-R}
\end{align}
where the inequality is because that \ref{eq_rel_rank} serves as a convex relaxation of  \ref{eq_rank}.

Since the  \ref{eq_rank}  requires the solution to be at most rank-$k$ 
% the rank constraint being satisfied
and the more tractable relaxation \ref{eq_rel_rank} may not enforce the rank constraint, a guaranteed low-rank solution of the  \ref{eq_rel_rank} can consistently favor various application requirements. 
That is, 
we are interested in finding the smallest integer $\hat{k}\geq k$ such that there is an optimal solution $\bm{X}^*$ of \ref{eq_rel_rank} satisfying $\rank(\bm{X}^*)\leq \hat{k}$. More importantly, 
the rank bounds are useful to develop 
approximation algorithms for  \ref{eq_rank} and solution algorithms for \ref{eq_rel_rank}, as well as  to understand 
%in understanding 
when the relaxation \ref{eq_rel_rank} meets \ref{eq_rank} (see, e.g.,  \citealt{burer2003nonlinear,burer2005local, burer2020exact,lau2011iterative,tantipongpipat2019multi}).  The analysis of rank bounds further inspires us a rank-reduction algorithm for \ref{eq_rel_rank}. 
Hence, this paper aims to provide: (i) theoretical rank  bounds  of  \ref{eq_rel_rank} solutions; and (ii) 
effective algorithms for solving \ref{eq_rel_rank} while satisfying rank bounds.

\subsection{Scope and Flexibility of Our LSOP Framework} \label{sub:app}
In this subsection, we discuss several interesting low-rank constrained problems in different matrix space $\Q$ where the proposed \ref{eq_rank} framework can be applied (as a substructure).
% of the proposed \eqref{eq_rank} framework. 
For those application examples, we specify their corresponding \ref{eq_rel_rank}s along with the rank bounds later.
% as well. 
% \vspace{1em}

\noindent\textit{Quadratically Constrained Quadratic Program (QCQP) with  $\Q:=\S_{+}^n$. }   The QCQP has been widely studied in many application areas, including optimal power flow, sensor network problems, signal processing \citep{josz2016ac,gharanjik2016iterative, khobahi2019optimized}, among others.
%The QCQP problem can be formulated in the following form:
%\begin{align*}
%\text{(QCQP)} \quad\min_{\bm x \in \Re^n} \left\{ \bm x^{\top} \bm Q_0 \bm x + \bm q_0^{\top} \bm x:  b_i^l \le \bm x^{\top} \bm Q_i \bm x + \bm q_i^{\top} \bm x  \le b_i^u, \forall i\in [m]\right\},
%\end{align*}
%where matrices $\bm Q_0, \bm Q_1, \cdots, 
%\bm  Q_{m} $ are symmetric but  may not be positive semidefinite.
%
%Introducing the matrix variable 
%$\bm X \in \S_{+}^{n+1} :=  \begin{pmatrix}
%1 & \bm x^{\top}\\
%\bm x & \bm x \bm x^{\top}
%\end{pmatrix}$, 
The QCQP of matrix form can be viewed as a special case of the proposed  \ref{eq_rank}:
\begin{equation}\label{app:qcqp}
  \min_{\bm X \in \X} \left\{\langle\bm A_0, \bm X\rangle: 
b_i^l \le\langle \bm A_i, \bm X\rangle \le  b_i^u, \forall i \in [m]\right\},  \ \  \X:= \{\bm X \in \S_+^{n}:  \rank(\bm X)\le 1\},
\tag{QCQP}
\end{equation}
where  $F(\bm X)=0$.
%where   ${\bm A}_{i} = \begin{pmatrix}
%0& { \bm q_i^{\top}}/{2}\\
%{\bm q_i}/{2}  & \bm Q_i\\
%\end{pmatrix}$ for each $i\in \{0\}\cup [m]$ and  $F(\bm X)=0$.
In fact, the domain set $\X$ can be extended to incorporate any closed convex spectral function $F(\bm X)$. For example, if there is a ball constraint (i.e., $\bm x^{\top} \bm x \le 1$) in the \ref{app:qcqp} like trust region subproblem, we can add a spectral constraint $\tr(\bm X) =  1+ \bm x^{\top} \bm x \le 2$ into set $\X$.
\vspace{1em}
%We notice that as $\conv(\X) = \S_+^{n+1}$, the corresponding LSOP-R of the QCQP \eqref{eq_qcqp} reduces to the well-known Semidefinite Programming (SDP) relaxation in literature, i.e,
%\begin{align*}
%\text{(LSOP-R of QCQP \eqref{eq_qcqp})} \quad \min_{\bm X \in \conv(\X)} \left\{\langle\bm A_0, \bm X\rangle: 
%b_i^l \le\langle \bm A_i, \bm X\rangle \le  b_i^u, \forall i \in [m], X_{11}=1 \right\},  \conv(\X):=\S_+^{n+1}.
%\end{align*}
%Note that we can strengthen the SDP relaxation by incorporating more 
%%induced side 
%constraints into the domain set $\X$, which will be illustrated in Subsection \ref{sec:dual}.
%

\noindent\textit{Low-Rank Kernel Learning with $\Q:=\S_{+}^n$.} Given an input kernel matrix $\bm Y \in \S_+^n$ of rank up to $k$,  the low-rank kernel learning aims to find a rank $\le k$ matrix $\bm X$ that closely approximates $\bm Y$, subject to additional linear  constraints. %This plays a crucial role in many machine learning tasks. 
The Log-Determinant divergence is a popular measure of the kernel learning quality (see, e.g., \citealt{kulis2009low}), defined as $\langle \bm Y^{\dag}, \bm X\rangle - \log\det[(\bm X+\alpha \bm I_n) (\bm Y+\alpha \bm I_n)^{-1}]$, where $\alpha>0$ is small and $\bm I_n$ is the identify matrix for addressing the rank deficiency issue. 
%Clearly, $\rank(\bm X \bm Y^{\dag}) \le k$. 
Thus,  the kernel learning under Log-Determinant divergence, as formulated  by \cite{kulis2009low}, becomes
\begin{equation}
\label{app:kernel}
\begin{aligned}
 &\min_{z \in \Re, \bm X \in \X} \left\{\langle \bm Y^{\dag}, \bm X\rangle - z + \log\det(\bm Y+\alpha \bm I_n): 
b_i^l \le \langle \bm A_i, \bm X\rangle \le  b_i^u, \forall i \in [m]\right\},  \\   &\quad\X:= \{\bm X \in \S_+^{n}:   \rank(\bm X)\le k, \log\det(\bm X+\alpha \bm I_n)\ge z\}.
\end{aligned}
\tag{Kernel Learning}
\end{equation}
Here,  a closed convex spectral function is defined as $F(\bm X):= z- \log\det(\bm X+\alpha \bm I_n)$ for any $z$. The proposed \ref{eq_rank} naturally performs a substructure over variable $\bm X$ in \ref{app:kernel}.
\vspace{1em}

\noindent\textit{Fair PCA with $\Q:=\S_{+}^n$.} 
It is recognized that the conventional Principal Component Analysis (PCA) may generate biased learning results against sensitive  attributes, such as gender, race, or education level \citep{samadi2018price}. Fairness  has recently been introduced to the PCA.
% these problems. 
Formally, the seminal work by  \cite{tantipongpipat2019multi} formulates the fair PCA  as follows in which the substructure over $\bm X$ is a special case of the proposed \ref{eq_rank}
\begin{align}\label{app:fpca}
  \max_{z\in \Re_+, \bm X \in \X} \left\{z: z\le \langle  \bm A_i, \bm X \rangle, \forall i \in [m]\right\}, \ \ {\X:= \{\bm X \in \S_+^{n}: \rank(\bm X)\le k, ||\bm X||_2 \le 1 \}}, \tag{Fair PCA}
\end{align}
where $F(\bm X):=||\bm X||_2-1$ denotes the spectral norm (i.e., the largest singular value) of a matrix and  matrices $\bm A_1, \cdots \bm A_m \in \S_+^n$ denote the sample covariance matrices from $m$ different groups. \vspace{1em}

\noindent\textit{Fair SVD with $\Q:=\Re^{n\times p}$.}
%Singular value decomposition (SVD), another classic unsupervised learning approach, allows us to extract and untangle information \citep{buet2022towards}.
Similar to \ref{app:fpca}, fair Singular Value Decomposition (SVD) seeks a fair representation of $m$ different  data matrices that
are non-symmetric, as formulated  below
\begin{align}\label{app:fsvd}
 \max_{z\in \Re, \bm X \in \X} \left\{z: z\le \langle  \bm A_i, \bm X \rangle, \forall i \in [m]\right\}, \ \ {\X:= \{\bm X \in \Re^{n\times p}: \rank(\bm X)\le k, ||\bm X||_2 \le 1 \}}. \tag{Fair SVD}
\end{align}
Here, matrices $\bm A_1, \cdots \bm A_m \in \Re^{n\times p}$ are  non-symmetric. %Specifically, our LSOP framework \eqref{eq_rank} appears as a subproblem over matrix $\bm X$ in fair SVD.
\vspace{1em}

\noindent\textit{Matrix Completion with $\Q:=\Re^{n\times p}$. }  In the matrix completion problem, the objective is to recover a
low-rank matrix $\bm X \in \Re^{n\times p}$ from  a subset of observed entries  $\{A_{ij}\}_{(i,j ) \in \Omega\subseteq [n]\times [p]}$, which arises  from a variety of applications including recommendation systems, computer vision, and signal processing \citep{chakraborty2013active,miao2021color}.
The proposed \ref{eq_rank} appears as a substructure of variable $\bm X$  in the following matrix completion formulation 
% As a special case of  \ref{eq_rank}, the matrix completion problem admits the following form:
\begin{equation} \label{app:lmc} 
\begin{aligned}
  \min_{\bm X\in \X, z \in \Re_+} \left\{z:  X_{ij}=A_{ij} , \forall (i,j)\in \Omega \right\}, \ \  \X:=\{\bm X\in \Re^{n\times p}: \rank(\bm X)\le k, \|\bm X\|_*\le z\},
  \end{aligned}
 \tag{Matrix Completion}
\end{equation}
where $F(\bm X):=\|\bm X\|_*-z$ builds on the nuclear norm.
The \ref{eq_rel_rank} counterpart of matrix completion reduces to a widely studied convex relaxation in the literature (see \Cref{sec:nonsym}).

%where function $F(\bm X)$ can stem from practical restriction on matrix $\bm X$. For example,
% \cite{cai2016matrix}  introduced the elementwise $\ell_{\infty}$-norm constraint in LMC, i.e., $|X_{ij }|\le \alpha$, for all $(i,j) \in [n]\times [p]$ given a constant $\alpha>0$. Then we can construct the convex spectral Frobenius norm function $F(\bm X):=\|\bm X\|_F -\sqrt{np}\alpha$.
\vspace{1em}

%\noindent\textbf{Symmetric Low-rank Representation (SLR).} In recent years, researchers have proactively studied the low-rank representation for the subspace clustering problem, which aims to segment data into clusters along with low-dimensional subspaces. That is, given a data matrix $\bm Z\in \Re^{d\times n}$ drawn from independent subspaces and dictionary $\bm A \in \Re^{d\times n}$, it is desired to find a low-rank representation matrix $\bm X\in \Re^{n\times n}$ such that $\bm Z  = \bm A \bm X$. In particular, each pair $(X_{ij}, X_{ji})$ denotes the interaction between a pair data points $(\bm z_i, \bm z_j)$.
% To ensure the weight consistency for each pair of data points, \cite{chen2017subspace} proposed the following SLR formulation that forced matrix $\bm X$ to be symmetric
%\begin{align*}
%\text{(SLR)} \quad \min_{\bm X\in \X, \nu} \nu +\lambda \|\bm Z - \bm A\bm X  \|_{2,1}, \ \ \X:=\{\bm X \in \S^n: \rank(\bm X)\le k,  \|\bm X\|_* \le \nu\},
%\end{align*}
%where  $\|\bm Z - \bm A\bm X  \|_{2,1}$ equals the sum of 2-norm of all row vectors  in matrix $\bm Z - \bm A\bm X \in \Re^{d\times n}$ to be zero. By introducing the auxiliary variable $\bm Y= \bm Z - \bm A\bm X$,
%the subproblem over variable $\bm X$ in the SLR can be recast as our \ref{eq_rank} with $m=dn$ linear inequalities as shown below.
%\begin{align}\label{app:slr}
%\text{(SLR)} \quad \min_{\bm X\in \X, \nu} \nu +\lambda \|\bm Z - \bm A\bm X  \|_{2,1}, \ \ \X:=\{\bm X \in \S^n: \rank(\bm X)\le k,  \|\bm X\|_* \le \nu\},
%\end{align}
%\vspace{1em}

\noindent \textit{Sparse Ridge Regression with $\Q:=\S^n$.}
%The  sparse optimization that requires the solution to be sparse has a wide range of applications, e.g., 
As machine learning  may encounter with datasets with many 
features, 
%where 
the sparsity has been enforced to select a handful of important ones to improve the interpretability of the learning outcomes. The zero norm that denotes the sparsity of a vector is equal to the rank of the corresponding diagonal matrix constructed by this vector. Hence, when $\Q:=\S^n$ in \eqref{eq_set}, the special diagonal matrix space   allows us to formulate \ref{eq_rank}  with the known sparse  ridge regression problem (see, e.g., \citealt{xie2020scalable}) as
\begin{align*}
 \min_{\bm x\in \Re^n}
\left\{\|\bm b-\bm A\bm x\|_2^2/n + \alpha\|\bm x\|_2^2: \|\bm x\|_0 \le k \right\},
\end{align*}
where $\bm A\in \Re^{m\times n}$ represents the data matrix, $\alpha >0$ is a regularizer, and $\|\bm x\|_0$ denotes the zero norm.
By introducing the auxiliary variables $\bm y = \bm b-\bm A\bm x$, $z\ge \|\bm x\|_2^2$ and letting $\bm X =\Diag(\bm x)$,  the sparse ridge regression reduces to 
\begin{equation} \label{app:sparse}
\begin{aligned}
&\min_{\bm X\in \X, \bm y \in \Re^m, z \in \Re_+}
\left\{  \|\bm y\|_2^2/n + \alpha z: \bm y= \bm b-\bm A\diag(\bm X)\right\}, \\ 
&\X:=\left\{\bm X\in \S^n:  \bm X =\Diag(\diag(\bm X)), \rank(\bm X) \le k, \|\bm X\|_F^2 \le z\right\}, 
\end{aligned}
\tag{Sparse Ridge Regression}
\end{equation}
where the constraint $\bm X =\Diag(\diag(\bm X))$ enforces matrix $\bm X$ to be diagonal and thus  $ \|\bm X\|_F^2=\|\diag(\bm X)\|_2^2=\|\bm x\|_2^2 \le z$ must hold.
The subproblem over  $\bm X$ above  follows our \ref{eq_rank} framework.

\subsection{Relevant Literature}
In this subsection, we survey the relevant literature on two aspects.

\textit{Convexification of Low-Rank Spectral Domain Set $\X$.} There are few works on the convexification of a low-rank spectral domain set $\X$ in \eqref{eq_set}.  
The work by \cite{bertsimas2021new} has successfully extended the perspective technique to convexify a special low-rank set $\X$ with $\Q:=\S^n$ in which all the eigenvalues in the function $F(\bm X):=f(\bm \lambda(\bm X))$ are separable. 
% under which the domain set $\X$ consists of a separable spectral constraint. 
Such approach, however, may fail to cover the general set $\X$, e.g., the spectral norm function $F(\bm X):=\|\bm X \|_2$ 
% the largest singular value constraint $\|\bm X \|_2 \le 1$  
in the \ref{app:fpca} which is not separable. Another seminal work \citep{kim2022convexification} leverages the majorization technique on the convexification of any 
permutation-invariant set. We observe that our domain set $\X$ in \eqref{eq_set} is permutation-invariant with eigenvalues or singular values; thus, the   majorization technique can be applied. It should be noted that the majorization technique may not always provide an explicit description for our convex hull $\conv(\X)$, except making further assumptions on the spectral constraint, as shown in \Cref{sec3}. 
% \cite{kim2022convexification}. 
To resolve this limitation, this paper  studies the inner approximation of the convex hull  $\conv(\X)$ to obtain a tractable relaxation of \ref{eq_rank} with the matrix space  $\Q:=\S^n$.
% Besides, 
This paper focuses on bounding 
the ranks of  \ref{eq_rel_rank} solutions, and  the convexification result of \cite{kim2022convexification} paves the path for the derivation of our bounds.

\textit{Rank Bounds for \ref{eq_rel_rank}.}
Given a domain set $\X=\{\bm X\in \S_{+}^n:\rank(\bm X)\le 1 \}$ particularly adopted in \ref{app:qcqp}, its convex hull is the positive semidefinite matrix space  and the corresponding \ref{eq_rel_rank} feasible set is a spectrahedron.
\cite{barvinok1995problems,deza1997geometry,pataki1998rank} independently showed that the rank of any extreme point in the spectrahedron is in the order of $\O(\sqrt{2\tilde m})$. Recently, in the celebrated paper on \ref{app:fpca} \citep{tantipongpipat2019multi}, the authors also proved a rank bound for all feasible extreme points of the corresponding \ref{eq_rel_rank}, which is  $\O(\sqrt{2 m})$, and the proof technique extended  that of \cite{pataki1998rank}. Another relevant work pays particular attention to the sparsity bound for the sparse  optimization problem \citep{askari2022approximation}, which relies on a different continuous relaxation from the \ref{eq_rel_rank}. 
Since the \ref{eq_rank} can encompass the sparse optimization when $\Q$ in \eqref{eq_set} denotes the diagonal  matrix space, our rank bounds are in fact sparsity bounds under this setting.
To the best of our knowledge, this is the first work to study the rank bounds for the generic partial convexification-- \ref{eq_rel_rank}. 
%It is worth mentioning that
 Our rank bounds  recover all the ones reviewed here for \ref{app:qcqp} and  \ref{app:fpca} from a different perspective, and  successfully reduce the sparsity bound of \cite{askari2022approximation} when applying to \ref{app:sparse}.

 % As mentioned,  our proposed bound holds for any domain set $\X)$ that is general enough to cover the existing ones.
%Our rank bounds shed light on the solution quality of \ref{eq_rel_rank} and relaxation gap between \ref{eq_rel_rank} and \ref{eq_rank}.
%In their celebrated paper,
%It is recognized \cite{burer2005local, lau2011iterative, tantipongpipat2019multi}  that the rank bounds  have played an important role in designing efficient approximation algorithms. 
% and both the computational efficiency and performance guarantee rely on the rank bound.

%We remark that, beyond the relaxation Kron introduced in this paper, one cancertainly apply higher-order lifting techniques
\subsection{Contributions and Outline}
We theoretically guarantee the solution quality of the relaxation \ref{eq_rel_rank} by bounding the ranks of all feasible extreme points with three different matrix spaces (i.e., $\Q=\S_+^n$, $\Q=\Re^{n \times p}$, and $\Q=\S^n$), respectively, and each matrix space has its own advantages and applications, as exemplified in Subsection \ref{sub:app}. Notably,  our rank bounds hold for any domain set $\X$ in the form of \eqref{eq_set}
% vary only  with the matrix space $\Q$,and hold for any spectral function $F(\bm X)$ in \eqref{eq_set}
and for any $\tilde m$ linearly independent inequalities, and they are attainable by the worst-case instances. 
% We investigate the \ref{eq_rel_rank} under three different matrix spaces (i.e., $\Q=\S_+^n$, $\Q=\S^n$, and $\Q=\Re^{n \times p}$), respectively, and each matrix space has its own advantages and applications, as exemplified in Subsection \ref{sub:app}. 
% We also illustrate why the matrix space is critical in determining the rank bounds.
Below summarizes the major contributions of this paper.
\begin{enumerate}[(i)]
\item  \Cref{sec2} studies  \ref{eq_rel_rank} with the positive semidefinite matrix space, i.e., $\Q:=\S_{+}^n$.
%
%We, by leveraging the majorization technique \citep{kim2022convexification}, explicitly describe  the convex hull $\conv(\X)$ and the \ref{eq_rel_rank}. 
We show that the rank of any extreme point in the feasible set of  \ref{eq_rel_rank} deviates  at most $\lfloor \sqrt{2\tilde m+9/4}-3/2\rfloor$ from the original rank-$k$ requirement in \ref{eq_rank}.
We establish this rank bound from a novel perspective, specifically that of analyzing the rank of various  faces in the convex hull of the domain set (i.e., $\conv(\X)$). Besides, the rank bound gives a sufficient condition under which the \ref{eq_rel_rank} exactly solves the original  \ref{eq_rank}. 

We conclude \Cref{sec2} by revisiting the three examples with matrix space $\Q:=\S_+^n$ in Subsection \ref{sub:app} and deriving  the rank bounds for their corresponding LSOP-Rs.

\item Section \ref{sec:nonsym} explores the non-symmetric matrix space, i.e., $\Q=\Re^{n\times p}$. 
% In general, the convex hull $\conv(\X)$ may not be easy to derive when $\Q:=\Re^{n\times p}$, and we instead use an alternative convex relaxation of set $\conv(\X)$ to reformulate the \ref{eq_rel_rank}, further termed LSOP-R-I. 
Advancing the analysis in  \Cref{sec2},  we show that the \ref{eq_rel_rank} admits the same rank bound
% smaller than 
as the one with $\Q:=\S_+^n$.
% due to the flexibility of the non-symmetric matrix space. 
% $\Q=\Re^{n\times p}$.

In Subsection \ref{sub:nonsymapp}, we present the non-symmetric applications:  \ref{app:fsvd} and \ref{app:lmc}, and derive the first-known rank bounds for their LSOP-Rs.

\item \Cref{sec3} discusses the symmetric indefinite matrix space, i.e., $\Q:=\S^n$ under two cases depending on whether the  function $f(\cdot)$ in \eqref{eq_set} is sign-invariant with eigenvalues or not.

In Subsection \ref{subsec:sign}, the sign-invariant property allows for an identical rank bound  of \ref{eq_rel_rank} to the positive semidefinite matrix space.

Beyond sign-invariance in Subsection \ref{subsec:nosign}, the convexification of domain set $\X$ involves the convex hull of a union of multiple convex sets, which inspires us a tighter relaxation than \ref{eq_rel_rank} by replacing $\conv(\X)$ with the union set, further termed LSOP-R-I. Then, we derive a rank bound of the order of $\O(\sqrt{4\tilde m})$  for the \ref{eq_rel_rank1}.

Finally, we extend the analysis to derive sparsity bounds in the diagonal matrix space and  apply the result to \ref{app:sparse}.

\item In \Cref{sec:algo}, we develop an efficient column generation algorithm
with a vector-based convex pricing oracle to solve \ref{eq_rel_rank}, where the rank of the output solution can be reduced to be no larger than the theoretical bound using a rank-reduction algorithm.
% In \Cref{sec:algo}, we develop an efficient column generation algorithm to solve \ref{eq_rel_rank} 
% with rank reduction procedure 
 In \Cref{sec:num}, we numerically test the proposed algorithms. %, where the output rank for the latter can be further reduced to meet the theoretical rank bound via an iterative rank reduction algorithm.

%\item By relaxing the sign-definite assumption when $\Q:=\S^n$, \Cref{sec3} provides a tighter LSOP-R set. %that $\P(k)$ for the original nonconvex set $\C(k)$. 
%When this tighter set is bounded and rank-$r^*$ generated, we also derive  a rank bound of $\k+\O(\sqrt{\tilde m})$;
%\item When set $\P(k)$ is unbounded,  all our analyses follow  and the proposed bounds  become rank bounds of all the extreme points in $\P(k)$;
%\item  demonstrates that the proposed rank bounds can be further improved for some desirable extreme points in set $\P(k)$; and
%\item Our rank bounds and exactness conditions under various matrix spaces can be directly applied to all LSOP  examples  in Subsection \ref{sub:app}, as presented in \Cref{table:rank}. For the QCQP and fair PCA, we  recover the existing results in literature.
%For fair SVD, matrix completion, and sparse ridge regression problems, this paper contributes the first-known rank bound to their corresponding LSOP-Rs. 
\end{enumerate}
 The detailed rank bounds for \ref{eq_rel_rank} with different matrix spaces and those for application examples can be found in \Cref{table:rank}. Note that we  further tighten the rank bounds using $\tilde k$ in \Cref{def:k}. We also show that the  rank bounds  marked with asterisk in \Cref{table:rank} are tight.

\begin{table}[htb]
\caption{Summary of Our Rank Bounds } 
\centering
\label{table:rank}
% \begin{threeparttable}
\setlength{\tabcolsep}{2.5pt}\renewcommand{\arraystretch}{1.2}
\begin{tabular}{c| r| r |r}
%\hline 
%\multicolumn{3}{|c|}{Approximation Bound of MESP} \\
\hline
\multicolumn{1}{c|}{{\textbf{LSOP Relaxation }}}&\multicolumn{1}{c|}{\textbf{Matrix Space} $\Q$}
&\multicolumn{1}{c|}{\textbf{Spectral Constraint }} & \multicolumn{1}{c}{{\textbf{Rank Bound}}}  \\
\hline
\multirow{2}{*}{\ref{eq_rel_rank}} & \multirow{2}{*}{$\S_+^n$} &\multirow{2}{*}{Any} &$\k \tnote{i} +\lfloor \sqrt{2\tilde m+9/4}-3/2\rfloor$ \\
& &&{(\Cref{them:rank})*} \\
\hline
\multirow{2}{*}{\ref{eq_rel_rank}}  &\multirow{2}{*}{$\Re^{n\times p}$}  &      \multirow{2}{*}{Any} &$\k +\lfloor \sqrt{2\tilde m+9/4}-3/2\rfloor$  \\
& &&{(\Cref{them:nonsym})*}  \\
\hline
\multirow{2}{*}{\ref{eq_rel_rank}}  &\multirow{2}{*}{$S^n$} & \multirow{2}{*}{Sign-invariant} &$\k+\lfloor \sqrt{2\tilde m+9/4}-3/2\rfloor$  \\
& &&{(\Cref{them:symsign})*}  \\
\hline
\multirow{2}{*}{\ref{eq_rel_rank1}} &\multirow{2}{*}{$S^n$}  & \multirow{2}{*}{Any} &$\k+\lfloor \sqrt{4\tilde m+9}\rfloor-3$ \\
& &&{(\Cref{them:sym_rank})}  \\
\hline
\multirow{2}{*}{\ref{eq_rel_rank1}} &\multirow{2}{*}{Diagonal  $S^n$}  & \multirow{2}{*}{Any} &$\k+\tilde m$ \\
& &&{(\Cref{them:sparsenonsign})*}  \\ 
\hline
\hline 
\multicolumn{1}{c|}{{\textbf{LSOP-R Examples }}}&\multicolumn{1}{c|}{\textbf{Matrix Space} $\Q$}
&\multicolumn{1}{c|}{\textbf{Spectral Constraint }} & \multicolumn{1}{c}{{\textbf{Rank Bound}}}  \\
\hline
\multirow{2}{*}{\ref{pc_kernel}} &\multirow{2}{*}{$\S_+^n$}& \multirow{2}{*}{$z- \log\det(\bm X+\alpha \bm I_n)\le 0$} & $k+\lfloor \sqrt{2\tilde m+9/4}-3/2\rfloor$\\
& &&(\Cref{cor:kernel})  \\
\hline
\multirow{2}{*}{\ref{pc_fpca}} &\multirow{2}{*}{$\S_+^n$} & \multirow{2}{*}{$\|\bm X\|_2-1 \le 0$} & $k+\lfloor \sqrt{2m +1/4}-3/2\rfloor$ \\
& &&(\Cref{cor:spca}) \\
\hline
\multirow{2}{*}{\ref{pc:qcqp}} & \multirow{2}{*}{$\S_+^n$} & \multirow{2}{*}{N/A} & 1+$\lfloor \sqrt{2\tilde m+1/4}-3/2\rfloor$ \\
& &&(\Cref{cor:qcqp})  \\
\hline
\multirow{2}{*}{\ref{pc_fsvd}} &\multirow{2}{*}{$\Re^{n\times p}$} &  \multirow{2}{*}{$\|\bm X\|_2-1 \le 0$} & $k+\lfloor \sqrt{2m+1/4} -3/2 \rfloor$\\
& &&(\Cref{cor:fsvd})  \\
\hline
\multirow{2}{*}{\ref{pc_lmc}} &\multirow{2}{*}{$\Re^{n\times p}$} &  \multirow{2}{*}{$\|\bm X\|_* -z \le 0$} & $\lfloor \sqrt{2|\Omega| +9/4}-1/2\rfloor$\\
& &&(\Cref{cor:lmc})  \\
\hline
\multirow{2}{*}{\ref{pc_sparse}} &\multirow{2}{*}{Diagonal $\S^n$} &  \multirow{2}{*}{$\|\bm X\|_F^2 -z \le 0$} & $k+\rank(\bm A)$\\
& &&(\Cref{cor:sparse})  \\
%\hline
%OLSO \eqref{app:oso} &$\S^n$ &  $\|X\|_2 \le 1$ & $ \underline k +1$\\
\hline
\end{tabular}% 
% \begin{tablenotes}
% \item[i] The integer $\k\le k$ follows \Cref{def:k};
% % \item[ii] This rank bound   has been shown in \cite{pataki1998rank}; 
% % \item[iii] This rank bound has been shown in  \cite{tantipongpipat2019multi}; 
% \item[ii] The rank of data matrix $\bm A \in \Re^{m\times n}$ in   \ref{app:sparse}.
% \end{tablenotes} 
% \end{threeparttable}
\end{table}

\noindent\textit{Notation and Definition.} We use bold lower-case letters
(e.g., $\bm x$) and bold upper-case letters (e.g., $\bm X$) to denote vectors and matrices, respectively, and use
corresponding non-bold letters (e.g., $x_i$) to denote their components. 
 We let $[n]:=\{1,2,\cdots, n\}$.  We let $\bm I_n$ denote the $n\times n$
identity matrix. For any $\lambda \in \Re$, we let $\lfloor \lambda \rfloor$ denote the greatest integer less than or equal to $\lambda$,  let $(\lambda)_+:=\max\{0, \lambda\}$, and let $\sign(\lambda)$ be 1 if $\lambda \ge 0$,  otherwise, $-1$. For a set $S$, we let $|S|$ denote its cardinality.
For a vector $\bm x \in \Re^n$, we let $|\bm x| :=(|x_1|,\ldots, |x_n|)^\top$ contain the absolute entries of $\bm x$, let $\|\bm x\|_0$ denote its zero norm, let $\|\bm x\|_1$ denote its one norm, let $\|\bm x\|_2$ denote its two norm, and let $\Diag(\bm x)$ denote a diagonal matrix with diagonal entries from $\bm x$. %denotes the one taking its entry-wise absolute values. 
For a matrix $\bm X \in \Q$, we let $\|\bm X\|_2$ be its largest singular value, let $\|\bm X\|_*$ be its nuclear norm,  let $\|\bm X\|_F$ be its Frobenius norm, and if matrix $\bm X$ is square, let $\bm X^{\dag}$ be its Moore–Penrose inverse, let $\tr(\bm X)$ be the trace, let $\diag(\bm X)$ be a vector including its diagonal elements, and let $\lambda_{\min}(\bm X), \lambda_{\max}(\bm X)$ denote the smallest and largest eigenvalue of matrix $\bm X$, respectively. 
For a matrix $\bm X\in \Q$ and an integer $i\in [n]$, we let $\|\bm X\|_{(i)}$ denote the sum of  first $i$ largest singular values of matrix $\bm X$. Note that $\|\bm X\|_{(n)} = \tr(\bm X)$  if matrix $\bm X$ is positive semidefinite. 
  For a set $T$ over $(\bm X, \bm x)$, we let $\Proj_{\bm X}(T):=\{\bm X: \exists \bm x, (\bm x, \bm X)\in T\}$ denote the projection of set $T$ into element $\bm X$. 
A function $f(\bm \lambda): \Re^n\to \Re$ is \textit{symmetric} if it is invariant with any permutation of $\bm \lambda$ and is \textit{sign-invariant} if $f(\bm \lambda)=f(|\bm \lambda|)$. Additional notation will be introduced later as needed. %We also caution the readers that this paper {may} reuse notation throughout, {e.g., domain set $\D$}, which will be appropriately specified in {different context}. %For example, for the different matrix spaces, the \ref{eq_rel_rank} has two variants: \ref{eq_rel_rank} and \ref{eq_rel_rank1}. %Besides, when referring to an example application, the domain set $\X$, LSOP, and LSOP-R will have a different,  specific formulation, respectively. Albeit with various contextual meanings, the notations play the same general role always.
% {$\Re_{++}$, subvector  and they based on contexts will be specified when necessary.}
% , e.g., the LSOP-R serves as a partial relaxation over domain set $\X$ for \ref{eq_rank} or for the LSOP example application.

\section{The Rank Bound in Positive Semidefinite Matrix Space }\label{sec2}
%\vspace{-1em}
To guarantee the solution quality  of the partial convexification \ref{eq_rel_rank}, this section derives a rank bound for all extreme points and for an optimal solution to \ref{eq_rel_rank}, provided that $\Q:=\S_{+}^n$ denotes the positive semidefinite matrix space. Notably, our results recover  the existing rank bounds for two special LSOPs:  \ref{app:qcqp} and \ref{app:fpca}, and are also  applicable to \ref{app:kernel}. 

%In the next two sections, we will show how the rank bound changes with  other matrix space settings (i.e., $\Q:=\S^n$ and $\Q:=\Re^{n\times p}$).

%\subsection{Feasible Set Characterization of \ref{eq_rel_rank}: Describing Convex Hull of Domain Set $\X$}
\subsection{Convexifying Domain Set}
This subsection provides an explicit characterization of  the convex hull of the domain set $\X$, i.e., $\conv(\X)$, which is a key component in the feasible set of \ref{eq_rel_rank},
%, which intersecting with $m$ linear matrix inequalities composes 
% to describe  the feasible set of \ref{eq_rel_rank}.

%\begin{observation}
%The spectral function $F(\bm X):=f(\bm \lambda(\bm X))$ in domain set $\X$ is always  permutation-invariant with the eigenvalues. 
%\end{observation}
%
Before deriving $\conv(\X)$, let us define an integer $\k\leq k$ as the strengthened  rank. When necessary in the theoretical analysis of rank bounds, we replace the rank requirement $k$ with the smaller
% better 
$\k$.
%In practice, we can use the integer $k$ if it is difficult to obtain $\k$
\begin{definition}[Strengthened Rank $\k$] \label{def:k}
For a domain set $\X$ in \eqref{eq_set}, we let $1\le \k \le k$ be the smallest integer such that
\begin{align*}
\conv(\X):=\conv\left( \left\{\bm X\in \S_+^n: \rank(\bm X)\le \k, F(\bm X):=f(\bm \lambda(\bm X))\le 0 \right\}\right).
\end{align*}
\end{definition}
% Throughout, we will use $\k$ to replace $k$  to further tighten our analysis.  
 The following example confirms that $\k$ can be indeed strictly less than $k$.
\begin{example}\label{eg1}
\rm
Suppose a domain set $\X:= \{\bm X\in \Q: \rank(\bm X)\le k, f(\|\bm X\|_*) \le 0\}$ with $k\ge 2$ and $f(\cdot)$ is a closed convex bounded function. Then sets $\X$ and $\conv(\X)$ are bounded. In this example, we show below that $\k=1<   k$.

For any matrix $\hat{\bm X}$ in domain set $\X$ with rank $r \ge 2$, let its singular value decomposition  be  $\hat{\bm X}={\bm Q} \Diag(\hat{\bm \lambda}) {\bm P}^{\top}$. %In particular, the eigenvalue vector $\hat{\bm \lambda}$ has  $r\ge 2$ nonzero elements and $\hat{\lambda}_1 \ge \cdots \ge \hat{\lambda}_n$. 
Next, let us construct $r$ vectors $\{\bm \lambda^i \}_{i\in [r]}\subseteq \Re_+^n$ as below, where for each $i\in [r]$,  
\begin{align*}
\lambda_{\ell}^i  = \begin{cases}
\|\hat{\bm \lambda}\|_1, & \text{if } \ell=i;\\
0, &  \text{if } \ell \in [n]\setminus \{i\}.
\end{cases}.
\end{align*}
Then,  for each $i\in [r]$,
we have that $\|{\bm \lambda}^i\|_0=1\le k$ and $ f(\|{\bm \lambda}^i\|_1) = f(\|\hat{\bm \lambda}\|_1) \le 0$, which means that the inclusion  ${\bm Q} \Diag({\bm \lambda}^i) {\bm P}^{\top} \in \X$ holds for all $i\in [r]$. Also, we have $\hat{\bm X} = \sum_{ i \in [r]} \frac{\hat{\lambda}_i}{\|
\hat{\bm \lambda}\|_1}  \bm Q \Diag({\bm \lambda}^i) \bm P^{\top}$ by the construction of vector $\{\bm \lambda^i \}_{i\in [r]}\subseteq \Re_+^n$,
%\begin{align*}
%\bm \lambda' = \sum_{ i \in [r]} \frac{|\lambda'_i| }{||\bm \lambda'||_1} \bm \lambda^i,
%\end{align*}
implying that $\hat{\bm X} $ cannot be an extreme point of the  set $\conv(\X)$. That is, any extreme point in the bounded set $\conv(\X)$ has a rank at most one. Hence, we must have $\k=1$ by \Cref{def:k}.
%That is, $\conv(\{\bm X\in \Q: \rank(\bm X) \le 1,  f(\|\bm X\|_2) \le 0\})$.
\qedA
\end{example}

We now turn to the explicit characterization of $\conv(\X)$, which builds on  the majorization below.

\begin{definition}[Majorization]
Given two vectors $\bm x_1, \bm x_2 \in \Re^n$,  we let $\bm x_1 \succ \bm x_2$ denote that $\bm x_1$ weakly majorizes  $\bm x_2$ (i.e., $\max_{\bm{z}\in \{0,1\}^n}\{\bm{z}^\top \bm{x}_1: \e^\top \bm{z}=\ell\}\geq \max_{\bm{z} \in \{0,1\}^n}\{\bm{z}^\top \bm{x}_2: \e^\top \bm{z}=\ell\}$ for all $\ell\in [n]$), and let $\bm x_1 \succeq \bm x_2$  denote that $\bm x_1$  majorizes  $\bm x_2$, i.e., $\bm x_1$ weakly majorizes  $\bm x_2$ and $\bm{e}^\top \bm{x}_1=\bm{e}^\top \bm{x}_2$, where $\bm e\in \Re^n$ denotes all-ones vector.
\end{definition}
The spectral function $F(\bm X):=f(\bm \lambda(\bm X))$ in \eqref{eq_set}  is invariant with any  permutation of eigenvalues (see, e.g., \citealt{drusvyatskiy2015variational}), implying a permutation-invariant domain set $\X$. This motivates us to apply the convexification result for a permutation-invariant set, as shown in the seminal work by \cite{kim2022convexification}.
\begin{restatable}{proposition}{propconv}\label{prop:conv}
For a domain set $\X$ in positive semidefinite matrix space, i.e., $\Q:=\S_{+}^n$ in \eqref{eq_set},   its convex hull $\conv(\X)$  is equal to
\begin{align*}
 \bigg\{\bm X \in \Q: \exists \bm x \in   \Re_+^n, f(\bm x)\le 0,  x_1 \ge \cdots \ge x_n,  x_{k+1}=0,  \|\bm X\|_{(\ell)} \le \sum_{i\in [\ell]} x_{i}, \forall \ell \in [k], \tr(\bm X) = \sum_{i\in [k]}x_i\bigg\},
\end{align*}
and  is a closed set.
\end{restatable}
\begin{proof}
    See Appendix \ref{proof:conv}. \qed
\end{proof}

%The result in  \Cref{prop:conv} gives  an explicit description of the feasible set of \ref{eq_rel_rank}.
%Besides,  
The function $\|\bm X\|_{(i)}$ in  \Cref{prop:conv} is known to admit a tractable semidefinite representation for each $i\in [n]$ (see, e.g., \citealt{ben2001lectures}) . 
%Compared to the seminal work by \cite{kim2021convexification} that exploited the majorization technique for convexifying a general permutation-invariant set,   our \Cref{prop:conv} focuses on the rank-constrained domain set $\X$ and gives explicit expression of its convex hull. 
%A sight generalization of \Cref{prop:conv} that  goes beyond the  permutation-variant set $\X$ is further discussed in \Cref{sec3}.
%Let us now turn to the main focus of this paper--deriving rank bounds for \ref{eq_rel_rank}.

\subsection{Rank Bound} \label{sub22}
In this subsection, we derive an upper bound of the ranks of all extreme points in the feasible set of  \ref{eq_rel_rank}, % and for an optimal solution to \ref{eq_rel_rank}. 
%More importantly, we would like to
which also sheds light on the rank gap between the relaxation \ref{eq_rel_rank} and the original \ref{eq_rank} at optimality.
Let us first introduce a key lemma to facilitate the analysis of rank bound.
%\Cref{prop:conv}  implies that
%describing $\P(k)$ highly depends on the choice of matrix space $\Q$. 
%A possible explanation is that $\bm \lambda(\bm X)$ has different sign and meaning for different definition of $\Q$;
%\item Our Part (i)  in \Cref{prop:conv} can be useful to generalize the result in \cite{bertsimas2021new} that further assumes the separability of functions $f_j(\cdot)$ to use the perspective technique.
%\end{enumerate}
%{the feasible set $\P(k)$; intuition about why different $\Q$ has different difficulty;  generalize \cite{bertsimas2021new}}

% for each $s\in [k+1]$, the inner minimization problem over set $\Y^s$ is a convex optimization. In addition, $\min_{s\in [k+1]} \V_{\rel}^s \ge \V_{\opt}$ holds; henceforth, problem \eqref{eq_rel_rank2} provides a stronger convex relaxation for the original problem \eqref{eq_rank}.

\begin{restatable}{lemma}{lemineq} \label{lem:ineq}
Given two vectors $\bm \lambda, \bm x \in \Re^n$  with their elements sorted in descending order and $\bm x \succeq \bm \lambda$, suppose that there exists an index $j_1 \in [ n-1]$ such that 
$\sum_{ i \in [j_1]} \lambda_i < \sum_{ i \in [j_1]}x_i$.  Then  we  have 
$$\sum_{ i \in [j]} \lambda_i < \sum_{ i \in [j]}x_i, \forall j\in [j_0, j_2-1],$$
where the indices $j_0, j_2$ satisfy $\lambda_{j_0} =\cdots =\lambda_{j_1} \ge \lambda_{j_1+1} =\cdots = \lambda_{j_2}$ with $1\le j_0\le j_1\le j_2-1 \le n-1$.
\end{restatable}
\begin{proof}
    See Appendix  \ref{proof:lemineq}. \qed
\end{proof}

Next, we are ready to prove one of our main contributions that relies on the inequalities over $[j_0, j_1]$ in \Cref{lem:ineq}. Specifically, we show that the rank bound of a face in  $\conv(\X)$ depends on the dimension of this face. %The inequalities over $[j_1+1, j_2-1]$ in \Cref{lem:ineq} that \Cref{sec3} needs are also now stated.  
%Before proceeding to the rank bound of \ref{eq_rel_rank}, let us investigate the rank bounds of faces in convex hull $\conv(\X)$ that establishes the LSOP-R feasible set with $m$ linear inequalities.
Below is the formal definition of faces and dimension of a closed convex set.

\begin{definition}[Face \& Dimension] \label{def:face}
A convex subset $F$ of a closed convex set $\T$ is called  a face of $\T$ if for any line segment $[a,b]\subseteq \T$ such that $F \cap (a,b)\neq \emptyset$, we have $[a,b]\subseteq F$. The dimension of a face  is equal to the dimension of its affine hull.
\end{definition}
\begin{restatable}{theorem}{thmproprank}\label{prop:rank}
Given   $\Q:=\S_{+}^n$ in domain set $\X$ in \eqref{eq_set}, suppose that $F^d\subseteq \conv(\X)$ is a $d$-dimensional face of the convex hull of the domain set. Then any point in face $F^d$ has a rank at most $\tilde k + \lfloor\sqrt{2d+9/4}-3/2\rfloor$, where  $\k\le k$ follows \Cref{def:k}.
\end{restatable}
\begin{proof}
    See Appendix \ref{proof:proprank}. \qed
\end{proof}

%Given a rank-$k$ constrained domain set $\X$ and its convex hull $\conv(\X)$,  
\Cref{prop:rank} suggests that the dimension of a face in set $\conv(\X)$ determines the rank bound.
%Interestingly, the rank bound is  almost linear in $\sqrt{2d}$ with $d$ denoting the dimension. 
When intersecting set $\conv(\X)$ with $m$ linear inequalities in \ref{eq_rel_rank}, the recent work by \cite{li2022exactness} established the one-to-one correspondence between the extreme points of the intersection set and no larger than $m$-dimensional faces of $\conv(\X)$ (see lemma 1 therein).
\begin{lemma}[\citealt{li2022exactness}]\label{lem:face}
For any closed convex set $\T$, if $\bm X$ is an  extreme point in the intersection of set $\T$ and $\tilde m$ linearly independent inequalities, then $\bm X$ must be contained in  a no larger than $\tilde m$-dimensional face of set $\T$.
\end{lemma}

Taken these results together, we next derive a rank bound
% the rank bounds 
for   \ref{eq_rel_rank}.

\begin{theorem}\label{them:rank}
Given $\Q:=\S_{+}^n$ in \eqref{eq_set} and integer  $\k\le k$ following \Cref{def:k}, we have
% the following statements hold:
\begin{enumerate}[(i)]
\item Each feasible extreme point in the \ref{eq_rel_rank} has a rank at most $\k + \lfloor\sqrt{2\tilde m+9/4}-3/2\rfloor$; and
\item There is an optimal solution to the \ref{eq_rel_rank} of rank at most $\k + \lfloor\sqrt{2\tilde m+9/4}-3/2\rfloor$ if
the \ref{eq_rel_rank}  yields a finite optimal value, i.e.,
  $\V_{\rel}>-\infty$. 
\end{enumerate}
\end{theorem}
\begin{proof}
The proof can be divided into two parts.

\noindent \textbf{Part (i).} Since the LSOP-R feasible set  intersects the closed convex  set $\conv(\X)$ with $\tilde m$ linearly independent inequalities, according to \Cref{lem:face}, each extreme point $\hat{\bm X}$ in the intersection  belongs to a no larger than  $\tilde m$-dimensional face  in set $\conv(\X)$. Using  \Cref{prop:rank}, the face containing $\hat{\bm X}$ must satisfy the  rank-$(\k + \lfloor\sqrt{2\tilde m+9/4}-3/2\rfloor)$ constraint and so does the extreme point $\hat{\bm X}$.

\noindent \textbf{Part (ii).} Since the \ref{eq_rel_rank} admits a finite optimal value, i.e., $\V_{\rel}>-\infty$ and a line-free feasible set, according to \cite{rockafellar2015convex}, the \ref{eq_rel_rank} can achieve  extreme point optimum. Thus, according to Part (i), there exists an optimal extreme point satisfying the desired rank bound. % {exists}.
\qed
\end{proof}

 %The rank bound which applies to all extreme points  of course includes the optimal one.
 We make the following remarks about  \Cref{them:rank}.
\begin{enumerate}[(i)]
\item \Cref{them:rank} shows that any extreme point in the feasible set of \ref{eq_rel_rank}, as a convex relaxation of \ref{eq_rank},  violates the rank-$k$ constraint by at most $(\k -k+  \lfloor\sqrt{2\tilde m+9/4}-3/2\rfloor)_+$.
\item Given $\Q:=\S_+^n$ in \eqref{eq_set}, the most striking aspect  is that the rank bound of \Cref{them:rank} is independent of  any   domain set $\X$, any linear objective function, and any $\tilde m$ linearly independent inequalities in \ref{eq_rel_rank}. We present three examples in Subsection \ref{sub:bound} to demonstrate the versatility of our rank bound;
\item From a new angle, the proposed rank bound for \ref{eq_rel_rank} arises from bounding the rank of various faces in  set $\conv(\X)$,  as shown in \Cref{prop:rank} through perturbing majorization constraints.
As a result, the derivation of our rank bounds in  \Cref{them:rank}  differs  from those for two specific LSOP-Rs of \ref{app:qcqp}  \citep{barvinok1995problems,deza1997geometry,pataki1998rank} and \ref{app:fpca}  \citep{tantipongpipat2019multi}. In particular, \cite{pataki1998rank} derived the rank bound  from analyzing faces of the whole feasible set of \ref{eq_rel_rank} of \ref{app:qcqp}, whereas our result, inspired by \Cref{prop:rank}, only focuses on 
the faces of a set $\conv(\X)$;  
\item The worse-case example in Appendix \ref{proof:tight} can attain the proposed rank bound in \Cref{them:rank}, which confirms the tightness;
% {\color{purple}The worse-case example in \cite{tantipongpipat2019multi}[appendix B] can be shown to attain the proposed rank bound in \Cref{them:rank}, which confirms the tightness;}
and

% Besides, our derivation can be applied to any \ref{eq_rel_rank} 
%in contrast to the existing ones, e.g., see the LSOP-R of  \ref{app:kernel} in the next subsection;
%\item The premise of finite optimal value in \Cref{them:rank} is used to guarantee extreme optimum of  \ref{eq_rel_rank}. In fact, all extreme points in LSOP-R feasible set satisfy the proposed rank bound, which naturally leads to  \Cref{cor:ext} below, as an extension of \Cref{them:rank};
\item 
%If the rank bound in  \Cref{them:rank} is no larger than $k$, then the \ref{eq_rel_rank} can exactly satisfy the rank-$k$ constraint. Consequently, 
An extra benefit of \Cref{them:rank} is to provide a sufficient condition about when \ref{eq_rel_rank} is equivalent to the original \ref{eq_rank}, as shown below.
%\item Motivated by the fact that not all $m$ linear inequalities are attributed to the formation of extreme points in  \ref{eq_rel_rank}, \Cref{cor:improve} presents the improved rank bounds by targeting those important  linear inequalities.
\end{enumerate}

%{add an example to illustrate this and permutation-invariant}
%
%\begin{corollary}\label{cor:ext}
%Given the positive semidefinite matrix space (i.e., $\Q:=\S_{+}^n$), each extreme point in the LSOP-R feasible set has a rank at most $\k + \lfloor\sqrt{2m+9/4}-3/2\rfloor,$ where $\k\le k$ follows \Cref{def:k}.
%\end{corollary}

\begin{proposition}\label{cor:ch}
Given   $\Q:=\S_{+}^n$ in \eqref{eq_set} and integer $\k\le k$ following \Cref{def:k},  if $\tilde{m} \le {(k-\k+2)(k-\k+3)}/{2} -2$ holds, we have that 
\begin{enumerate}[(i)]
\item Each feasible extreme point in the \ref{eq_rel_rank} has a rank at most $k$; and
\item The \ref{eq_rel_rank} achieves the same optimal value as the original \ref{eq_rank},
%There is an optimal solution to both \ref{eq_rel_rank} and \ref{eq_rank},
i.e., $\V_{\opt} = \V_{\rel}$ 
if the \ref{eq_rel_rank}  yields a finite optimal value, i.e., $\V_{\rel}>-\infty$.
\end{enumerate}
\end{proposition}
\begin{proof}
When $\tilde{m} \le {(k-\k+2)(k-\k+3)}/{2} -2$, the inequality  $\k + \lfloor\sqrt{2\tilde m+9/4}-3/2\rfloor \le k$ holds. Then, using the result in \Cref{them:rank}, 
 the rank of each extreme point in the \ref{eq_rel_rank} feasible set is no larger than $k$, and the \ref{eq_rel_rank} has an optimal extreme point  $\bm X^*$ of rank at most $k$. Given $-\infty <\V_{\rel}\le \V_{\opt}$, the feasible $\bm X^*$  must be also optimal to \ref{eq_rank}. Thus completes the proof.
\qed
\end{proof}

%\item Finally, it is interesting to note that \Cref{them:rank} can contribute an upper bound of the rank to the classic Matrix Completion (MC) problem, which is prevalent in machine learning and statistics. The bound  is discussed below. 
%\item Finally, it is interesting to note that \Cref{them:rank} can contribute to learning a Markov model for systems that admit a low-dimensional latent structure \citep{zhu2022learning} and Matrix Completion (MC) problem. 
%\end{enumerate}

\Cref{cor:ch} contributes to the literature of the \ref{eq_rel_rank} objective exactness that refers to the equation $\V_{\opt} = \V_{\rel}$, as it indicates that for any  \ref{eq_rank}, the corresponding \ref{eq_rel_rank} always achieves  the same optimal value whenever $\tilde m \le 1$ and $\V_{\rel}>-\infty$. % Previous research in this direction involves several special cases of our \ref{eq_rank}, including  \ref{app:qcqp} and \ref{app:fpca} (see, e.g., \citealt{kilincc2021exactness,li2022exactness}), and their analysis cannot be readily generalized. 

\subsection{Applying to Kernel Learning, Fair PCA, and QCQP }\label{sub:bound}
%{
This subsection revisits three  \ref{eq_rank} examples in positive semidefinite matrix space.
More specifically, our proposed rank bound in \Cref{them:rank}  recovers the existing results for  \ref{app:fpca} and  \ref{app:qcqp} from a different perspective, and can be applied to  \ref{app:kernel}, providing the first-known rank bound of its \ref{eq_rel_rank} counterpart. 

\textit{Kernel Learning.} First, recall that in the \ref{app:kernel}, 
%{it is equipped with a particular} 
its domain set $\D$ is defined for
% for 
any given value of variable $z$, i.e., 
$\D:=\{\bm X\in \S_+^n: \rank(\bm X)\le k,  \log\det(\bm X +\alpha \bm I) \ge z\},$
which, according to \Cref{prop:conv}, has an explicit representation of the convex hull $\conv(\X)$.
% for any given value $z$. 
Thus, the \ref{eq_rel_rank} corresponding to \ref{app:kernel} can be  written as 
\begin{equation}
\label{pc_kernel}
\begin{aligned}
\V_{\rel}:= \min_{\begin{subarray}{c}\bm X \in \conv(\X), z \in \Re\end{subarray}} \left\{\langle \bm Y^{\dag}, \bm X\rangle - z + \log\det(\bm Y+\alpha \bm I):
b_i^l \le \langle \bm A_i, \bm X\rangle \le  b_i^u, \forall i \in [m]\right\},  
\end{aligned}\tag{Kernel Learning-R} 
\end{equation}
where $ \conv(\X) := \{\bm X \in \S_+^{n}: \exists \bm x \in \Re_+^n, \sum_{i\in [n]} \log(\alpha + x_i)\ge z, 
 x_1 \ge \cdots \ge x_n, x_{k+1}=0, \|\bm X\|_{(\ell)} \le \sum_{i\in [\ell]} x_i, \forall \ell \in [k],\tr(\bm X) = \sum_{i\in [k]} x_i\}.$

Note that the existing works on rank bounds fail to cover \ref{pc_kernel}  (see, e.g., \citealt{barvinok1995problems, pataki1998rank, tantipongpipat2019multi}). 
Our results in \Cref{them:rank} and  \Cref{cor:ch} can fill this gap since by exploring the formulation structure, the \ref{pc_kernel}  can be viewed as a special case of the \ref{eq_rel_rank}. 
 
 \begin{restatable}{corollary}{corkernel}{\rm \textbf{(Kernel Learning)}} \label{cor:kernel}
 There exists an optimal solution of the \ref{pc_kernel} with rank at most  $k +  \left\lfloor \sqrt{2 {\tilde m}+{9}/{4}}-{3}/{2}\right\rfloor$. Besides,  if there is $\tilde m\le 1$ linearly independent inequality, the \ref{pc_kernel}  achieves the same optimal value as \ref{app:kernel}.
 \end{restatable}
\begin{proof}
	Suppose that $(\hat{\bm X}, \hat{z})$ is an optimal solution to the \ref{pc_kernel}. Then for a given solution $\hat{z}$, all the points in the following set are also optimal to \ref{pc_kernel}
	\begin{align*}
		\left\{\bm X \in \conv(\D): \langle \bm A_i, \bm X\rangle = \langle \bm A_i, \hat{\bm X}\rangle, \forall i \in [m]\right\}, 
	\end{align*}
	where  $\D:=\{\bm X\in \S_+^n: \rank(\bm X)\le k,  \log\det(\bm X +\alpha \bm I) \ge \hat{z}\}$ is based on $\hat{z}$.
	As $\tilde k\le k$ and the domain set $\D$ is bounded, according to \Cref{them:rank}, all the extreme points in the  optimal set above have a rank at most $k +  \lfloor \sqrt{2 \tilde{m}+9/4}-3/2\rfloor$. Thus,  there is an optimal solution $(\bm X^*, \hat{z})$ to the \ref{pc_kernel} with $\bm X^*$ satisfying the desired rank bound. Finally, the rank bound does not exceed $k$ when $\tilde m\le 1$, implying that $(\bm X^*, \hat{z})$ is also optimal to the original \ref{app:kernel}.
	\qed
\end{proof}

 \textit{Fair PCA.} Another application of \Cref{them:rank} is to provide a  rank bound for partial convexification \ref{eq_rel_rank} of the \ref{app:fpca} defined as follows
\begin{equation}
    \begin{aligned}\label{pc_fpca}
\V_{\rel}:= \max_{z\in \Re, \bm X \in \conv(\X)} \left\{z: z\le \langle  \bm A_i, \bm X \rangle, \forall i \in [m]\right\}, \\
 \conv(\X)= \{\bm X \in \S_+^{n}: \tr(\bm X)\le k, ||\bm X||_2 \le 1 \}.
\end{aligned}\tag{Fair PCA-R}
\end{equation}
Note that the  \ref{app:fpca} is equipped with a domain set $\X:=\{\bm X\in \S_{+}^n: \rank(\bm X)\le k, \|\bm X\|_2\le1 \}$, whose convex hull can be readily derived without the majorization technique, as shown in \ref{pc_fpca}. The rank bound below recovers the one in \cite{tantipongpipat2019multi}. 

\begin{restatable}{corollary}{corspca}{\rm \textbf{(Fair PCA)}}\label{cor:spca}
There exists an optimal solution of the  \ref{pc_fpca} with rank at most  $k +  \left\lfloor {\sqrt{2m + 1/4}}-{3}/{2}\right\rfloor$. Besides,  if there are $m\le 2$ sample covariance matrices, the \ref{pc_fpca}   achieves the same optimal  value as \ref{app:fpca}.
\end{restatable}
\begin{proof}
    See Appendix \ref{proof:spca}. \qed
\end{proof}
% From a different angle, 

We close this subsection by discussing a special domain set $\X:=\{\bm X \in \S_{+}^n: \rank(\bm X)\le 1\}$, which  can be naturally adopted in \ref{app:qcqp}. The  \ref{eq_rel_rank} under this setting reduces to a semidefinite program relaxation of  \ref{app:qcqp}:
\begin{align} \label{pc:qcqp}
\V_{\rel}:= \min_{\bm X \in \conv(\X)} \left\{\langle\bm A_0, \bm X\rangle: 
b_i^l \le\langle \bm A_i, \bm X\rangle \le  b_i^u, \forall i \in [m] \right\},  \ \  \conv(\X):=\S_+^n. \tag{QCQP-R}
\end{align}
We show that \Cref{them:rank} can recover a well-known rank bound of \ref{pc:qcqp} (see, e.g., \citealt{pataki1998rank}).

\begin{restatable}{corollary}{corqcqp}{\rm \textbf{(QCQP)}}\label{cor:qcqp}
There exists an optimal solution of \ref{pc:qcqp}  with rank at most  $1 +  \left\lfloor {2 {\tilde m}+{1}/{4}}-{3}/{2}\right\rfloor$ if  $\V_{\rel}>-\infty$. Besides, if there are $\tilde m\le 2$ linearly independent inequalities, the \ref{pc:qcqp}   achieves the same optimal value as the original \ref{app:qcqp} if  $\V_{\rel}>-\infty$.
\end{restatable}
\begin{proof}
    See Appendix \ref{proof:qcqp}. \qed
\end{proof}

 \Cref{cor:qcqp} implies that the \ref{pc:qcqp} coincides with the original \ref{app:qcqp} when there are $\tilde m\le 2$ linearly independent constraints, which is consistent with those of previous studies on the objective exactness of \ref{pc:qcqp} (see \citealt{burer2020exact,kilincc2021exactness,li2022exactness}).

%Following Theorem 9 and Theorem 10, from the primal and dual perspectives, one can furtherstrengthen the results of Theorem 11 by considering binding constraints at optimality and thenonnegative Lagrangian dual multpliers, respectively, which of course requires extra assumptions.
%
%When optimizing a linear objective function over set $\P(k)$, it is of great interest to investigate the rank bound of optimal extreme points.  

%\noindent \textbf{Matrix Completion.} 

%Note that when analyzing  the symmetric case  (i.e., $\Q=\S^n$), the rank upper bound and exactness result in Part (ii) of \Cref{them:rank} and \Cref{cor:ext} build on both \Cref{asum1} and sign-invariant  functions. 
%Motivated by \Cref{cor:sym} that describes $\conv(\X)$ for the symmetric case with only  \Cref{asum1},  
%we derive an  upper bound of the rank under this setting, aiming to generalize Part (ii) of \Cref{them:rank}. Specifically, using \Cref{cor:sym},  
%the LSOP-R problem \eqref{eq_rel_rank} can be decomposed into $\k+1$ subproblems as described in \eqref{eq_rel_rank2},  and taking the minimum optimal value among these subproblems offers a stronger convex relaxation than the original LSOP-R. 

\section{The Rank Bound in Non-symmetric Matrix Space}\label{sec:nonsym}
This section  derives a rank bound of \ref{eq_rel_rank}  in non-symmetric matrix space, i.e., $\Q=\Re^{n\times p}$ in \eqref{eq_set}. % by extending the analyses in \Cref{sec2}. 
Due to the non-symmetry, %{the eigen-decomposition in \Cref{sec2} is not possible in this section; hence, 
the analysis relies on the singular value decomposition in this section. % will  be emphasized.
% Surprisingly, non-symmetry. 
% brings in extra
% % modeling 
% flexibility and  allows us to
% % we can 
% improve 
% the rank bound of order $\sqrt{2\tilde m}$ in \Cref{them:rank}
% % , which is 
% to be of order $\sqrt{\tilde m}$.
% almost linear in $\sqrt{\tilde m}$.
More importantly, the nonnegative singular values enable us to rewrite function $f(\cdot)$ in \eqref{eq_set} as 
%For any non-symmetric matrix $\bm X\in \Re^{n\times p}$, the convex spectral function $F(\bm X):=f(\bm \lambda(\bm X))$ in domain set $\X$ is defined by the nonnegative singular values of matrix $\bm X$, i.e., $\bm \lambda(\bm X)\in \Re_+^n$. Thus, equivalently, we can rewrite 
$f(\bm \lambda(\bm X)):=f(|\bm \lambda(\bm X)|)$, which admits the sign-invariant property. Therefore, when $\Q=\Re^{n\times p}$,  the domain set $\D$ in \eqref{eq_set} is permutation- and sign-invariant with singular values, and its convex hull $\conv(\X)$ can be readily described   according to \cite{kim2022convexification} (see $\conv(\X)$ in Appendix \ref{proof:nonface}). 
% unfortunately,  may not be always available, and instead we provide a  relaxation of set $\conv(\X)$  as shown below, which suffices for our analysis.

% Inspired by Part (ii) in \Cref{prop:nonsymconv}, we redefine  the \ref{eq_rel_rank} built on set  $ \Y$ instead of $\conv(\X)$.
%  \begin{align} \label{eq_rel_rank}
%  \hat{\V}_{\rel}:=\min_{{\bm X \in  \Y}}\left\{\langle\bm A_0, \bm X \rangle: 
%  b_i^l \le  \langle \bm A_i, \bm X\rangle \le b_i^u, \forall i \in [m] \right\} \le \V_{\rel} \le \V_{\opt}, \tag{LSOP-R-I}
%  \end{align}
%  \hl{any example set $Y \neq$ set X?}
% where set $\Y$ is defined in equation \eqref{eq_sety}.
% Note that we term the relaxation as ``LSOP-R-I"  to differentiate with \ref{eq_rel_rank}. %while replacing $\conv(\X)$ by set $\Y$,  the \ref{eq_rel_rank}  still serves as a partial relaxation of \ref{eq_rank} over domain set $\X$.
% It is worthy of mentioning that the \ref{eq_rel_rank}  is equivalent to the \ref{eq_rel_rank} (i.e., $ \conv(\X)=\Y$ and $\hat{\V}_{\rel} = \V_{\rel}$) under some conditions.
% \begin{remark}\label{remark1}
% Given  $\Q:=\S_{+}^n$ in \eqref{eq_set}, the convex hull
% $ \conv(\X)$ defined in \Cref{prop:conv} is exactly equal to set $\Y$ and thus  $\hat{\V}_{\rel}=\V_{\rel}$. Given $\Q:=\Re^{n\times p}$ and a  sign-invariant  function $f(\cdot)$ in \eqref{eq_set}, we must have $ \conv(\X)=\Y$ following Part (iii) in \Cref{prop:nonsymconv} and thus  $\hat{\V}_{\rel}=\V_{\rel}$.
% \end{remark}

We derive a rank bound below for any fixed-dimension face in the set $\conv(\X)$ with $\Q:=\Re^{n\times p}$. %, and the majority of analyses runs parallel to Subsection \ref{sub22}. 

\begin{restatable}{theorem}{thmnonface}\label{them:nonface}
Given  $\Q:=\Re^{n\times p}$ in domain set $\X$ in \eqref{eq_set}, suppose that $F^d\subseteq \conv(\X)$ is a $d$-dimensional face in the convex hull of the domain set $\X$. Then any point in face $F^d$ has a rank at most $\tilde k + \lfloor\sqrt{2d+9/4}-3/2\rfloor$, where  $\k\le k$ follows \Cref{def:k}.
\end{restatable}
\begin{proof}
    See Appendix \ref{proof:nonface}. \qed
\end{proof}
 
In \Cref{them:nonface}, we guarantee the rank bounds for all faces of set $\conv(\X)$ with $\Q:=\Re^{n\times p}$. 
 Analogous to the link between \Cref{prop:rank} and \Cref{them:rank},  the statement in \Cref{lem:face}  still holds for $\Q:=\Re^{n\times p}$. Hence, \Cref{them:nonface}  directly implies a rank bound for \ref{eq_rel_rank}.
 
 \begin{theorem}\label{them:nonsym}
 Given  $\Q:=\Re^{n\times p}$ in \eqref{eq_set} and integer  $\k\le k$ following \Cref{def:k}, suppose that the \ref{eq_rel_rank} admits a line-free feasible set, then  we have that
 \begin{enumerate}[(i)]
 \item Each feasible extreme point in the \ref{eq_rel_rank} has a rank at most $\k + \lfloor\sqrt{2\tilde m+9/4}-3/2\rfloor$; and
 \item There is an optimal solution to \ref{eq_rel_rank} of rank at most $\k + \lfloor\sqrt{2\tilde m+9/4}-3/2\rfloor$ if the \ref{eq_rel_rank}  yields a finite optimal value, i.e., ${\V}_{\rel}>-\infty$.
 \end{enumerate}
 \end{theorem}

We make the several remarks about the rank bound in \Cref{them:nonsym} with $\Q=\Re^{n\times p}$.
\begin{enumerate}[(a)]
\item The premise of a line-free  feasible set ensures the existence of extreme points in the feasible set of \ref{eq_rel_rank}   \citep{rockafellar2015convex};
\item  When $\Q:=\S_{+}^n$ or $\Q:=\Re^{n\times p}$, both \Cref{prop:rank,them:nonface} are based on 
% in fact investigate the faces of 
analyzing the perturbation of majorization constraints in a permutation-invariant 
set $\conv(\X)$, which leads to the same rank bounds in \Cref{them:rank,them:nonsym}. %, {which sheds light on the importance of permutation invariance in deriving rank bounds. 
This  motivates us to
% This implies that the the majorization 
% plays a crucial role in deriving rank bounds. Motivated by this, we 
explore another type of 
% majorization-based 
permutation-invariant set beyond $\conv(\X)$ to derive the rank bounds in symmetric yet
% but not 
indefinite matrix space in \Cref{sec3};
\item The rank bound in \Cref{them:nonsym} is tight (see Appendix \ref{proof:tight}); and
%\item \ref{eq_rel_rank} can be explicity described for any domain set $\X$
% \item {\Cref{them:rank,them:nonsym} present the same rank bounds for \ref{eq_rel_rank} since the }
% \item 
% The differences between \Cref{them:rank} and \Cref{them:nonsym}  are mainly because any $n\times n$ symmetric matrix  can be converted to a vector of length $n(n+1)/2$, while being $n^2$ in non-symmetric case. 
% Such vector length decides the magnitude of rank bounds. %, which can be further 
% This phenomenon is also observed  in the next section for the symmetric case.
%\item  Likewise, the rank bound in \Cref{them:nonsym} can be further improved by targeting right linear inequalities as discussed in \Cref{cor:improve}, which is omitted for brevity.
\item A notable side product of \Cref{them:nonsym} is that by letting $\k+ \lfloor\sqrt{2m+9/4}-3/2\rfloor\le k$, \Cref{them:nonsym} reduces to a sufficient condition of the exactness of \ref{eq_rel_rank}. 
\end{enumerate}

%yields the same optimal solution as the original \ref{eq_rank}.
 
\begin{proposition}\label{cor:nonsym}
Given  $\Q:=\Re^{n\times p}$ in \eqref{eq_set}  and integer $\k\le k$ following \Cref{def:k}, if $\tilde m \le {(k-\k+2)(k-\k+3)/2}-2$ holds,  then we have that
\begin{enumerate}[(i)]
\item Each feasible extreme point in the \ref{eq_rel_rank} has a rank at most $k$; and
\item The \ref{eq_rel_rank} achieves the same optimal value as the original \ref{eq_rank}, 
%There is an optimal solution to both \ref{eq_rel_rank} and \ref{eq_rank}, 
i.e., $\V_{\opt} = \V_{\rel}$ 
if the \ref{eq_rel_rank}  yields a finite optimal value, i.e., $\V_{\rel}>-\infty$ and admits a line-free feasible set.
\end{enumerate}
\end{proposition}
\begin{proof}
Using \Cref{them:nonsym}, the proof is identical to that of \Cref{cor:ch} and thus is omitted. \qed
\end{proof}

\subsection{Applying to Fair SVD and Matrix Completion } \label{sub:nonsymapp}
This subsection investigates two \ref{eq_rank} application examples with $\Q:=\Re^{n\times p}$  in Subsection \ref{sub:app} and provides the first-known rank bounds for  their corresponding \ref{eq_rel_rank}s by leveraging \Cref{them:nonsym}.

\textit{Fair SVD.} The domain set in \ref{app:fsvd} is $\X:=\{\bm X\in \Re^{n\times p}: \rank(\bm X)\le 1, \|\bm X\|_2 \le 1\}$. %is both permutation- and sign-invariant with singular values, which enables us to  \Cref{prop:nonsymconv}. 
Therefore, for  \ref{app:fsvd}, its \ref{eq_rel_rank} can be formulated below
\begin{equation}
\begin{aligned}\label{pc_fsvd}
&\V_{\rel}:= \max_{z\in \Re, \bm X \in \conv(\X)} \left\{z: z\le \langle  \bm A_i, \bm X \rangle, \forall i \in [m]\right\},  \\
 &\ \conv(\X)= \{\bm X \in \Re^{n\times p}: \|\bm X\|_*\le k, ||\bm X||_2 \le 1 \}.
 \end{aligned}
 \tag{Fair SVD-R}
\end{equation}
Note that in \ref{pc_fsvd}, we have that $\k=k$   and the domain set $\X$ is compact, which enables us to directly apply  \Cref{them:nonsym}.

 \begin{restatable}{corollary}{corfsvd}{\rm \textbf{(Fair SVD)}} \label{cor:fsvd}
 There exists an optimal solution of \ref{pc_fsvd} with rank at most  $k +  \lfloor \sqrt{2m + 1/4}-3/2\rfloor$. Besides,  if there are $m\le 2$ groups of data matrices, the \ref{pc_fsvd}  yields the same optimal value as the original \ref{app:fsvd}.
 \end{restatable}
\begin{proof}
	See Appendix \ref{proof:corfsvd}. \qed
	\end{proof}
 
%\begin{remark}[FSVD]
%According to \Cref{them:nonsym}, the \ref{eq_rel_rank} corresponding to FSVD \eqref{app:fsvd} has an optimal extreme point  with the rank at most $k+\sqrt{m+1}-1$. Besides,  if there are $m\le 2$ data matrices, the \ref{eq_rel_rank}  can achieve the same optimal solution and value as FSVD \eqref{app:fsvd} based on \Cref{cor:nonsym}.
%{improve to $m\le 3$}
%\end{remark} 

\textit{Matrix Completion.}  \ref{app:lmc} that performs low-rank recovery from observed samples is a popular technique in machine learning, which is known to be
notoriously NP-hard. Alternatively, we study the following convex relaxation   of \ref{app:lmc}: %, which is formulated by
\begin{equation} \label{pc_lmc}
\begin{aligned}
\V_{\rel}:=\min_{\bm X\in \conv(\X), z \in \Re_+} \left\{z:  X_{ij}=A_{ij} , \forall (i,j)\in \Omega \right\}, \ \  \conv(\X):=\{\bm X\in \Re^{n\times p}:  \|\bm X\|_*\le z\},
\end{aligned} \tag{Matrix Completion-R}
\end{equation}
which is equivalent to $\min_{\bm X\in \Re^{n\times p}} \left\{\|\bm X\|_*:  X_{ij}=A_{ij} , \forall (i,j)\in \Omega \right\}.$
Despite the popularity of the convex relaxation-- \ref{pc_lmc} in the literature  (see, e.g., \citealt{cai2010singular, miao2021color}), there is no  guaranteed rank bound of its solution so far.
 We fill this gap by deriving the following rank bound for the \ref{pc_lmc}.
 \begin{restatable}{corollary}{corlmc}{\rm \textbf{(Matrix Completion)}}\label{cor:lmc}
 Suppose that there are  $m:= |\Omega|$ observed entries in \ref{app:lmc}. Then there exists an optimal solution to \ref{pc_lmc} 
 % optimal extreme point solution has a
 with rank at most  $\lfloor\sqrt{2|\Omega|+9/4}-1/2\rfloor$. Besides,  if there are $|\Omega|\le (k+1)(k+2)/2-2$ observed entries, the \ref{pc_lmc}  can achieve the same optimal value as the original \ref{app:lmc}.
\end{restatable}
\begin{proof}
    See Appendix \ref{proof:lmc}. \qed
\end{proof}

%It is worth mentioning that the result in \Cref{cor:lmc} is independent of the constant $U$ in partial convexfication \eqref{pc_lmc}, so the selection of the most suitable constant can be made based on the specific LMC \eqref{app:lmc}. More details are discussed in the numerical study section.
%Finally, following the spirit of \Cref{cor:face} for $\Q:=\S_{+}^n$, it is worth mentioning that the domain set $\X$ and its convex hull admit the following inclusion property under the setting of $\Q:=\Re^{n\times p}$.
%
%  \begin{corollary}\label{cor:nonface}
%  Given  the non-symmetric matrix space $\Q:=\Re^{n\times p}$, suppose that function $f(\cdot)$ is nondecreasing, the domain set $\X$ contains each no larger than $({(k-\k+2)^2} -2)$-dimensional face of its convex hull $ \conv(\X)$.
%  \end{corollary}

\section{Rank Bounds in  Symmetric Indefinite Matrix Space} \label{sec3}
This section  focuses on the symmetric indefinite matrix space, i.e., $\Q:=\S^n$ and its special case--diagonal matrix space which can be applied to \ref{app:sparse}. So far, scant attention has been paid to the \ref{eq_rank} of $\Q:=\S^n$ in literature. In contrast to a matrix in $\Q:=\S_+^n$ or $\Q:=\Re^{n\times p}$ that formulates the spectral and rank constraints by the nonnegative eigenvalues or singular values, respectively, a symmetric indefinite matrix  may have both negative and positive eigenvalues. 
This requires a different analysis of
% different technique to derive 
rank bounds. 
Specifically,  we discuss two cases of $\Q:=\S^n$  
depending on whether the function $f(\cdot)$ in \eqref{eq_set} is sign-invariant or not. %, which lead to different rank bounds for \ref{eq_rel_rank}.

 \subsection{The Rank Bound with Sign-Invariance}\label{subsec:sign}
%To avoid simultaneously negative and positive eigenvalues, a simple idea is to use the sign-invariant property  to reduce the symmetric to non-symmetric case. 
In this subsection,  we study the case where the function $f(\cdot)$ in domain set $\X$ in \eqref{eq_set} is sign-invariant, i.e., $f(\bm \lambda)=f(|\bm \lambda|)$ for any $\bm \lambda\in \Re^n$. In this case, it suffices to focus on the absolute eigenvalues of a symmetric matrix, i.e., singular values.
Therefore,  the  rank bound in \Cref{them:nonsym} for the non-symmetric matrix space can be extended, as shown below. 

\begin{restatable}{theorem}{thmsymsign}
    \label{them:symsign}
Given  $\Q:=\S^n$ in \eqref{eq_set} and integer  $\k\le k$ following \Cref{def:k}, suppose that the function $f(\cdot)$ in \eqref{eq_set} is sign-invariant and \ref{eq_rel_rank} admits a line-free feasible set, then we have that
\begin{enumerate}[(i)]
\item Each feasible extreme point in the \ref{eq_rel_rank} has a rank at most $\k + \lfloor\sqrt{2\tilde m+9/4}-3/2\rfloor$; and
\item There is an optimal solution to \ref{eq_rel_rank} of rank at most $\k + \lfloor\sqrt{2\tilde m+9/4}-3/2\rfloor$ if the \ref{eq_rel_rank} yields a finite optimal value, i.e., $\V_{\rel}>-\infty$. 
\end{enumerate}
\end{restatable}
\begin{proof}
    See Appendix \ref{proof:symsign}. \qed
\end{proof}

For the case of  $\Q:=\S^n$ with sign-invariant property, we obtain the same  rank bound   as those in \Cref{them:rank,them:nonsym}. Besides, there is an example of \ref{eq_rel_rank} that attains this rank bound (see Appendix \ref{proof:tight}).
% \item For the case of $\Q:=\S_{+}^n$ or $\Q:=\S^n$ with sign-invariant property, we obtain the same  rank bound in the order of $\O(\sqrt{2m})$  as shown in \Cref{them:rank} and \Cref{them:symsign}. In contrast, the non-symmetric case of $\Q:=\Re^{n\times p}$ leads to a rank bound almost linear in $\sqrt{m}$  in \Cref{them:nonsym}. As previously emphasized, given different matrix spaces, this deviation arises due to the distinct lengths of their corresponding vectors;
Following the spirit of \Cref{cor:ch},  \Cref{them:symsign} also gives rise to a sufficient condition about when the \ref{eq_rel_rank} matches the original \ref{eq_rank}  as shown in Appendix \ref{proof:symsign}.

 \subsection{Rank Bounds  Beyond Sign-Invariance}\label{subsec:nosign}
When the function $f(\cdot)$ in the domain set $\X$ with $\Q:=S^n$ in \eqref{eq_set} is not sign-invariant, we analyze a tighter tractable relaxation for \ref{eq_rank} than the \ref{eq_rel_rank}, termed LSOP-R-I and then derive the rank bound of \ref{eq_rel_rank1}, which can be applied to sparse optimization in the next subsection.

%Before proceeding to LSOP-R\#, let us first show the convex hull of domain set $\X$ in symmetric indefinite matrix space. 
As indicated previously, the  positive and negative eigenvalues may simultaneously exist in a symmetric matrix; therefore, even if we have  eigenvalues of a symmetric matrix  sorted in descending order, the rank constraint cannot be readily dropped  without precisely locating zero eigenvalues. This  indicates a need of disjunctive description for $\conv(\X)$ as detailed below. 

%
%To begin with, following the spirit of \Cref{prop:conv}, we study the convexification of domain set $\X$ in symmetric indefinite matrix space.
\begin{restatable}{proposition}{propsymconv} \label{prop:symconv}
For a domain set $\X$ with $\Q:=\S^n$ and an integer $\k\le k$  in \Cref{def:k},  we have 
\begin{enumerate}[(i)]
\item $\conv(\X) = \conv\left( \cup_{s \in[\k+1]} \Y^s\right) $; and
\item $\X \subseteq \cup_{s \in[\k+1]} \Y^s $,
\end{enumerate}
where for each $s\in[\k+1]$,  set $\Y^s$ is closed and is defined as
\begin{equation}\label{eq_ys}
\begin{aligned}
\Y^{s}:= \left\{\bm X \in \S^n: \exists \bm x \in \Re^n,  f(\bm x)\le 0,  x_1\ge  \cdots \ge x_n, x_{s} = x_{s+n-\k-1}=0,
\bm x \succeq \bm \lambda(\bm X) \right\}.
\end{aligned}
\end{equation}
Note that $\bm \lambda(\bm X)\in \Re^n$ denotes the eigenvalue vector of the symmetric matrix $\bm X \in \S^n$.
\end{restatable}
\begin{proof}
    See Appendix \ref{proof:symconv}. \qed
\end{proof}

%Note that when analyzing  the symmetric case  (i.e., $\Q=\S^n$), the rank upper bound and exactness result in Part (ii) of \Cref{them:rank} and \Cref{cor:ext} build on both \Cref{asum1} and sign-invariant  functions. 
%Motivated by \Cref{cor:sym} that describes $\conv(\X))$ for the symmetric case with only  \Cref{asum1},  
%we derive an  upper bound of the rank under this setting, aiming to generalize Part (ii) of \Cref{them:rank}. Specifically, using \Cref{cor:sym},  
%the LSOP-R problem \eqref{eq_rel_rank} can be decomposed into $k+1$ subproblems as described in \eqref{eq_rel_rank2},  and taking the minimum optimal value among these subproblems offers a stronger convex relaxation than the original LSOP-R. 

Note that (i) the majorization $\bm x \succeq \bm \lambda(\bm X)$ admits a tractable semidefinite representation according to \cite{ben2001lectures}
%(ii) Using disjunctive  programming,  we can represent set $\conv( \cup_{s \in[\k+1]} \Y^{s}) $ in the extended space; 
and (ii) the inclusions $\X \subseteq \cup_{s \in[\k+1]} \Y^{s} \subseteq \conv(\D)$ in \Cref{prop:symconv} 
%that
%\begin{observation}
%Given $\Q:=\S^n$, the convex hull $\conv(\X)$  described in \Cref{prop:symconv} satisfies
%\begin{align*}
%\X \subseteq \cup_{s \in[\k+1]} \Proj_{\bm X}(\Y^{s}) \subseteq \conv(\X) =\conv( \cup_{s \in[\k+1]} \Proj_{\bm X}(\Y^{s}) ),
%\end{align*}
%where for each $s\in [\k+1]$, set $\Y^s$ is defined in equation \eqref{eq_ys}.
%\end{observation}
%In fact, the union $\cup_{s \in[\k+1]} \Proj_{\bm X}(\Y^{s})$ can be a proper subset of convex hull $\conv(\X)$, which 
inspire us a stronger partial relaxation  for \ref{eq_rank} with $\Q:=\S^n$, i.e., 
\begin{align} \label{eq_rel_rank1}
\V_{\rel} \le \V_{\rel\text{-\rm I}}:= \min_{s\in [\k+1]} \min_{{\bm X \in \Y^{s}}}\left\{\langle\bm A_0, \bm X \rangle: 
b_i^l \le  \langle \bm A_i, \bm X\rangle \le b_i^u, \forall i \in [m] \right\} \le \V_{\opt}, \tag{LSOP-R-I}
\end{align}
where integer $\k$ follows \Cref{def:k} and for each $s\in [\k+1]$, set $\Y^s$ is defined in  \eqref{eq_ys}. %In particular, the strict inequality $\V_{\rel} < \V_{\rel\text{-\rm I}}$ holds.

%It is worth mentioning that when $\Q:=\S_{+}^n$ or $\Q:=\Re^{n\times p}$, we have that $\conv(\X):=\Y_{}$ the integer $s$ achieves the unique value $(\k+1)$; thus, the \ref{eq_rel_rank1} reduces to \ref{eq_rel_rank} or \ref{eq_rel_rank}. 

Then, to derive the rank bound for the tighter \ref{eq_rel_rank1}, it is sufficient to study all inner minimization problems over matrix variable $\bm X$. For each $s\in [\k +1]$, we see that the set $\Y^{s}$ in \eqref{eq_ys} that builds on the  majorization constraints is permutation-invariant with eigenvalues of $\bm X$, % has a similar structure to set $\conv(\X)$ in \Cref{prop:conv} with $\Q:=\S_+^n$, 
which motivates us to derive the following rank bounds for faces of set $\Y^s$. % into the current setting.
\begin{restatable}{theorem}{thmsymface}\label{them:symface}
Given $\Q:=\S^{n}$ in \eqref{eq_set}, for each $s\in [\k+1]$, suppose that $F^d\subseteq \Y^s$ is a $d$-dimensional face of the closed convex set $\Y^s$ defined in  \eqref{eq_ys}, then any point in face $F^d$ has rank at most 
\begin{align*}
\begin{cases}
\k + \lfloor\sqrt{2d+9/4}\rfloor -3/2, & \text{if }  s\in \{1, \k+1\};\\
\k + \lfloor\sqrt{4d+9}\rfloor -3, & \text{otherwise.}  
\end{cases},
\end{align*}
 where  $\k\le k$ follows \Cref{def:k}.
\end{restatable}
\begin{proof}
    See Appendix \ref{proof:symface}. \qed
\end{proof}
The results in \Cref{them:symface} lead to the following rank bounds for the relaxation \ref{eq_rel_rank1}. 
\begin{theorem}\label{them:sym_rank}
Given   $\Q:=\S^n$ in \eqref{eq_set} and integer  $\k\le k$ following \Cref{def:k}, suppose that the inner minimization of \ref{eq_rel_rank1} always admits a line-free feasible set, then we have that
\begin{enumerate}[(i)]
\item For each $s\in [\k+1]$, any feasible extreme point in the $s$th inner minimization  of \ref{eq_rel_rank1} has a rank at most $\k + \lfloor\sqrt{2\tilde m+9/4}-3/2\rfloor$ if $s\in \{1, \k+1\}$, and $\k + \lfloor\sqrt{4\tilde m+9}-3\rfloor$, otherwise;
\item There exists an optimal solution to \ref{eq_rel_rank1} of rank at most $\k + \lfloor\sqrt{4\tilde m+9}-3\rfloor$ if
the \ref{eq_rel_rank1}  yields a finite optimal value, i.e.,
$\V_{\rel\text{-\rm I}} >-\infty$. 
\end{enumerate}
\end{theorem}

For the rank bound results in \Cref{them:sym_rank}, we remark that
\begin{enumerate}[(a)]
\item When $s\in \{1, \k+1\}$, based on the definition of set $\Y^s$ in \eqref{eq_ys}, all the eigenvalues of a feasible extreme point  of the $s$th inner minimization  of \ref{eq_rel_rank1}   are either nonpositive or nonnegative. Hence, the rank bound recovers the one in \Cref{them:rank}; 
\item When $s\in [2,\k]$, each inner minimization of \ref{eq_rel_rank1}  has the same rank bound. In this setting, we  split a symmetric indefinite matrix into two parts with respective to positives and negative eigenvalues; then analyze the rank bounds of one positive definite matrix and one negative definite matrix, respectively; and  integrate the rank bounds. As a result, the rank bound  gets slightly worse and is in the order of $\O(\sqrt{4\tilde m})$ rather than $\O(\sqrt{2\tilde m})$; and
\item The rank bounds in \Cref{them:sym_rank} shed light on \ref{eq_rel_rank1} objective exactness 
% the similar special cases as \Cref{cor:ch} with $\Q:=\S^{n}$, 
as shown below.
\end{enumerate}
\begin{proposition}\label{cor:sym}
Given $\Q:=\S^{n}$ in \eqref{eq_set} and integer $\k\le k$ following \Cref{def:k},  when $ \tilde m  \le \frac{(k-\k+3)(k-\k+5)}{4}-2$, the \ref{eq_rel_rank1} achieves the same optimal value as the original \ref{eq_rank} if the \ref{eq_rel_rank1}  yields a finite optimal value, i.e., $\V_{\rel\text{-\rm I}}>-\infty$ and admits a line-free feasible set.
\end{proposition}

\subsection{Rank Bound in  Diagonal Matrix Space and Applying to Sparse Ridge Regression}
This subsection studies  the partial convexification \ref{eq_rel_rank} with respect to the sparse optimization in  vector space, which restricts the number of nonzero elements of any feasible solution (see, e.g., \citealt{li2020exact}). Since a vector can be equivalently converted into a  diagonal matrix, 
 we introduce  the following special domain set $\X$ adopted in  the sparse optimization problem:
% The rank bounds 
% with $\Q:=\S^n$ in \eqref{eq_set} are useful to derive the sparsity bounds {for the partial convexification of} the sparse optimization. 
% Since a vector can be equivalently converted into a symmetric diagonal matrix, the zero-norm of a vector equals the rank of its corresponding diagonal matrix. 
%  Formally, this subsection focuses on a particular domain set below, {provided that $\Q$ denotes the symmetric indefinite diagonal matrix space   in \eqref{eq_set}}
 \begin{equation} \label{set:sparse}
 \begin{aligned}
\X := \{ \bm X \in \S^n:
 \bm X =\Diag(\diag(\bm X)), \rank(\bm X) \le k, F(\bm X):=f(\diag(\bm X)) \le 0 \},
 \end{aligned}
 \end{equation}
%\hl{write down LSOP}
where the variable $\bm X$ is a   diagonal matrix and function $f(\cdot)$ is continuous, convex, and symmetric.
%Since a vector can be equivalently converted into a symmetric diagonal matrix,  
Thus, the rank bound of \ref{eq_rel_rank}  over the convex hull of the  domain set $\X$ in \eqref{set:sparse} is in fact the sparsity bound. To be consistent, we stick to  the rank bound in the below. 
 
The sparse domain set $\X$ in \eqref{set:sparse} admits a   diagonal matrix space, a special case of $\Q:=\S^n$, which enables us to extend the rank bounds in \Cref{them:symsign} and \Cref{them:sym_rank} depending on whether the function $f(\cdot)$ in \eqref{set:sparse} is sign-invariant or not, as summarized below. Similarly, given a domain set $\X$ in \eqref{set:sparse}, the analysis of rank bounds of \ref{eq_rel_rank}  or \ref{eq_rel_rank1}  requires an explicit description of the corresponding  set $\conv(\X)$ or the union of sets $\{\Y^s\}_{s\in [\k+1]}$ (see Appendix \ref{proof:sparseface}). 

\vspace{-0.5em}
\begin{restatable}{theorem}{themsparsesign}
    \label{them:sparsesign}
Given a sparse domain set $\X $ in \eqref{set:sparse} and integer $\k\le k$ in \Cref{def:k}, 
suppose that the function $f(\cdot)$  in \eqref{set:sparse} is sign-invariant and \ref{eq_rel_rank} admits a line-free feasible set. Then we have 
\begin{enumerate}[(i)]
\item Each feasible extreme point in the \ref{eq_rel_rank} has a rank at most $\k + \tilde m$; and
\item There is an optimal solution to \ref{eq_rel_rank} of rank at most $\k + \tilde m$ if
the \ref{eq_rel_rank}  yields a finite optimal value, i.e.,
$\V_{\rel}>-\infty$.
\end{enumerate}
\end{restatable}
%\begin{proof}
%%The domain set $\X$ admits a convex hull $\conv(\X):=\hat \Y$ as described in \Cref{prop:convsparse}. 
%Plugging the rank bounds of the set $\conv(\X)$ in \Cref{them:sparseface}, the rest of the proof is nearly identical to that in \Cref{them:rank} and thus is omitted. \qed
%\end{proof}

%For a general domain set $\X $ in \eqref{set:sparse}, its convex hull involves the union of sets $\{\hat \Y^s\}_{s\in [\k+1]}$ as shown in Part (ii) of \Cref{prop:convsparse}, which motivates us to instead use the stronger relaxation \ref{eq_rel_rank1} including $(\k+1)$ inner minimization problems over $\{\hat \Y^s\}_{s\in [\k+1]}$.

\vspace{-1em}
\begin{restatable}{theorem}{themsparsenonsign}
    \label{them:sparsenonsign}
Given a sparse domain set $\X $ in \eqref{set:sparse} and integer $\k\le k$ following \Cref{def:k},  suppose that the inner minimization of \ref{eq_rel_rank1} always admits a line-free feasible set. Then
 we have 
\begin{enumerate}[(i)]
\item For any $s\in [\k+1]$, each feasible extreme point in the $s$th inner minimization of \ref{eq_rel_rank1} has a rank at most $\k + \tilde m$; and
\item There is an optimal solution to \ref{eq_rel_rank1} of rank at most $\k + \tilde m$ if
the \ref{eq_rel_rank1}  yields  a finite optimal value, i.e.,
$\V_{\rel\text{-\rm I}}>-\infty$.
\end{enumerate}
\end{restatable}
%\begin{proof}
%According to  \Cref{them:sparseface}, any $d$-dimensional face of set $\hat \Y^s$ satisfies the rank-$(\k+d)$ for all $s\in [\k+1]$. Using \Cref{lem:face} and following the analysis of \Cref{them:rank}, we can show that for each inner minimization of  \ref{eq_rel_rank1}, all the extreme points in the intersection of  set $\hat \Y^s$ and $\tilde m$ linearly independent inequalities admit a rank-$(\k+\tilde m)$ bound. \qed 
%\end{proof}

The proof of the rank bound in \Cref{them:sparsesign,them:sparsenonsign}  can be found in Appendix \ref{proof:sparseface}. We note that
\begin{enumerate}[(i)]
\item  For the \ref{eq_rank} with a sparse domain set $\X $ in \eqref{set:sparse}, its  relaxation corresponds to \ref{eq_rel_rank} or \ref{eq_rel_rank1} which depends on the property of function $f(\cdot)$. Interestingly, either relaxation violates the original low-rankness by at most $\tilde m$;

\item  Applying the  diagonal matrix space to the \ref{eq_rel_rank} and \ref{eq_rel_rank1} dramatically changes their rank bounds to be linear in $\tilde m$, in contrast to those of $\sqrt{\tilde m}$ derived in \Cref{them:symsign,them:sym_rank}. This  is because an $n\times n$ diagonal  matrix maps onto a  vector of only  length $n$, whereas a general symmetric matrix equals a size-$\O(n^2)$ vector; and
\item We give an example in Appendix \ref{proof:tight} that  attains the rank bound in \Cref{them:sparsesign,them:sparsenonsign}.
%\item The rank bound in \Cref{them:sparsesign} can be directly applied to  \ref{app:sparse} and provides a sparsity bound for its \ref{eq_rel_rank}, which improves the existing one.
\end{enumerate}

As indicated previously, the  \ref{app:sparse} admits a sparse domain set in \eqref{set:sparse}, i.e.,
$$\X :=\{\bm X \in \S^n: \bm X =\Diag(\diag(\bm X)), \rank(\bm X) \le k, \|\bm X\|_F^2:=\|\diag(\bm X)\|_2^2  \le z \},$$
where the function $f(\diag(\bm X))$ in \eqref{set:sparse} is specified to be $\|\diag(\bm X)\|_2^2-z$ and  thus  is sign-invariant.
Specifically,  the \ref{eq_rel_rank} corresponding to  \ref{app:sparse} can be formulated by 
\begin{equation} \label{pc_sparse}
\begin{aligned}
\V_{\rel}:=\min_{\bm X\in \conv(\X ), \bm y \in \Re^m, z \in \Re_+} \left\{\frac{1}{n}\|\bm y\|_2^2 +\alpha z:   \bm A \diag(\bm X) = \bm b - \bm y   \right\}
\end{aligned}, \tag{Sparse Ridge Regression-R}
\end{equation}
where given the variable  $z\in \Re_+$,  $\conv(\X ):= \Big\{\bm X\in  \S^n: \bm X =\Diag(\diag(\bm X)),  \exists  \bm x \in \Re_+^n, \|\bm x\|_2^2 \le z,  x_1 \ge \cdots \ge x_n,  x_{\k+1}=0, 
\|\bm X\|_{(\ell)} \le \sum_{i\in [\ell]} x_{i}, \forall \ell \in [\k-1], \|\bm X\|_* \le \sum_{i\in [\k]}x_i\Big\}.$ 

 \begin{restatable}{corollary}{corsparse}{\rm \textbf{(Sparse Ridge Regression)}} \label{cor:sparse}
Given the data  $\bm A\in \Re^{m\times n}$ in \ref{app:sparse}, there is an optimal  solution $\bm X^*$ of \ref{pc_sparse} of rank at most  $k+\rank(\bm A)$.
\end{restatable}
\begin{proof}
    See Appendix \ref{proof:sparse}. \qed
\end{proof}

It is worth mentioning that  the diagonal elements of an optimal diagonal solution $\bm X^*$ to \ref{pc_sparse} can serve as a relaxed solution for the original \ref{app:sparse} in vector space. According to the rank bound in \Cref{cor:sparse},  the zero-norm of such relaxed solution (i.e., $\diag(\bm X^*)$) is no larger than $k+\rank(\bm A)$, smaller than the recent sparsity bound (i.e., $k+\rank(\bm A)+1$) shown in \cite{askari2022approximation}. 
%Therefore, the proposed \ref{pc_sparse} compromises the original sparsity by at most $\tilde m$.
%.Introducing binary variables to reformulate the zero-norm constraint in \ref{app:sparse},  \cite{askari2022approximation} derived a sparsity bound of $k+\tilde m+1$  for the resulting continuous relaxation by leveraging 
%whose proof relies on Shapley-Folkman theorem. 

\section{Solving LSOP-R by Column Generation and Rank-Reduction Algorithms}\label{sec:algo}
This section studies an efficient column generation algorithm for solving \ref{eq_rel_rank}, coupled with a rank-reduction algorithm. The algorithms together return an optimal solution to \ref{eq_rel_rank}  whose rank must not exceed the theoretical guarantees.
Throughout this section, we suppose that \ref{eq_rel_rank} yields a finite optimal value, i.e, $\V_{\rel}>-\infty$ and set $\conv(\X)$ contains no line.
\vspace{-0.5em}
 %  via an iterative procedure.
%the optimal solution $\bm X^* \in \hat{\C}_{\rel}$ can be represented by the convex combination below.
%\[\bm  X^* = \sum_{t\in  [\tau]} \alpha_t \bm X_t + \sum_{t \in  [\hat{\tau}]} \beta_{t} \bm Y_t, \]
%where  $\{\bm X_1, \cdots, \bm X_\tau\}$ denote extreme points of set $\hat{\C}_{\rel}$ 
%using the representation theorem  can be equivalently formulated as below
\subsection{Column Generation Algorithm}
The column generation does not require an explicit description of the convex hull of domain set $\X$, and instead use the representation theorem 18.5 in \cite{rockafellar2015convex} to describe the line-free convex set $\conv(\X)$ via the linear programming  formulation (with possibly infinite number of variables), implying an equivalent linear reformulation of \ref{eq_rel_rank} as shown below
\begin{equation} \label{eq_rel_dwr}
\begin{aligned}
\text{(LSOP-R)} \quad  \V_{\rel}:=\min_{ \bm \alpha \in \Re_+^{|\H|}, \bm \gamma \in \Re_+^{|\J|}}\bigg\{&\sum_{h\in \H}\alpha_h \langle\bm A_0, \bm X_h \rangle + \sum_{j\in \J}\gamma_j \langle\bm A_0, \bm d_j \rangle: 
 \sum_{h\in \H}\alpha_h=1,
\\& 
%\\& 
b_i^l  \le  \sum_{h\in \H}\alpha_h \langle\bm A_i, \bm X_h \rangle + \sum_{j\in \J}\gamma_j \langle\bm A_i, \bm d_j \rangle \le b_i^u, \forall i \in [m]\bigg\},  
\end{aligned}
\end{equation}
where $\{\bm X_h\}_{h\in \H}$ and $\{\bm d_j \}_{j\in \J}$ consist of all the extreme points and extreme directions in the set $\conv(\X)$, respectively. 
%It is worth mentioning that the LSOP-R \eqref{eq_rel_dwr} can be applied to the linear relaxation at each node and embedded in the branch and price algorithm. Delete it?}

It is often difficult to enumerate all the possible extreme  points $\{\bm X_h\}_{h\in H}$ and extreme directions $\{\bm d_j \}_{j\in J}$   in the set $\conv(\X)$ and directly solve the  (semi-infinite) linear programming problem \eqref{eq_rel_dwr} to optimality. %and directly solve the large-scale LSOP-R  problem \eqref{eq_rel_dwr} with $|H|+|J|$ variables.
Alternatively, the column generation algorithm starts with the \textit{restricted master problem} (RMP) which only  contains a small portion of extreme points and extreme directions and then solves the \textit{pricing problem} (PP) to iteratively generate an improving point. The detailed implementation is presented in \Cref{algo_cg}.
The column generation \Cref{algo_cg} can be effective in practice since
a small number of points in $\conv(\X)$ are often needed 
to return a (near-)optimal solution to LSOP-R  \eqref{eq_rel_dwr} according to Carath\'eodory theorem. 

As shown in \Cref{algo_cg},  at each iteration, we first solve a linear program--\ref{eq:rmp} based on a subset of extreme points $\{\bm X_h\}_{h\in {\H}}$ and a subset of extreme directions $\{\bm d_j\}_{j\in {\J}}$  collected from the set $\conv(\X)$. Using the Lagrangian multipliers $(\nu, \bm \mu^l, \bm \mu^u)$,
the dual problem of the \ref{eq:rmp}  is equal to
\begin{equation}\label{eq:dual}
\begin{aligned}
%\max_{\nu\in \Re, \bm \mu^l \in \Re_+^m, \bm \mu^u \in \Re_+^m} \bigg\{ -\nu + (\bm b^l)^{\top} \bm \mu^l - (\bm b^u)^{\top} \bm \mu^u:  \nu \ge \bigg \langle -\bm A_0 + \sum_{i\in [m]} (\mu^l_i - \mu^u_i) \bm A_i, \bm X_h \bigg\rangle, \forall h \in {\H}, \\
%0 \ge \bigg \langle -\bm A_0 + \sum_{i\in [m]} (\mu^l_i - \mu^u_i) \bm A_i, \bm d_j \bigg\rangle, \forall j \in {\J} \bigg\}\\
\max_{\bm \mu^l \in \Re_+^m, \bm \mu^u \in \Re_+^m} \bigg\{ -\max_{h\in \H} \bigg \langle -\bm A_0 + \sum_{i\in [m]} (\mu^l_i - \mu^u_i) \bm A_i, \bm X_h \bigg \rangle  + (\bm b^l)^{\top} \bm \mu^l - (\bm b^u)^{\top} \bm \mu^u: \\
0 \ge \bigg \langle -\bm A_0 + \sum_{i\in [m]} (\mu^l_i - \mu^u_i) \bm A_i, \bm d_j \bigg\rangle, \forall j \in {\J} \bigg\}.
\end{aligned}
\end{equation}
Then, given an optimal dual solution $(\nu^*, (\bm \mu^l)^*, (\bm \mu^u)^*)$ at Step 4, the  \ref{eq_pp} at Step 6 of \Cref{algo_cg} finds an improving point or direction for LSOP-R \eqref{eq_rel_dwr} in set $\conv(\X)$.

When the \ref{eq_pp} is unbounded, there exists a direction $\hat{\bm d}$ in the  set $\conv(\X)$ such that $\langle -\bm A_0 + \sum_{i\in [m]} ((\mu_i^l)^* - (\mu_i^u)^* ) \bm A_i, \hat{\bm d}\rangle>0$, which violates the constraint in the dual problem \eqref{eq:dual}. Hence, adding this direction to  the \ref{eq:rmp} leads to a smaller objective value. Besides, when the \ref{eq_pp} yields a finite optimal value no larger than $\nu^*+\epsilon$, we can show that \Cref{algo_cg} finds an $\epsilon-$optimal solution to the LSOP-R \eqref{eq_rel_dwr}, as detailed below. %the objective value of dual problem \eqref{eq:dual} becomes $-\V_{\rm P} + (\bm b^l)^{\top} (\bm \mu^l)^* - (\bm b^u)^{\top} (\bm \mu^u)^*$ after adding  the optimal solution $\hat{\bm X}$ to PP \ref{eq_pp}. That is, the improvement of \ref{eq:rmp} is at most $\V_{\rm P} -\nu^*$; and (iii) \Cref{algo_cg} terminates if there is no  improvement.
\begin{restatable}{proposition}{propconver} \label{prop:conver}
%Suppose that LSOP-R \eqref{eq_rel_dwr} yields a finite optimal value, i.e., $\V_{\rel}>-\infty$. Then
The output solution of column generation \Cref{algo_cg}  is $\epsilon$-optimal to the LSOP-R \eqref{eq_rel_dwr} if the LSOP-R \eqref{eq_rel_dwr} yields a finite optimal value, i.e., $\V_{\rel}>-\infty$. That is, suppose that \Cref{algo_cg}  returns $\bm  X^*$, then the inequalities $\V_{\rel}\leq \langle \bm A_0, \bm  X^*\rangle \le \V_{\rel} + \epsilon$ hold.
\end{restatable}
\begin{proof}
    See Appendix \ref{proof:conver}. \qed
\end{proof}

The convergence of the column generation \Cref{algo_cg} is guaranteed by   the theory of linear programming, since the \ref{eq_pp} in fact  seeks the positive reduced cost. The finite termination or the rate of convergence of this method has been established under the assumption of finiteness of extreme points and extreme directions or revising the problem setting (e.g., adding penalty term or enforcing the strict improvement of the \ref{eq:rmp}  at each iteration). For more details, we refer interested readers to the works in \cite{amor2004stabilization, amor2009choice,desrosiers2005primer}. The time complexity of  \Cref{algo_cg} lies in the number of iterations as well as in solving the  \ref{eq_pp}. Fortunately, we show that the \ref{eq_pp} can be easily solved and even admits closed-form solutions given some special domain set $\X$ in the next subsection.

%\begin{align*}
% \text{(Sparse Regression)} \quad \min_{(z, \bm X) \in \Re \times \X}\{z: b_i^l \le \langle\bm A_i,\bm X\rangle \le b_i^u, \forall i\in [m] \}
%\end{align*}
%where $\X:=\left\{\bm X\in \S^n:\rank(\bm X)\le k, \bm X = \Diag(\diag(\bm X)),  \|\bm y - \bm A \diag(\bm X)\|_2  \le z
%\right\}$

\begin{algorithm}[htbp] 
\caption{The Column Generation Algorithm for Solving LSOP-R \eqref{eq_rel_dwr} with $\V_{\rel}>-\infty$} \label{algo_cg}
\begin{algorithmic}[1]
\State \textbf{Input:} Data $\bm A_0$, $\{\bm A_i, b_i^l, b_i^m\}_{i\in [m]}$,  domain set $\X$ with a line-free convex hull, and optimality gap $\epsilon >0$
%\State Let $\bm{v}_i \in \Re^d$ denote the $i$-th column vector of matrix $\bm{V}$ for each $i \in [n]$

%\State Let $\bm{v}_i \in \Re^d$ denote the $i$-th column vector of matrix $\bm{V}$ for each $i \in [n]$
\State Initialize ${\H}=\{1\}$, ${\J}=\emptyset$, and a feasible solution $\bm X_1$ to LSOP-R \eqref{eq_rel_dwr}
%\State Initialize  the best primal solution $\bm X^* = \bm X_t$, dual solution $\nu^* = 0$, upper bound $UB = f(\bm X^*)$, and lower bound $LB = -\infty$
\For{$t = 1,2, \cdots$}
\State Solve the \ref{eq:rmp} below and denote by $(\bm \alpha^*, \bm \gamma^*)$  the optimal solution
\begin{equation}\label{eq:rmp}
\begin{aligned}  
\min_{ \bm \alpha \in \Re_+^{|{\H}|}, \bm \gamma \in \Re_+^{|{\J}|}}\bigg\{&\sum_{h\in {\H}}\alpha_h \langle\bm A_0, \bm X_h \rangle + \sum_{j\in {\J}}\gamma_j \langle\bm A_0, \bm d_j \rangle:  \sum_{h\in \H}\alpha_h=1, 
\\& 
%\\& 
b_i^l  \le  \sum_{h\in \H}\alpha_h \langle\bm A_i, \bm X_h \rangle + \sum_{j\in \J}\gamma_j \langle\bm A_i, \bm d_j \rangle \le b_i^u, \forall i \in [m]\bigg\}, 
\end{aligned} \tag{RMP}
\end{equation}
and let $(\nu^*, (\bm \mu^l)^*, (\bm \mu^u)^*)$ denote the corresponding optimal dual solutions 
\State Update solution $\bm X^* :=\sum_{h\in {\H}} \alpha^*_h \bm X_{i} + \sum_{j\in \J}  \gamma^*_j \bm d_j$

%\State Given $\nu^*$, we solve the following linear minimization oracle:
%\begin{align*}
%\bm X_{t+1} \in \argmin_{\bm X \in \X} \tr\left[\left(\nabla f(\bm X^*) + \nu^* \bm C \right)\bm X \right]
%\end{align*}
\State Solve the \ref{eq_pp} and let $\hat{\bm X}$ denote its optimal solution
\begin{align} \label{eq_pp}
	\V_{\rm P}:= \max_{\bm X\in \conv(\X)}  \bigg \langle -\bm A_0 + \sum_{i\in [m]} ((\mu_i^l)^* - (\mu_i^u)^* ) \bm A_i, \bm X \bigg\rangle.\tag{Pricing Problem}
\end{align}

\If{$\V_{\rm P} = \infty$}
\State $\hat{\bm X}$ can be expressed by $\tilde{\bm X}+ \theta \hat{\bm d}$ with $\theta>0$
\State  Add direction $ \hat{\bm d}$ to set $\{\bm d_j\}_{j\in \J}$ and update $\J:= \J\cup\{|\J|+1\}$ 
\ElsIf{$\V_{\rm P} > \nu^*+\epsilon$}
\State Add point $ \hat{\bm X}$ to set $\{\bm X_h\}_{h\in \H}$ and update  $\H:= \H\cup\{|\H|+1\}$ 
\Else
\State Terminate the iteration
\EndIf

\EndFor
\State \textbf{Output:} Solution $\bm X^*$%$(\bar{x},\bar{w})$
\end{algorithmic}
\end{algorithm}

\subsection{Efficient Pricing Oracle}
In this subsection, we show that despite not describing the convex hull of domain set $\X$,  the  \ref{eq_pp} at Step 6 of \Cref{algo_cg} is efficiently solvable and even admits the closed-form solution for some cases. 
% for a special class of  spectral sets, which benefits many applications, e.g., FPCA \eqref{eq:fpca}, FSVD \eqref{eq:fsvd}, and Matrix Completion. 
Given the linear objective function,  the \ref{eq_pp} can be simplified by replacing $\conv(\X)$ with the domain set $\X$, i.e.,
\begin{align*} \label{eq_sep1}
\V_{\rm P}:= \max_{\bm X \in \Q} \left\{\left\langle  \tilde{\bm A}, \bm X\right\rangle:  \rank(\bm X)\le k, f(\bm \lambda(\bm X)) \le 0  \right\},\tag{Pricing Problem-S}
\end{align*}
where  we define $\tilde{\bm A} :=  -\bm A_0 + \sum_{i\in [m]} ((\mu_i^l)^* - (\mu_i^u)^* ) \bm A_i$ throughout this subsection. 
%Since we state in the introduction that 

Next, let us explore the solution structure of \ref{eq_sep1}. Note that
matrices $\bm A_0$  and $\{\bm A_i\}_{i\in [m]}$ are symmetric if matrix variable $\bm X$ is symmetric, and so does matrix $\tilde{\bm A}$.

\begin{restatable}{theorem}{thmppsol} \label{them:ppsol}
The \ref{eq_sep1} has an optimal solution $\bm X^*$, as specified follows
\begin{enumerate}[(i)]
\item Given $\Q := \S_+^n$ in \eqref{eq_set}, suppose that $\tilde{\bm A} := \bm U \Diag(\bm \beta) \bm U^{\top} $ denotes the eigen-decomposition of  matrix $\tilde{\bm A}$ with eigenvalues $\beta_1 \ge \cdots \ge \beta_n$. Then $\bm X^* = \bm U  \Diag(\bm \lambda^*) \bm U^{\top}$ and  $\bm \lambda^*\in \Re_+^n$ equals
\begin{align} \label{eq:ppc1}
\bm \lambda^* \in \argmax_{\bm \lambda \in \Re_+^n} \left\{\bm \lambda^{\top} \bm \beta:  \lambda_i =0, \forall i\in [k+1, n],  f(\bm \lambda)\le 0\right\}.
\end{align}
\item Given $\Q:=\Re^{n\times p}$ in \eqref{eq_set}, suppose that $\tilde{\bm A} := \bm U \Diag(\bm \beta) \bm V^{\top}$ is the singular value decomposition of matrix $\tilde{\bm A}$ with  singular values $\beta_1 \ge \cdots \ge \beta_n$.
 Then $\bm X^* = \bm U  \Diag(\bm \lambda^*) \bm V^{\top}$ and  $\bm \lambda^*\in \Re_+^n$ equals
\begin{align} \label{eq:ppc2}
\bm \lambda^*   \in \argmax_{\bm \lambda \in \Re_+^n} \left\{\bm \lambda^{\top} \bm \beta:  \lambda_i =0, \forall i\in [k+1, n],  f(\bm \lambda)\le 0 \right\}.
\end{align}
\item Given $\Q := \S^n$ in \eqref{eq_set}, suppose that  $\tilde{\bm A} := \bm U \Diag(\bm \beta) \bm U^{\top} $ denotes the eigen-decomposition of matrix $\tilde{\bm A}$ with  eigenvalues $\beta_1 \ge \cdots \ge \beta_n$.
 Then  $\bm X^* = \bm U  \Diag(\bm \lambda^*) \bm U^{\top}$ and  $\bm \lambda^*\in \Re^n$ equals
\begin{align} \label{eq:ppc3}
\bm \lambda^*   \in \argmax_{s\in [k+1]} \max_{\bm \lambda \in \Re^n} \big\{\bm \lambda^{\top} \bm \beta:  \lambda_1 \ge \cdots \ge \lambda_n, \lambda_s =\lambda_{s+n- k-1}=0,
f(\bm \lambda)\le 0\big\}.
\end{align}
\item Given $\Q := \S^n$ and a sign-invariant  function $f(\cdot)$ in \eqref{eq_set}, suppose that $\tilde{\bm A}:=\bm U \Diag(\bm \beta) \bm U^{\top}$ denotes  the eigen-decomposition of matrix $\tilde{\bm A}$ with   eigenvalues satisfying $|\beta_1| \ge \cdots \ge |\beta_n|$. Then $\bm X^* = \bm U  \Diag(\bm \lambda^*) \bm U^{\top}$ and $\bm \lambda^*\in \Re^n$ equals 
\begin{align} \label{eq:ppc4}
\bm \lambda^*  \in
\argmax_{\bm \lambda \in \Re^n} \left\{|\bm \lambda|^{\top} |\bm \beta|:  |\lambda_i| =0, \forall i\in [k+1, n],  f(|\bm \lambda|)\le 0 \right\}.
\end{align} 
\end{enumerate}
%Suppose that $\bm A = \sum_{i\in [n]}  \beta_i \bm q_i \bm q_i^{\top}$ denotes the eigen-decomposition of matrix $\bm A \in \S^n$ where the  eigenvalues satisfy $\beta_1 \le \cdots \le  \beta_n$.
%When $\Q =\S^n$,  the pricing problem \eqref{eq_sep1} can be formulated by
%\begin{align*}
%\min_{\ell \in [k+1]} \max_{\theta \ge 0, \bm \mu\in \Re^n} \left\{-\theta h^*(-(\bm \beta^{\ell}+ \bm \mu)/\theta)-\theta t: \mu_i \le 0, \forall i\in [\ell-1],  \mu_i \ge 0, \forall i\in [\ell + n-k,n]  \right\},
%\end{align*}
%where for each $\ell \in [k+1]$, the vector $\bm \beta^{\ell}:= [ \beta_1, \cdots, \beta_{\ell-1}, 0,\cdots, 0, \beta_{n+\ell-k}, \cdots, \beta_{n}]$.
\end{restatable}
\begin{proof}
    See Appendix \ref{proof:ppsol}. \qed
\end{proof}

We make the following remarks about \Cref{them:ppsol}:
\begin{enumerate}[(a)]
\item For any domain set $\X$ in  \eqref{eq_set}, the \ref{eq_sep1} involves a rank constraint and is  thus nonconvex. \Cref{them:ppsol} provides a striking result  that the \ref{eq_sep1}  can be equivalent to solving a  convex program over a vector variable of length $n$ (see problems \eqref{eq:ppc1}-\eqref{eq:ppc4}). Therefore, the column generation \Cref{algo_cg} can admit an efficient implementation for solving \ref{eq_rel_rank}; 
\item  By leveraging \Cref{them:ppsol},  we can derive the closed-form solutions to the \ref{eq_sep1} for a special family of domain set $\X$ as shown in \Cref{cor:closed};
\item Although the  convex hull of the domain set $\X$ may admit an explicit description based on the majorization technique (see, e.g., \Cref{prop:conv}), the description involves  
%a number of 
many auxiliary variables and constraints and is semidefinite representable. Therefore,  the original \ref{eq_pp} over $\conv(\X)$ can still be challenging to optimize and scale. Quite differently, \Cref{them:ppsol} allows us to target a simple convex program for solving the \ref{eq_pp}, which significantly enhances the \Cref{algo_cg}
performance in our numerical study; and 
% \item As illustrated in \Cref{sec3},  the sign-invariance of function $f(\cdot)$ has an impact on the rank bounds of \ref{eq_rel_rank} under the symmetric case $\Q=\S^n$ (see \Cref{them:symsign,them:sym_rank}).  Likewise, this impact exists in solving the \ref{eq_sep1} and the distinction between Parts (iii) and (iv) in \Cref{them:ppsol} exactly arises from the sign-invariance; and 
\item The column generation \Cref{algo_cg} can be extended to solving \ref{eq_rel_rank1} with $\Q:=\S^n$ whose pricing oracle can be also simplified, as shown in \Cref{cor:sympp}.
\end{enumerate}

 According to \Cref{them:ppsol}, the \ref{eq_sep1} can be equivalently converted to solving the maximization problem over variable $\bm \lambda$. Therefore, it suffices to derive closed-form  solutions  to problems \eqref{eq:ppc1}-\eqref{eq:ppc4}.  For any $\ell \in [1,\infty]$, we let  $\|\cdot\|_{\ell}$ denote the $\ell$-norm of a vector.
\begin{restatable}{corollary}{corclosed}\label{cor:closed}
For a domain set $\X:=\{\bm X \in \Q: \rank(\bm X) \le k,  \|\bm \lambda(\bm X)\|_{\ell}\le c \}$, where $c\ge 0$, $\ell \in [1,\infty]$, and $1/\ell+1/q=1$, we have that
\begin{enumerate}[(i)]
    \item Given $\Q:=\S_+^n$, the solution $\bm \lambda^*$ is optimal to problem \eqref{eq:ppc1} where $\lambda^*_i = c \sqrt[\ell]{\frac{(\beta_i)_+^{q}}{\sum_{j\in [k]} (\beta_j)_+^{q}}} $ for all $i\in [k]$ and $\lambda^*_i =0$ for all $i\in [k+1, n]$;
    \item Given $\Q:=\Re^{n\times p}$, the solution  $\bm \lambda^*$ is optimal to problem \eqref{eq:ppc2} where $\lambda^*_i = c \sqrt[\ell]{\frac{\beta_i^{q}}{\sum_{j\in [k]} \beta_j^{q}}} $ for all $i\in [k]$ and $\lambda^*_i =0$ for all $i\in [k+1, n]$;
    \item Given $\Q:=\S^n$, the solution  $\bm \lambda^*$ is optimal to problem \eqref{eq:ppc4} where $\lambda^*_i = \sign(\beta_i) c \sqrt[\ell]{\frac{|\beta_i|^{q}}{\sum_{j\in [k]} |\beta_j|^{q}}} $ for all $i\in [k]$ and $\lambda^*_i =0$ for all $i\in [k+1, n]$.
\end{enumerate}
\end{restatable}
\begin{proof}
    See Appendix \ref{proof:closed}. \qed
\end{proof}

Since the \ref{eq_rel_rank1} is equivalent to solving $(\k+1)$ subproblems over sets $\{\Y^s\}_{s\in [\k+1]}$ defined in \eqref{eq_ys},  we apply column generation \Cref{algo_cg} to  solving each subproblem  and then take the best one among $(\k+1)$ output solutions. To be specific, for each $s\in [\k+1]$, the \Cref{algo_cg} generates a new column by solving 
\begin{align}\label{pp}
 \V_{\rm P}^s:= \max_{\bm X\in \Y^s}  \bigg \langle -\bm A_0 + \sum_{i\in [m]} ((\mu_i^l)^* - (\mu_i^u)^* ) \bm A_i, \bm X \bigg\rangle = \max_{\bm X\in \Y^s}  \left \langle \tilde{\bm A}, \bm X \right\rangle,
\end{align}
which replaces the feasible set $\conv(\X)$ in \ref{eq_pp} with set  $\Y^s$ in \eqref{eq_ys}. 
We close this subsection by simplifying the pricing problem \eqref{pp} as follows.
\begin{corollary}\label{cor:sympp}
Given $\Q := \S^n$ in \eqref{eq_set}, suppose that  $\tilde{\bm A} := \bm U \Diag(\bm \beta) \bm U^{\top} $ denotes the eigen-decomposition of  matrix $\tilde{\bm A}$ with  eigenvalues $\beta_1 \ge \cdots \ge \beta_n$.
 Then for each $s\in [\k+1]$, the pricing problem \eqref{pp} has an optimal solution $\bm X^* = \bm U  \Diag(\bm \lambda^*) \bm U^{\top}$, where  $\bm \lambda^*\in \Re^n$ equals
\begin{align} 
\bm \lambda^* :=\max_{\bm \lambda \in \Re^n} \big\{\bm \lambda^{\top} \bm \beta:  \lambda_1 \ge \cdots \ge \lambda_n, \lambda_s =\lambda_{s+n- k-1}=0,
f(\bm \lambda)\le 0\big\}.
\end{align}
\end{corollary}
\begin{proof}
The proof is identical to that of Part (iii) in \Cref{them:ppsol} and thus is omitted. \qed
\end{proof}
 
\subsection{Rank-Reduction Algorithm}
The column generation \Cref{algo_cg} may not always be able to output a solution that satisfies the theoretical rank bounds. To resolve this, this subsection designs a rank-reduction  \Cref{algo_reduce} to find an alternative solution of the same or smaller objective value while satisfying the desired rank bounds. %for closing the theory-practice gap which can iteratively reduce the rank of any solution returned by column generation \Cref{algo_cg}, whenever the solution violates the rank bounds of \ref{eq_rel_rank}. 

Given a (near-)optimal solution $\bm X^*$ of \ref{eq_rel_rank} returned by \Cref{algo_cg} that violates the proposed rank bound, our rank-reduction  \Cref{algo_reduce} runs as follows: (i) since $\bm X^*$ is not an extreme point of \ref{eq_rel_rank}, then we can find a direction $\bm Y$ in its feasible set, along which the objective value will decrease or stay the same; (ii) we move $\bm X^*$ along the direction $\bm Y$ until a point on the boundary of the feasible set of \ref{eq_rel_rank}; (iii) we update $\bm X^*$ to be the new boundary point found; and (iv) finally, we terminate the iteration when no further movement is available, i.e., $\bm X^*$ is a feasible extreme point in \ref{eq_rel_rank} that must satisfy the rank bound. 
Hence, the rank-reduction procedure in \Cref{algo_reduce} can be viewed as searching for an extreme point. 
%The detailed implementation can be found in \Cref{algo_reduce}. 
%, including the extension to \ref{eq_rel_rank1} under $\Q:=\S^n$.

Suppose that $\bm \lambda^* \in \Re^n$ denotes eigen-(singular-)value vector of matrix $\bm X^*$ in the descending order. Next, we show how to construct  matrix $\bm Y$ and vector $\bm x^*$ with different matrix spaces in the rank-reduction \Cref{algo_reduce} (i.e., Step 5 and Step 4), which represents moving direction and determines the moving distance, respectively. 
%First, 
%the vector $\bm x^*$ at Step 8 stems from the description of $\conv(\X)$. To be specific, 
%{Given $\bm X^*\in\conv(\X)$ with eigenvalues $\bm \lambda^*$}, 

When $\Q:=\S_+^n$, 
according to the convex hull description in \Cref{prop:conv}, there exists an optimal solution $\bm x^*$ at Step 4 such that $f(\bm x^*) \le 0$.
%the following set with respect to $\bm x$ is nonempty
%\[\bigg\{\bm x \in \Re_+^n:  f(\bm x) \le 0, x_1\ge \cdots \ge x_n, x_{k+1} =0,  \sum_{i\in [\ell]} x_i \ge \sum_{i\in [\ell]}  \lambda_i^*, \forall \ell \in [\k-1], \sum_{i\in [\k]} x_i= \sum_{i\in [n]}  \lambda_i^*   \bigg\},  \]
%where the last $\k$ linear constraints are equivalent to $\bm x \succeq \bm \lambda^*$.  
When $\Q:=\Re^{n\times p}$,  the convex hull $\conv(\X)$ described in \Cref{prop:nonsymconv} also ensures us a solution $\bm x^*$.
%\[\bigg\{\bm x \in \Re_+^n:  f(\bm x) \le 0, x_1\ge \cdots \ge x_n, x_{k+1} =0,  \sum_{i\in [\ell]} x_i \ge \sum_{i\in [\ell]}  \lambda_i^*, \forall \ell \in [\k-1], \sum_{i\in [\k]} x_i\le \sum_{i\in [n]}  \lambda_i^*   \bigg\}. \]
When $\Q:=\S^n$, we note that  the solution $\bm X^*$ is feasible to \ref{eq_rel_rank1} and thus belongs to some set $\Y^s$ with $s\in [\k+1]$, which, according to  \Cref{prop:symconv},
%. Hence, according to  \Cref{prop:symconv}, there exists an $s\in [\k+1]$ such that the following set is nonempty
%$$\bigg\{\bm x \in \Re^n: f(\bm x) \le 0, x_1 \ge \cdots \ge x_n, x_s=x_{s+n-\k-1}=0, \sum_{i\in [\ell]} x_i \ge \sum_{i\in [\ell]}  \lambda_i^*, \forall \ell \in [n-1], \sum_{i\in [n]} x_i= \sum_{i\in [n]}  \lambda_i^*\bigg\},$$
allows us to compute a pair of solutions $(\bm x^*, s^*)$. 
%The nonemptiness of above sets indicates that at Step 4 of \Cref{algo_reduce}, the result $f(\bm x^*) \le 0$ always holds at optimality.
%Besides, to find an optimal solution $\bm x^*$, 

Next, we find the matrix $\bm Y$ at Step 5 based on the proofs of rank bounds. Suppose that the near-optimal solution $\bm X^*$ has rank $r$. When $\Q:=\S_+^n$,  we define 
\begin{align}\label{eq_direct1}
    \bm Y:=\bm Q_2 \bm \Delta \bm Q_2^{\top}, \ \ \exists \bm \Delta \in \S^{r-\k+1},  \langle  \bm Q_2^{\top} \bm A_i \bm Q_2, \bm \Delta \rangle = 0, \forall i \in [m],   \tr(\bm \Delta)  =0,
\end{align}
where matrix $\bm Q_2 \in \Re^{n\times r-\k+1}$ consists of eigenvectors corresponding to the $(r-\k+1)$ least nonzero  eigenvalues (i.e., $\lambda_{\k}^*\cdots,\lambda^*_r$) of matrix $\bm X^*$. When $\Q:=\Re^{n\times p}$, we define 
\begin{align}\label{eq_direct2}
\bm Y:=\bm Q_2 \bm \Delta \bm P_2^{\top}, \ \ \exists \bm \Delta \in \S^{r-\k+1},  \langle  \bm P_2^{\top} \bm A_i \bm Q_2, \bm \Delta \rangle = 0, \forall i \in [m],   \tr( \bm \Delta )  =0,
\end{align}
where matrices $\bm Q_2 \in \Re^{n\times r-\k+1}$, $\bm P_2 \in \Re^{p\times r-\k+1}$ consist of left and right singular vectors corresponding to the $(r-\k+1)$ least nonzero singular values of matrix $\bm X^*$. 
When $\Q:=\S^n$,  we considers the positive and negative eigenvalues in matrix $\bm X^*$, respectively, and suppose $\lambda^*_1  \ge \cdots \ge  \lambda^*_{d_1-1}>0= \lambda^*_{d_1}=\cdots = \lambda^*_{d_2-1} =0 > \lambda^*_{d_2}\ge \cdots\ge \lambda^*_{n}$. Then, inspired by the proof of \Cref{them:symface}, let us
 define a block matrix $\bm \Delta$ below corresponding to positive and negative eigenvalues, respectively,
\begin{equation}
    \begin{aligned}\label{eq_direct3}
&\bm Y:=\begin{pmatrix}
\bm Q_1^2  & \bm Q_3^1 
\end{pmatrix}  \bm \Delta \begin{pmatrix}
\bm Q_1^2  & \bm Q_3^1
\end{pmatrix}^{\top}, \ \ \exists  \bm \Delta \in \S^{r-\k+2}:= \begin{pmatrix}
\bm \Delta_1 \in \S^{d_1-s^*+1} &\bm 0\\
\bm 0 & \bm \Delta_3 \in \S^{s^*+n-\k -d_2+1}
\end{pmatrix}, \\
&\langle  \begin{pmatrix}
\bm Q_1^2  & \bm Q_3^1
\end{pmatrix}^{\top} \bm A_i \begin{pmatrix}
\bm Q_1^2  & \bm Q_3^1
\end{pmatrix}, \bm \Delta \rangle = 0, \forall i \in [m], \tr(\bm \Delta_1) =0, \tr(\bm \Delta_3)=0,
\end{aligned}
\end{equation}
where given solution $s^*$ at Step 4 of \Cref{algo_reduce},  eigenvector matrices $\bm Q_1^2, \bm Q_3^1$ correspond to the positive eigenvalues $(\lambda^*_{s^*-1}, \cdots, \lambda^*_{d_1-1})$ and negative eigenvalues $(\lambda^*_{d_2}, \cdots,\lambda^*_{s^*+n-\k})$, respectively.
% \yongchun{It is noting that
%%According to \eqref{eq_direct1}--\eqref{eq_direct3},
% the nonzero matrix $\bm \Delta$ in \eqref{eq_direct1}--\eqref{eq_direct3} that resides in the null space of linear equations can be easily obtained.} 

We let $\bm Y:=-\bm Y$ if $\langle\bm A_0, \bm Y\rangle>0$.
Given a direction $\bm Y$, we can find the largest moving distance $\delta^*$ at Step 9 of \Cref{algo_reduce}.  To be specific, following the  proofs in \Cref{prop:rank,them:nonface,them:symface}, we let $\bm\Lambda_2:=\Diag(\lambda^*_{\k},\cdots, \lambda^*_r)$  when $\Q:=\S_+^n$ or $\Q:=\Re^{n\times p}$ and let $\bm \Lambda_1^2  \in \S_{++}^{d_1-s^*+1} := \Diag(\lambda^*_{s^*-1}, \cdots, \lambda^*_{d_1-1})$, $\bm \Lambda_3^1 \in- \S_{++}^{s^*+n-\k-d_2+1}:= \Diag(\lambda^*_{d_2}, \cdots, \lambda^*_{s^*+n-\k})$  when $\Q:=\S^n$.

Following the searching procedure above, we show that the \Cref{algo_reduce} can output an alternative  solution satisfying the  rank bounds as summarized  below.
Besides, we show that \Cref{algo_reduce} always terminates and can strictly reduce the rank at each iteration for a special case in \Cref{cor:converge}. More specifically, at each iteration of  \Cref{algo_reduce}, the convex program at Step 4 can be efficiently solved by commercial solvers or first-order methods; we can also readily find the direction $\bm Y$ in \eqref{eq_direct1}--\eqref{eq_direct3}  using the matrix $\bm \Delta$  in the null space of linear equations; and the moving distance $\delta^*$ at Step 9 requires solving a simple univariate convex optimization problem.
 In our numerical study, the \Cref{algo_reduce} efficiently reduces the rank  and converges fast.

\begin{algorithm}[ht] 
\caption{Rank Reduction  for \ref{eq_rel_rank} or \ref{eq_rel_rank1}  with $\V_{\rel}>-\infty$, $\V_{\rel\text{-\rm I}}>-\infty$} \label{algo_reduce}
\begin{algorithmic}[1]
\State \textbf{Input:}  Data $\bm A_0$, $\{\bm A_i, b_i^l, b_i^m\}_{i\in [m]}$,  domain set $\X$ with a line-free convex hull, matrix space $\Q\in \{\S_+^n, \Re^{n\times p}, \S^n\}$, integer $\k$ being defined in \Cref{def:k}, and scalar $\delta^*:=1$

%\If{$\Q:=\S_+^n$ or $\Q:=\Re^{n\times p}$ }
    \State Initialize an $\epsilon$-optimal solution $\bm X^*$ of \ref{eq_rel_rank} or \ref{eq_rel_rank1} returned by \Cref{algo_cg}
 %   \Else
%\State Initialize an $\epsilon$-optimal solution $\bm X^*$ of \ref{eq_rel_rank1} returned by  \Cref{algo_cg}
%    \EndIf

\While{$\delta^* > 0$}
\State Let $\bm \lambda^* \in \Re^n$ denote eigenvalues  of matrix $\bm X^*$  in descending order  if $\bm X^*\in \Q$ is symmetric, or denote singular values  otherwise. Then, compute
\begin{align*}
 \begin{cases}
      \bm x^* \in  \argmin_{\bm x \in   \Re_+^n}
\{f(\bm x):  x_1 \ge \cdots \ge x_n,  x_{k+1}=0, \bm x \succeq \bm \lambda^* \}, & \text{if } \Q:=\S_+^n; \\
        \bm x^* \in \argmin_{\bm x \in   \Re_+^n}
\{f(\bm x):  x_1 \ge \cdots \ge x_n,  x_{k+1}=0, \bm x \succ \bm \lambda^* \}, & \text{if } \Q:=\Re^{n\times p}; \\
 ( \bm x^*, s^*) \in \argmin_{\bm x \in   \Re^n, s\in [\k+1]}
\{f(\bm x):  x_1 \ge \cdots \ge x_n,  x_{s}=x_{s+n-\k-1}=0, \bm x \succeq \bm \lambda^* \}, & \text{if } \Q:=\S^n.
  \end{cases}  
\end{align*}

\State Find a nonzero  matrix $\bm Y$
\begin{align*}
  \bm  Y:= \begin{cases}
      \bm Q_2 \bm \Delta \bm Q_2^{\top} \text{ as defined in \eqref{eq_direct1},} & \text{if } \Q:=\S_+^n;  \\
     \bm Q_2 \bm \Delta \bm P_2^{\top}\text{ as defined in \eqref{eq_direct2},} & \text{if } \Q:=\Re^{n\times p};\\
\begin{pmatrix}
\bm Q_1^2  & \bm Q_3^1 
\end{pmatrix}  \bm \Delta \begin{pmatrix}
\bm Q_1^2  & \bm Q_3^1
\end{pmatrix}^{\top}\text{ as defined in \eqref{eq_direct3},} & \text{if } \Q:=\S^n. 
  \end{cases} 
\end{align*}

\If{$\bm Y$ does not exist }
\State Set $\delta^*=0$ and break the loop
\Else
\State Let $\bm Y := -\sign\left(\langle\bm A_0, \bm Y\rangle\right) \bm Y$ and  $\bm X(\delta):= \bm X^*+\delta \bm Y$, and  compute $  \delta^*$ by
$$
 \begin{cases}
      \argmax_{\delta \ge 0} \left\{\delta: \bm x^* \succeq \bm \lambda (\bm X(\delta)),    
\lambda_{\min}( \bm  \Lambda_2+\delta \bm \Delta)\ge 0  \right\}, & \text{if } \Q:=\S_+^n;  \\ 
\argmax_{\delta \ge 0} \left\{\delta:  \bm x^* \succ \bm \lambda (\bm X(\delta)), 
\lambda_{\min}( \bm  \Lambda_2+\delta \bm \Delta)\ge 0 \right\}, & \text{if } \Q:=\Re^{n\times p}; \\  
\argmax_{\delta \ge 0} \left\{\delta: \bm x^* \succeq \bm \lambda(\bm X(\delta)),  \lambda_{\min}(\bm \Lambda_{1}^2 + \delta \bm \Delta_1)\ge 0, \lambda_{\max}(\bm \Lambda_{3}^1 + \delta \bm \Delta_3)\le 0  \right\}, & \text{if } \Q:=\S^{n}. 
   \end{cases}
$$
\State Update $\bm X^*:=\bm X(\delta^*)$ 

\EndIf

\EndWhile
\State \textbf{Output:} Solution $\bm X^*$%$(\bar{x},\bar{w})$
\end{algorithmic}
\end{algorithm}

\vspace{-0.4em}
\begin{restatable}{theorem}{thmreduce}\label{them:reduce}
For the  rank-reduction \Cref{algo_reduce}, the following statements must hold: \par
\noindent (i)
 \Cref{algo_reduce} always terminates; and \par\noindent (ii) Let $\bm X^*$ denote the output solution of \Cref{algo_reduce}. Then $\bm X^*$ is $\epsilon$-optimal to either \ref{eq_rel_rank} or \ref{eq_rel_rank1}, i.e., 
 \begin{align*}
 \begin{cases}
   \V_{\rel}\le  \langle\bm A_0, \bm X^*\rangle \le  \V_{\rel} +\epsilon  & \text{if } \Q:=\S_+^n \text{ or } \Q:= \Re^{n\times p};\\
\V_{\rel\text{-\rm I}}  \le  \langle\bm A_0, \bm X^*\rangle\le \V_{\rel\text{-\rm I}} +\epsilon  & \text{if } \Q:=\S^n.
 \end{cases},
 \end{align*}
 and the rank of solution $\bm X^*$ satisfies
 \begin{align*}
     \rank(\bm X^*) \le \begin{cases}
         \tilde k + \lfloor \sqrt{2\tilde m+9/4}-3/2\rfloor, & \text{if } \Q:=\S_+^n \text{ or } \Q:=\Re^{n\times p};\\
     \tilde k + \lfloor \sqrt{4\tilde m+9}\rfloor-3, & \text{if } \Q:=\S^n.
     \end{cases},
 \end{align*}
 where integer $\k$ is being defined in \Cref{def:k}.
\end{restatable}
\begin{proof}
    See Appendix \ref{proof:reduce}. \qed
\end{proof}

\begin{restatable}{corollary}{corconverge}\label{cor:converge}
For a domain set $\X$ with $\k=1$ being defined in \Cref{def:k}, the rank-reduction \Cref{algo_reduce} reduces the solution's rank by at least one at each iteration.
\end{restatable}
\begin{proof}
	See Appendix \ref{proof:converge}. \qed
\end{proof}

\section{Numerical Study}\label{sec:num}
In this section, we numerically test the proposed rank bounds,  column generation \Cref{algo_cg}, and rank-reduction \Cref{algo_reduce}   on three \ref{eq_rank} examples: 
 multiple-input and multiple-output (MIMO) radio communication network, optimal power flow, and \ref{app:lmc}. All the experiments are conducted in Python 3.6 with calls to Gurobi 9.5.2 and MOSEK 10.0.29 on a PC with 10-core CPU, 16-core GPU, and 16GB of memory. All the times reported are wall-clock times.

\subsection{LSOP  with Trace and Log-Det Spectral Constraints: MIMO Radio Network 
} 
In recent years, there has been an increasing interest in a special class of \ref{eq_rank} with trace and log-det spectral constraints in the communication and signal processing field. Specifically, the objective is to find a low-rank  transmit covariance matrix $\bm X \in \S_+^n$, subject to the trace and log-det constraints, where the former limits the  transmit power and the latter is often used for entropy and mutual information requirements (see, e.g., \citealt{yu2010rank} and references therein). Formally, we consider the \ref{eq_rank}  built on  a special domain set $\X $ below:
\begin{align*}
\X :=\{\bm X\in \S_+^n: \rank(\bm X)\le k, \tr(\bm X) \le U, \log\det(\bm I_n + \bm X) \ge L\},
\end{align*}
where $U, L\ge 0$ are pre-specified constants and here, the spectral function in \eqref{eq_set} is defined as $F(\bm X):= \max\{\tr(\bm X)-U, L-\log\det(\bm I_n + \bm X)\}$.

Since the domain set $\X $ above is defined in the positive semidefinite matrix space $\Q:=\S_+^n$, using the convex hull description in \Cref{prop:conv}, we can explicitly describe  the  $\conv(\X )$.
Then, for this specific \ref{eq_rank}, its 
\ref{eq_rel_rank} can be formulated as
\begin{align}\label{eq_lsopr1}
 \V_{\rel} := \min_{{\bm X \in \conv(\X )}}\left\{\langle\bm A_0, \bm X \rangle: 
b_i^l \le  \langle \bm A_i, \bm X\rangle \le b_i^u, \forall i \in [m] \right\}, %\tag{LSOP-$\text{R}_1$}
\end{align}
where  
$
\conv(\X ):= \{\bm X \in \Q: \exists \bm x \in   \Re_+^n,   \sum_{i\in [n]}x_i \le U, \sum_{i\in [n]}\log(1+ x_i)\ge L,  x_1 \ge \cdots \ge x_n,  x_{k+1}=0,
\|\bm X\|_{(\ell)} \le \sum_{i\in [\ell]} x_{i}, \forall \ell \in [k], \tr(\bm X) = \sum_{i\in [k]}x_i\}.  
$

The \ref{eq_rel_rank} \eqref{eq_lsopr1} can be recast as a semidefinite program and  be directly solved by off-the-shelf solvers like MOSEK since the set $\conv(\X )$ above is semidefinite representable. However, the semidefinite program is  known to be hard to scale. 
An efficient approach to solving the \ref{eq_rel_rank} \eqref{eq_lsopr1} is through our  column generation \Cref{algo_cg}, which at each iteration, solves a linear program \ref{eq:rmp} and a vector-based \ref{eq_pp} \ref{eq:ppc1} based on \Cref{them:ppsol}, as defined below
\begin{align}\label{eq_pp_eg1}
\max_{\bm \lambda \in \Re_+^n} \bigg\{\bm \lambda^{\top} \bm \beta: \lambda_i=0, \forall i \in [k+1, n],  \sum_{i\in [k]}  \lambda_i \le U, \sum_{i\in [k]} \log(1+\lambda_i) \ge L \bigg\}.
\end{align}
Alternatively, without using \Cref{them:ppsol}, the original \ref{eq_pp}  can be solved as a semidefinite program by  plugging the set $\conv(\X )$ above when generating a new column. This column generation procedure is termed as the ``naive column generation algorithm." %That is, the proposed \Cref{algo_cg} is superior to the naive column generation in the acceleration  of \ref{eq_pp} \eqref{eq_pp_eg1}.
In the following, we numerically test the three methods available for solving the \ref{eq_rel_rank} \eqref{eq_lsopr1} on synthetic data in \Cref{table1}, which demonstrates the efficiency of our \Cref{algo_cg}. 

For the \ref{eq_rel_rank} \eqref{eq_lsopr1}, we consider a single testing case by fixing the parameters $n, m, k$ and generating random data $L,U$, $\bm A_0\in \S^n$, and $\{\bm A_i, b_i^l, b_i^u \}_{i\in [m]} \in \S^n \times \Re \times \Re$. Particularly, we first generate a rank-one matrix $\bm X_0:= \frac{1}{\|\bm x\|_2} \bm x \bm x^{\top}$, where the vector $\bm x \in \Re^n$ follows the standard normal distribution. Then, we  define the constants $U, L$ to be $U:=\tr(\bm X_0) + \sigma_1$ and $L:= \log\det(\bm I_n+\bm X_0) - \sigma_2$ with random numbers $\sigma_1, \sigma_2$  uniform in $[0,1]$, which  ensures the inclusion $\bm X_0 \in \X $. Next, we generate matrices $\{\bm A_i\}_{i\in \{0\}\cup [m]}\in \Re^{n\times n}$  from standard normal distribution and symmetrize them by $\bm A_i := (\bm A_i+\bm A_i^{\top})/2$ for all $i\in \{0\}\cup[m]$. 
Finally, for each $i\in [m]$, we let $b_i^l:=\langle\bm A_i, \bm X_0\rangle- \theta_i^1$ and $b_i^u:=\langle\bm A_i, \bm X_0\rangle + \theta_i^2$  with the random constants $ \theta_i^1, \theta_i^2$ uniform in $[0,1]$.

In \Cref{table1},  we compare the computational time (in seconds) and the rank of output solution of the \ref{eq_rel_rank} \eqref{eq_lsopr1} returned by three methods: MOSEK, column generation  with
and without the acceleration of \ref{eq_pp} \eqref{eq_pp_eg1}. It should be mentioned in \Cref{table1} and all the following tables that we let ``CG" denote the column generation and mark 	``--" in a unsolved case reaching the one-hour limit; for any output solution, we count the number of its eigenvalues or  singular values greater than $10^{-10}$ as its rank; and  we set the optimality gap of column generation algorithm to be $\epsilon:=10^{-4}$ such that the output value is no larger $\epsilon$ than $\V_{\rel}$. It is seen that our proposed column generation  \Cref{algo_cg} performs very well in computational efficiency and solution quality, dominating
the performance of commercial solver MOSEK and naive column generation in both respects. In particular, the naive column generation solves two testing cases within one hour limit. With the efficient  \ref{eq_pp} \eqref{eq_pp_eg1} used in the proposed \Cref{algo_cg}, all the cases are solved within two minutes. It is interesting to see that the solution returned by \Cref{algo_cg} always satisfies the rank-$k$ constraint. That is,  for these instances, CG \Cref{algo_cg}  finds an optimal solution to the original \ref{eq_rank}. %the equivalence of  the relaxation \ref{eq_rel_rank} \eqref{eq_lsopr1} and  original \ref{eq_rank}. 
%that the proposed column generation method solves all the instances to optimality. 
Plugging the output solution into the rank-reduction \Cref{algo_reduce}, the rank can be further reduced by one or two within two seconds, implying that the integer $\k$, albeit unknown, should be much smaller $k$ for these instances.  
By contrast,  MOSEK, based on the interior point method, tends to yield a high-rank solution. Finally, we observe that the theoretical rank bound based on the integer $k$, as presented in the last column, is consistently higher than that of the solution found by the  \Cref{algo_cg} or \Cref{algo_reduce}. This suggests that the theoretical bound may not be tight for these instances.

\begin{table}[htb]
\centering
\caption{Numerical performance of three approaches to solving LSOP-R \eqref{eq_lsopr1}}
%  \fontsize{7}{10}\selectfont
\setlength{\tabcolsep}{3pt}\renewcommand{\arraystretch}{1}
\begin{tabular}{c c c|r r|r  r|r  r| r r| r}
\hline 
\multirow{3}{*}{$n$} & \multirow{3}{*}{$m$} & \multirow{3}{*}{$k$} & \multicolumn{2}{c|}{\multirow{2}{*}{MOSEK}} & \multicolumn{2}{c|}{\multirow{2}{*}{Naive CG \tnote{i}}} & \multicolumn{4}{c|}{Proposed Algorithms} & \multicolumn{1}{c}{\multirow{2}{*}{\begin{tabular}[c]{@{}c@{}}Rank bound\end{tabular}}} \\ \cline{8-11}
 &  &  & \multicolumn{2}{c|}{} & \multicolumn{2}{c|}{} & \multicolumn{2}{c|}{CG \Cref{algo_cg}} & \multicolumn{2}{c|}{\Cref{algo_reduce}} & \multicolumn{1}{c}{} \\ \cline{4-12} 
 &  &  & \multicolumn{1}{c}{rank} & \multicolumn{1}{c|}{time(s)} & \multicolumn{1}{c}{rank} & time(s) & rank & \multicolumn{1}{r|}{time(s)} & reduced rank & time(s) & (\Cref{them:rank}) \\ \hline
50 &  5 & 5 & 48 & 17 & 3 &223 & 3 &1 & 2 & 1&7\\
50 &  10 & 5 & 29  & 19 &  5& 1261 & 5 &1 & 3 & 1 & 8\\
50 &  10 & 10  & 32& 183 &--\tnote{ii} &-- & 5 & 1 & 3 & 1 & 13\\
100 &  10 & 10  &-- &-- &-- &-- & 2 &2 & 1 & 1 &13\\
100 &  15 & 10  &-- &-- &-- &-- & 5 &2 & 3 & 1 & 14\\
100 &  15 & 15  &-- &-- &-- &-- & 5 &3 &3 & 1 & 19\\
200 &  15 & 15  &-- &-- &-- &-- & 4 & 3 & 2 & 1 & 19\\
200 &  20 & 15  &-- &-- &-- &-- & 7 & 9 & 5 & 1 & 20 \\
200 &  25 & 25 &-- &-- &-- &-- & 11 & 21 & 9 & 1 & 30\\
500 &  25 & 25  &-- &-- &-- &-- & 10 &24 & 8 & 2& 30\\
500 &  50 & 25  &-- &-- &-- &-- & 21 &99 & 20 & 2& 33\\
500 &  50 & 50 &-- &-- &-- &-- & 22 &104 & 20 & 2& 58\\
\hline
\end{tabular}%
\label{table1}
\end{table}

\subsection{LSOP with Trace Constraints: Optimal Power Flow}
This subsection numerically tests a special \ref{app:qcqp} widely adopted in the optimal power flow (OPF) problem 
\citep{bedoya2019qcqp,eltved2022strengthened}
\begin{align}\label{eq_opf}
\min_{\bm X\in \X}\left\{ \langle \bm A_0, \bm X\rangle:  \langle \bm A_i, \bm X \rangle \le b_i^u,\forall i \in [m] \right\}, \ \
\X:=\{\bm X\in \S_+^{n}: \rank(\bm X) \le 1, L \le \tr(\bm X) \le U\},\tag{OPF}
\end{align}
%which is an \ref{eq_rank} example with the specified  domain set $\X$ above. 
where the spectral function in \eqref{eq_set} is $F(\bm X):= \max\{\tr(\bm X)-U, L-\tr(\bm X) \}$. % for the sake of bounded optimal solutions. 
%Different from the set $\conv(\X )$ used in \ref{eq_rel_rank} \eqref{eq_lsopr1},
The \ref{eq_rel_rank} corresponding to \ref{eq_opf} can be formulated as
\begin{equation}\label{eq_lsopr2}
\begin{aligned}
& \V_{\rel} := \min_{{\bm X \in \conv(\X )}}\left\{\langle\bm A_0, \bm X \rangle: 
 \langle \bm A_i, \bm X\rangle \le b_i^u, \forall i \in [m] \right\}, \\
 &\ \conv(\X) := \{\bm X\in \S_+^{n}:  L \le \tr(\bm X) \le U\}, 
 \end{aligned}
 \tag{OPF-R}
\end{equation}
and consequently, the  \ref{eq_pp} in \Cref{algo_cg}  is equivalent to
\begin{align*} 
 \max_{\bm X\in \conv(\X)}  \bigg \langle -\bm A_0 + \sum_{i\in [m]} ((\mu_i^l)^* - (\mu_i^u)^* ) \bm A_i, \bm X \bigg\rangle, \ \ \conv(\X) := \{\bm X\in \S_+^{n}:  L \le \tr(\bm X) \le U\}.
\end{align*}

The  \ref{eq_pp} above, according to Part (i) of \Cref{them:ppsol},   is equivalent to solving problem \eqref{eq:ppc1} which now admits a closed-form solution; that is, if $ -\bm A_0 + \sum_{i\in [m]} ((\mu_i^l)^* - (\mu_i^u)^* ) \bm A_i= \bm U \Diag(\bm \beta) \bm U^{\top} $ with eigenvalues $\beta_1 \ge \cdots \ge \beta_n$, then an optimal solution of the pricing problem is $\bm X^* = \bm U  \Diag(\bm \lambda^*) \bm U^{\top}$, where  $
\lambda^*_1 = 
    U 
$ if $\beta_1 \ge 0$, and $L$, otherwise and $ \lambda^*_i =0
$ for all $i\in [2,n]$.

Following the synthetic data generation detailed in the previous subsection, we compare the performance of our column generation \Cref{algo_cg} with MOSEK and naive column generation when solving  \ref{eq_lsopr2}. The computational results are displayed in \Cref{table:qcqp}. We can see that
despite with a simple description of set $\conv(\X)$, the MOSEK and naive column generation are difficult to scale under the curse of semidefinite program. Our column generation \Cref{algo_cg} can efficiently return a low-rank solution for \ref{eq_lsopr2} that satisfies the theoretical rank bound in \Cref{them:rank}; however, the output rank is larger than one. We observe that the output solution of \Cref{algo_cg}   may not attain all $ m$ linear equations in practice, and the rank bound can be strengthened using only the binding constraints.
Thus, we tailor the rank-reduction \Cref{algo_reduce} for \ref{eq_lsopr2} with $\k=k=1$, where we let $\bm x^*:=(\sum_{i\in [n]} \lambda_i^*, 0, \cdots,0)$ at Step 4 in \Cref{algo_reduce} and  compute matrix $\bm Y$ at Step 5 based on the binding constraints,
which
successfully reduces the output rank of \Cref{algo_cg}, as presented in \Cref{table:qcqp}. In fact, the rank-reduction procedure can slightly improve the output  solution from \Cref{algo_cg} due to the objective function decreasing  along the direction  $\bm Y$.  

\begin{table}[htb]
\centering
\caption{Numerical performance of three approaches to solving \ref{eq_lsopr2}}
\setlength{\tabcolsep}{2.7pt}\renewcommand{\arraystretch}{1}
\begin{tabular}{ccc|rr|rr|rrrr|r}
\hline
\multirow{3}{*}{$n$} & \multirow{3}{*}{$m$} & \multirow{3}{*}{$k$} & \multicolumn{2}{c|}{\multirow{2}{*}{MOSEK}} & \multicolumn{2}{c|}{\multirow{2}{*}{Naive CG \tnote{i}}} & \multicolumn{4}{c|}{Proposed Algorithms} & \multicolumn{1}{c}{\multirow{2}{*}{\begin{tabular}[c]{@{}c@{}}Rank bound\end{tabular}}} \\ \cline{8-11}
 &  &  & \multicolumn{2}{c|}{} & \multicolumn{2}{c|}{} & \multicolumn{2}{c|}{CG \Cref{algo_cg}} & \multicolumn{2}{c|}{\Cref{algo_reduce}} & \multicolumn{1}{c}{} \\ \cline{4-12} 
 &  &  & \multicolumn{1}{c}{rank} & \multicolumn{1}{c|}{time(s)} & \multicolumn{1}{c}{rank} & time(s) & rank & \multicolumn{1}{r|}{time(s)} & reduced rank & time(s) & (\Cref{them:rank}) \\ \hline
1000 & 50 & 1 & 28 & 160 & --\tnote{ii} & -- & 3 & \multicolumn{1}{r|}{42} & 2 & 3 & 9 \\
1000 & 60 & 1 & 32 & 195 & -- & -- & 5 & \multicolumn{1}{r|}{80} & 3 & 10 & 10 \\
1500 & 60 & 1 & 27 & 642 & -- & -- & 3 & \multicolumn{1}{r|}{113} & 2 & 11 & 10 \\
1500 & 75 & 1 & 186 & 724 & -- & -- & 6 & \multicolumn{1}{r|}{344} & 4 & 35 & 11 \\
2000 & 75 & 1 & 40 & 1850 & -- & -- & 5 & \multicolumn{1}{r|}{594} & 3 & 67 & 11 \\
2000 & 90 & 1 & 12 & 2236 & -- & -- & 4 & \multicolumn{1}{r|}{483} & 2 & 27 & 13 \\
2500 & 90 & 1 & -- & -- & -- & -- & 5 & \multicolumn{1}{r|}{1323} & 3 & 122 & 13 \\
2500 & 100 & 1 & -- & -- & -- & -- & 4 & \multicolumn{1}{r|}{1326} & 2 & 114 & 13 \\ \hline
\end{tabular}
% \begin{tabular}{c c c|r r|r  r|r r | r r| r}
% \hline 
% \multirow{2}{1em}{$n$}&  \multirow{2}{1em}{$m$}&  \multirow{2}{1em}{$k$}
% & \multicolumn{2}{c|}{MOSEK} & \multicolumn{2}{c|}{Naive CG \tnote{i}} & \multicolumn{2}{c|}{ CG \Cref{algo_cg}} & \multicolumn{2}{c|}{ \Cref{algo_reduce}} &  Rank bound
% \\ \cline{4-12}
% & &
% &  \multicolumn{1}{c}{rank}  & \multicolumn{1}{c|}{time(s)}   &\multicolumn{1}{c}{rank} &time(s)   &rank  & time(s) &reduced rank &time(s) & (\Cref{them:rank}) \\
% \hline
% 1000 &  50 & 1 & 28 & 160 &  --\tnote{ii} & -- &  3& 42&  2& 3& 9\\
% 1000 &  60 & 1 & 32 & 195 &  -- & -- &  5& 80& 3 & 10 &10\\
% 1500 &  60 & 1 &27 & 642 &  --& -- &  3& 113&  2 & 11 & 10\\
% 1500 &  75 & 1 &  186& 724  &  -- & -- &6 & 344 & 4 & 35 & 11 \\
% 2000 &  75 & 1 & 40 & 1850 & -- & -- &  5& 594 & 3 &  67 & 11\\
% 2000 &  90 & 1 & 12 & 2236 & -- & -- & 4& 483&  2&  27& 13 \\
% 2500 &  90 & 1 & -- & -- & -- & --  & 5 &1323&  3&  122  & 13\\
% 2500 &  100 & 1 & -- & -- & -- & -- & 4 & 1326 &  2& 114  & 13\\
% \hline
% \end{tabular}%
\label{table:qcqp}
\end{table}
\vspace{-1.4em}
\subsection{Matrix Completion}
This subsection solves the \ref{pc_lmc} by the following bilevel optimization form that allows for the noisy observations:
\begin{equation}\label{lmcr1}
	\begin{aligned}
		\V_{\rel}:=\min_{z \in \Re_+}\quad &z\\
		\text{s.t.} \quad & \underline{\delta} \ge \min_{\bm X, \delta} \left\{\delta:  |X_{ij}-A_{ij}|\le \delta , \forall (i,j)\in \Omega, \bm X\in \conv(\X(z)) \right\} \\ &\conv(\X(z)):=\{\bm X \in \Re^{n\times p}: \|\bm X\|_* \le z\}
	\end{aligned} \tag{Matrix Completion-R1},
\end{equation}	
where $\underline{\delta}\ge 0$ is small and  depends on the noisy level of observed entries $\{A_{ij}\}_{(i,j)\in\Omega}$ with $\Omega$ denoting the indices of the observed entries. If $\underline{\delta}= 0$, then \ref{lmcr1} is equivalent to the noiseless  \ref{pc_lmc}.  Following the work of \cite{zhang2012matrix}, we generate a synthetic rank-$k$ matrix $\bm A \in \Re^{n\times p}$ from the model 
$\bm A = \bm U \bm V^{\top} +\epsilon \bm Z$, where  $\bm U \in \Re^{n\times k}, \bm V \in \Re^{p\times n}$  compose i.i.d. entries from the standard normal distribution and $\bm Z \in \Re^{n\times p}$ is a Gaussian white noise. In addition, we set a relatively small noise level $\epsilon=10^{-4}$ and accordingly, we let $\underline{\delta}=10^{-4}$ in  \ref{lmcr1}. The set $\Omega$  is sampled uniformly from $[n] \times [p]$ at random.

For any  solution $z$ of the  \ref{lmcr1},  our column generation \Cref{algo_cg} is able to efficiently solve the embedded optimization problem  over $(\bm X, \delta)$ as this problem falls into our \ref{eq_rel_rank} framework.  This motivates us to use the binary search algorithm for solving \ref{lmcr1}. To do so, we begin with a pair of upper and lower bounds of $z$, where we let $z_U=\|\bm A_{\Omega}\|_*$ and $z_L=0$ with matrix $\bm A_{\Omega}$ containing the observed entries $\{A_{ij}\}_{(i,j)\in\Omega}$ and all other zeros; at each step $t$, we let $z_t = (z_U+z_L)/2$ and compute $(\bm X_t, \delta_t)$  by the proposed \Cref{algo_cg}; and if $\delta_t\ge \underline{\delta}$, we let $z_L = z_t$, otherwise, we let $z_U = z_t$; and we terminate this searching procedure when $z_U -z_L \le 10^{-4}$ holds.

\Cref{table:mc} presents the numerical results, which show that even if being nested within the iterative binary search algorithm,  our \Cref{algo_cg}  is more scalable than the other two methods on the testing cases. However, unlike solving LSOP-R \eqref{eq_lsopr1} and \ref{eq_lsopr2}, the \Cref{algo_cg}  fails to return a low-rank solution, provided that \ref{lmcr1} admits a  rank bound of $\lfloor\sqrt{2|\Omega|+9/4}-1/2\rfloor$ based on \Cref{cor:lmc}.  In addition, we have that $\underline k=1$. Then, according to \Cref{cor:converge},  the rank-reduction \Cref{algo_reduce}  converges fast to reach the rank bound, as shown in \Cref{table:mc}. It is seen that the reduced rank is always less than the original rank $k$, which demonstrates  the strength of \ref{lmcr1} in exactly solving the original \ref{app:lmc}.
We note that \Cref{algo_reduce} may not yield a rank strictly less than the theoretical  bound, which is because given $\underline \delta=10^{-4}$,  all the   linear inequalities in \ref{lmcr1} are nearly binding. 

\vspace{-0.5em}
\begin{table}[htb]
	\centering
	\caption{Numerical performance of three approaches to solving \ref{lmcr1}}
		\setlength{\tabcolsep}{2.4pt}\renewcommand{\arraystretch}{1}
		\begin{tabular}{cccc|rr|rr|rrrr|r}
			\hline
			\multirow{3}{*}{$n$} & 
			\multirow{3}{*}{$p$} & \multirow{3}{*}{$|\Omega|$} & \multirow{3}{*}{$k$} & \multicolumn{2}{c|}{\multirow{2}{*}{MOSEK}} & \multicolumn{2}{c|}{\multirow{2}{*}{Naive CG \tnote{i}}} & \multicolumn{4}{c|}{Proposed Algorithms} & \multicolumn{1}{c}{\multirow{2}{*}{\begin{tabular}[c]{@{}c@{}}Rank bound\end{tabular}}} \\ \cline{9-12}
		&	&  &  & \multicolumn{2}{c|}{} & \multicolumn{2}{c|}{} & \multicolumn{2}{c|}{CG \Cref{algo_cg}} & \multicolumn{2}{c|}{\Cref{algo_reduce}} & \multicolumn{1}{c}{} \\ \cline{5-13} 
		&	&  &  & \multicolumn{1}{c}{rank} & \multicolumn{1}{c|}{time(s)} & \multicolumn{1}{c}{rank} & time(s) & rank & \multicolumn{1}{r|}{time(s)} & reduced rank & time(s) & (\Cref{them:nonsym}) \\ \hline
			100 & 50 & 30 & 10 & 22 & 57& --\tnote{ii} & -- & 23& \multicolumn{1}{r|}{19} & 7 & 1 & 7 \\
			100 & 100 & 40 & 10 & 29 & 303 & -- & -- & 31 & \multicolumn{1}{r|}{52} & 8 & 1 & 8 \\
			100 & 200 & 50 & 10 & -- & --& -- & -- & 36& \multicolumn{1}{r|}{120} & 9& 3 & 9\\
			300 & 200 & 60 & 15 & -- & --& -- & -- & 49& \multicolumn{1}{r|}{342} & 10 & 11 & 10 \\
			300 & 300 & 70 & 15 & -- & --& -- & -- & 60& \multicolumn{1}{r|}{502} & 11 & 16 & 11 \\
			300 & 400 & 80 & 15 & --  & --& -- & -- & 64 & \multicolumn{1}{r|}{1434} & 12 & 29 & 12 \\
%			500 & 500 & 100 & 20 & --  & --& -- & -- \\
			\hline
		\end{tabular}
	\label{table:mc}
\end{table}

\section{Conclusion}\label{sec_conclusion}
This paper studied the low-rank spectral constrained optimization problem by deriving the rank bounds for its partial convexification. 
Our rank bounds remain consistent across different domain sets in the form of \eqref{eq_set} with a fixed matrix space.
% On the contrary, the underlying matrix space has a significant impact on the proposed rank bounds through mapping into vector space of different  sizes. 
This paper specifically investigated positive semidefinite,  non-symmetric, symmetric, and diagonal matrix spaces. 
Due to the flexible domain set, we have applied the proposed  rank bounds  to various low-rank application examples, including kernel learning, QCQP, fair PCA, fair SVD, matrix completion, and sparse ridge regression. 
To harvest the promising theoretical results, we develop an efficient column generation algorithm for solving the partial convexification coupled with a rank-reduction algorithm. One possible future direction is to study the general  nonlinear objective functions.
% To solve the partial convexification while satisfying the , this paper developed an efficient column generation algorithm, which was numerically  evaluated on radio network and optimal power flow problems.
%A further study on the tensor space would be of great interest.

\ACKNOWLEDGMENT{This research has been supported in part by \exclude{the National Science Foundation grants 2246414 and
2246417, and }the Georgia Tech ARC-ACO fellowship. The authors would like to thank Prof. Fatma Kılınç-Karzan from Carnegie Mellon University for her valuable suggestions on the earlier version of this paper.}

\bibliography{reference.bib}

\newpage
\titleformat{\section}{\large\bfseries}{\appendixname~\thesection .}{0.5em}{}
\begin{appendices}
% \renewcommand{\thesection}{Appendix}
% \addcontentsline{toc}{chapter}{Appendices}
% \renewcommand\thesection{Appendix A}
\section{Proofs}
\subsection{Proof of Proposition \ref{prop:conv}}\label{proof:conv}
\propconv*
\begin{proof}
The derivation of set $\conv(\X)$ can be found in \cite{kim2022convexification}[theorems 4 and 7]. We thus focus on proving the closeness of set $\conv(\X)$.
First, we define set $\T$ below
\begin{align*}
\T=\bigg\{(\bm X, \bm x) \in \Q\times   \Re_+^n: f(\bm x)\le 0,  x_1 \ge \cdots \ge x_n,  x_{k+1}=0,  \|\bm X\|_{(\ell)} \le \sum_{i\in [\ell]} x_{i}, \forall \ell \in [k], \tr(\bm X) = \sum_{i\in [k]}x_i\bigg\},
\end{align*}
and the equation $\conv(\X)=\Proj_{\bm X}(\T)$ holds.

%
%\noindent \textbf{Part (i).} First, let us define the vector-based set $\tilde{\X}:= \{\bm x \in \Re_+^n: ||\bm x||_0 \le k, f(\bm  x)\le 0 \}$. Since function $f(\cdot)$  is convex and symmetric,  the domain set $\X$ is permutation-invariant with the eigenvalues of matrix. When $\Q:=\S^n_+$ denotes the positive semidefinite matrix space, all eigenvalues must be nonnegative
%Then, according to results in \cite{kim2021convexification}, the convex hull of the domain set $\X$ is equivalent to 
%the projection of the set $\Y$ below onto matrix  $\bm X$ 
%\begin{align*}
%\Y:= \left\{(\bm X, \bm x )\in  \S_+^n\times \Re_+^n: \bm x \in \conv\left(\tilde{\X}\cap \{  x_1 \ge \cdots \ge  x_n \ge 0\}\right), \bm x \succeq \bm \lambda(\bm X) \right\}.
%\end{align*}
%Above,  the inner convex hull over vector $\bm x$ is equal to
%$$ \conv\left(\tilde{\X}\cap \{  x_1 \ge \cdots \ge  x_n\ge 0\}\right)= \left\{\bm x \in \Re_+^n: f_j(\bm x)\le 0, \forall j\in [d],  x_1\ge  \cdots \ge x_n, x_{k+1}=0\right\},$$ where the zero-norm constraint in set $\tilde \X$ can be readily replaced by $x_{k+1}=0$,  and the majorization constraint $\bm x \succeq \bm \lambda(\bm X)$ reduces to $\Delta_i(\bm X) \le \sum_{\ell\in [i]} x_{\ell}, \forall i \in [k-1], \Delta_n(\bm X) = \tr(\bm X)= \sum_{\ell\in [k]} x_{\ell}$.
%
For any convergent sequence $\{\bm X_t\}_{t\in [T]}$ in set $\conv(\X)$, suppose that the sequence converges to $\bm X^*$, i.e., $\lim\limits_{t\to \infty} \bm X_t= \bm X^*$. Without loss of generality, we can assume that the convergent sequence $\{\bm X_t\}_{t\in [T]}$  is bounded. Then, for each $\bm X_t\in \conv(\X)$ with bounded eigenvalues, there exists a bounded vector $\bm x_t\in \Re_+^n$  such that  $(\bm x_t, \bm X_t) \in \T$. We now obtain a bounded sequence $\{(\bm x_t, \bm X_t)\}_{t\in [T]} $ in set $\T$. According to the Bolzano-Weierstrass theorem, a bounded sequence has a convergent subsequence. Suppose that a convergent subsequence of $\{(\bm x_t, \bm X_t)\}_{t\in [T]} $  converges to $(\bm x^*, \hat{\bm X})$. On the other hand, every subsequence of a convergent sequence converges to the same limit, implying that $\hat{\bm X} = \bm X^*$. Given that function $f(\cdot)$ is closed in \eqref{eq_set}, set $\T$ is clearly closed and thus we have $(\bm x^*, \bm X^*)\in \T$. Since  $\conv(\X)$ is the projection of set $\T$ onto matrix space, the limit $\bm X^*$ must belong to $\conv(\X)$. This proves the closeness.
\qed
\end{proof}

\subsection{Proof of Lemma \ref{lem:ineq}} \label{proof:lemineq}
\lemineq*
\begin{proof}
The proof includes two parts by dividing the inequalities over $[j_0, j_2-1]$ into two subintervals: $[j_0, j_1]$ and $[j_1+1, j_2-1]$.
\begin{enumerate}[(i)]
\item  If the equality can be attained for some $j^* \in[j_0, j_1-1]$, i.e., $\sum_{ i \in [j^*]} \lambda_i = \sum_{ i \in [j^*]}x_i$, given $\sum_{ i \in [j_1]} \lambda_i < \sum_{ i \in [j_1]}x_i$,
then we have
\begin{align*}
\sum_{ i \in [j^*+1, j_1]} \lambda_i < \sum_{ i \in [j^*+1, j_1]} x_i \Longrightarrow \lambda_{j^*} = \cdots =\lambda_{j_1} < \frac{\sum_{ i \in [j^*+1, j_1]} x_i}{j_1-j^*}\le x_{j^*},
\end{align*}
and thus $\sum_{ i \in [j^*-1]} \lambda_i = \sum_{ i \in [j^*-1]}x_i + x_{j^*}-\lambda_{j^*} > \sum_{ i \in [j^*-1]}x_i$, which contradicts with $\bm x \succeq \bm \lambda $.

\item 
If the equality can be attained  for some $j^* \in[j_1+1, j_2-1]$, i.e.,
$\sum_{ i \in [j^*]} \lambda_i = \sum_{ i \in [j^*]}x_i$, given $\sum_{ i \in [j_1]} \lambda_i < \sum_{ i \in [j_1]}x_i$,
then we have
\begin{align*}
\sum_{ i \in [j_1+1, j^*]} \lambda_i > \sum_{ i \in [j_1+1, j^*]} x_i  \Longrightarrow \lambda_{j_1+1} =\cdots= \lambda_{j^*+1}  > \frac{\sum_{ i \in [j_1+1, j^*]} x_i}{j^*-j_1} \ge x_{j^*+1},
\end{align*}
and thus $\sum_{ i \in [j^*+1]} \lambda_i  = \sum_{ i \in [j^*]} x_i + \lambda_{j^*+1}   >\sum_{ i \in [j^*+1]} x_i $,
which contradicts with $\bm x \succeq \bm \lambda $. \qed
\end{enumerate}
\end{proof}
\subsection{Proof of Theorem \ref{prop:rank}}\label{proof:proprank}
\thmproprank*
\begin{proof}
We use contradiction to prove this result and replace $k$ by $\k$. Given $\Q=\S_{+}^n$, suppose that point $\bm X^* \in \S_+^n$ belongs to face $F^d$  with a rank $r > \k + \lfloor\sqrt{2d+9/4}-3/2\rfloor$, i.e.,  
\begin{align}\label{ineq}
{d}+1 < (r-\k+1)(r-\k+2)/2.
\end{align}
Then we let $\lambda_1\ge \cdots \ge \lambda_{r}>0 = \lambda_{r+1}=\cdots =\lambda_n$ denote the eigenvalues of $\bm X^*$ and form the eigenvalue vector $\bm \lambda\in \Re_+^n$. 
%
%It follows that for any $i \in [k, r-1]$, $\Delta_{r-1}(\bm X^*) 
%\begin{align*}
%\sum_{ i \in [\bar{k}]} \lambda_i  \le \cdots \le \sum_{ i \in [r-1]} \lambda_i < \sum_{ i \in [r]} \lambda_i = \sum_{ i \in [\bar{k}]} x_i.
%\end{align*}
%
Since $\lambda_{r+1}=\cdots = \lambda_n=0$, we can rewrite $\bm X^*$ as $\bm X^* = \bm Q_1 \bm \Lambda_1 \bm Q^{\top}_1+\bm Q_2 \bm \Lambda_2 \bm Q^{\top}_2$, where 
\begin{align*}
&\bm \Lambda_1 = \Diag(\lambda_1,\cdots, \lambda_{\k-1}) \in \S_{++}^{\k-1}, \ \  \bm \Lambda_2 = \Diag(\lambda_{\k},\cdots, \lambda_r) \in \S_{++}^{r-\k+1},
\end{align*}
and $\bm Q_1\in \Re^{n\times \k-1}$, $\bm Q_2 \in \Re^{n\times r-\k+1}$ are corresponding eigenvector matrices.

For any $d$-dimensional face $F^d$ with $d\ge 1$, there are ($d+1$)  points such that $F^d\subseteq \aff(\bm X_1, \cdots, \bm X_{d+1})$, where $\aff(\bm X_1, \cdots, \bm X_{d+1})=\{\sum_{i\in [d+1]} \alpha_i \bm X_i: \bm \alpha \in \Re^{d+1}, \sum_{i\in [d+1]} \alpha_i=1 \}$ denotes the affine hull of these points.
Note that any size--$n\times n$ symmetric matrix can be recast into  a vector of length $n(n+1)/2$. 
By inequality \eqref{ineq},  there is  a nonzero symmetric matrix $\bm \Delta \in \S^{r-\k+1}$ satisfying $(d+1)$ equations below
\begin{align}\label{eq:new}
\langle  \bm Q_2^{\top} (\bm X_{1}- \bm X_i) \bm Q_2, \bm \Delta \rangle = 0, \forall i \in [d], \ \  \langle  \bm Q_2^{\top} \bm Q_2, \bm \Delta \rangle =\tr(\bm \Delta) =0,
\end{align}
where the first $d$ equations indicate that matrix $\bm Q_2 \bm \Delta \bm Q^{\top}_2$ is orthogonal to the face $F^d$, i.e., $\bm Q_2 \bm \Delta \bm Q^{\top}_2 \perp F^d$.

Then let us construct the two matrices $\bm X^+ (\delta)  \in \S_+^n$ and $\bm X^- (\delta)  \in \S_+^n$ as below
\begin{align*}
&\bm X^+(\delta)  = \bm X^*+\delta \bm Q_2 \bm \Delta \bm Q^{\top}_2= \bm Q_1 \bm \Lambda_1 \bm Q^{\top}_1+\bm Q_2  (\bm \Lambda_2+\delta \bm \Delta) \bm Q^{\top}_2,\\
&   \bm X^-(\delta)  =  \bm X^*-\delta \bm Q_2 \bm \Delta \bm Q^{\top}_2= \bm Q_1 \bm \Lambda_1 \bm Q^{\top}_1+\bm Q_2  (\bm \Lambda_2-\delta \bm \Delta) \bm Q^{\top}_2,
\end{align*}
for any $\delta >0$. Thus, $\langle \bm A_i, \bm X^+ (\delta)\rangle = \langle \bm A_i, \bm X^- (\delta)\rangle = \langle \bm A_i, \bm X^*\rangle$ for all $i\in [m]$ and $\tr(\bm X^+ (\delta)) =\tr( \bm X^- (\delta)) =\tr(\bm X^*)$ based on equations \eqref{eq:new}.
We further show that when $\delta>0$  is small enough, $\bm X^+ (\delta), \bm X^- (\delta) $ also belong to set $ \conv(\X)$ as detailed by the claim below.  If so,  the equation $\bm X^*=\frac 1 2 \bm X^+ (\delta)+ \frac 1 2 \bm X^- (\delta)$ implies that $\bm X^+ (\delta), \bm X^- (\delta) \in F^d$ according to the definition of face $F^d$. A contradiction with fact that  $\bm Q_2 \bm \Delta \bm Q^{\top}_2 \perp F^d$. Therefore, the point $\bm X^*$ must have a rank no larger than $\k + \lfloor\sqrt{2d+9/4}-3/2\rfloor$. Indeed, we can show that
\begin{claim}\label{claim1}
There is a scalar $\underline{\delta} >0$ such that 
for any $\delta \in (0, \underline{\delta} )$, we have  $\bm X^+ (\delta), \bm X^- (\delta)  \in \conv(\X)$.
\end{claim}
\begin{proof}
First,  according to the characterization of $\conv(\X)$ in \Cref{prop:conv} and $\bm X^* \in \conv(\X)$, there is a vector $\bm x^*\in \Re_+^n$ such that $\bm \lambda \preceq \bm x^*$. 
Let $\bm \lambda^+(\delta)\in \Re_+^n$ and $\bm \lambda^-(\delta)\in \Re_+^n$  denote the eigenvalue vectors of $\bm X^+ (\delta)$ and $\bm X^- (\delta)$, respectively, provided that $\delta>0$ is small enough to ensure $\bm X^+ (\delta), \bm X^- (\delta)  \in \S_+^n$. According to \Cref{prop:conv}, $\bm X^+ (\delta), \bm X^- (\delta)  \in \conv(\X)$ holds if $\bm \lambda^+(\delta), \bm \lambda^-(\delta) \preceq \bm x^*$. 
Next, we will show that if the scalar $\delta>0$ only imposes slight perturbation  on eigenvalue vector $\bm \lambda$ of matrix $\bm X^*$, then $\bm \lambda^+(\delta), \bm \lambda^-(\delta) \preceq \bm x^*$ still holds given $\bm \lambda \preceq \bm x^*$. Let us first analyze the perturbation from $\bm \lambda$ to $\bm \lambda^+(\delta)$ caused by  scalar $\delta$. 

Specifically, according to \Cref{prop:conv},  a pair $(\bm X^*, \bm x^*)$ satisfies  $x_1^*\ge \cdots \ge x_{\k}^*\ge 0= x_{\k+1}^*= \cdots = x^*_n$, $\|\bm X^*\|_{(\ell)}\le  \sum_{i\in [\ell]}x^*_i$ for each $\ell \in [\k-1]$ and $\tr(\bm X^*) = \sum_{i\in [\k]} x_i^*$. Thus, we have
$$\sum_{i\in [\ell]} \lambda_i =\|\bm X^*\|_{(\ell)} < \|\bm X^*\|_{(r)} =\tr(\bm X^*) = \sum_{i\in [\k]} x_i^*=\sum_{i\in [\ell]} x_i^*, \ \ \forall \ell \in [\k, r-1], \vspace{-.4em}$$
where the first inequality  and the second equality are from the fact that $\rank(\bm X^*)=r$ and $\bm X^*\in \S_{+}^n$, and the last equality is due to $x^*_{\k+1}= \cdots = x^*_n=0$.
Suppose that $\lambda_{0}=\infty$ and $\lambda_{j_0-1} >\lambda_{j_0} =\cdots = \lambda_{\k}$ for some $j_0 \in [\k]$. The according to \Cref{lem:ineq} with $j_1=\k$, we have 
\begin{align}\label{eq_pert}
\lambda_{j_0-1} >\lambda_{j_0}, \ \ \sum_{ i \in [\ell]} \lambda_i <  \sum_{ i \in [\ell]} x_i^*, \forall \ell \in [j_0, \k].
\end{align}
Then we define a  constant  $c^*>0$ by
$$c^* := \min\bigg\{\lambda_{j_0-1} - \lambda_{j_0}, \min_{\ell \in [j_0, \k]} \bigg(\sum_{ i \in [\ell]} x_i^*-\sum_{ i \in [\ell]} \lambda_i\bigg) \bigg\}.$$

Suppose that matrix $\bm \Lambda_2 +\delta \bm \Delta \in \S_+^{r-\k+1}$ admits eigenvector matrix $\bm P\in \Re^{(r-\k+1)\times (r-\k+1)}$. Then it is easy to check that $
\begin{pmatrix}
\bm Q_1^{\top} \\
\bm P^{\top} \bm Q_2^{\top}
\end{pmatrix} \begin{pmatrix}
\bm Q_1 & \bm Q_2 \bm P
\end{pmatrix}  = \bm I_r$, since matrix $\begin{pmatrix}
\bm Q_1 & \bm Q_2
\end{pmatrix} \in \Re^{n\times r}$ is orthonormal. It follows that the nonzero eigenvalues of matrix $\bm X^+(\delta)$ contain those of $\bm \Lambda_1$ and $\bm \Lambda_2+\delta \bm \Delta$. That is, $\{\lambda_i\}_{i\in [\k-1]}$ are also eigenvalues of matrix $\bm X^+(\delta)$. Therefore, the perturbation from  $\bm \lambda$ to $\bm \lambda^+(\delta)$ is captured by the impact of matrix $\delta \bm \Delta$ on $\bm \Lambda_2$.
We let $ \lambda_{\max}$ denote the largest singular value of matrix $\bm \Delta$, according to Weyl's inequality, the eigenvalue vector $\bm \lambda^+(\delta)$ of matrix $\bm X^+(\delta)$ satisfies
\begin{align}\label{eq:eigen2}
\bm \lambda^+(\delta) \le \begin{pmatrix}
\lambda_1,\cdots, \lambda_{\k-1},\lambda_{\k} + \delta \lambda_{\max}, \cdots, \lambda_{r} + \delta \lambda_{\max},  0, \cdots, 0
\end{pmatrix},
\end{align}
where note that the eigenvalue vector $\bm \lambda^+(\delta)$ above may not follow a descending order and the inequality is the component-wise convention for vectors.
Without loss of generality, we let $\lambda^+_1(\delta) \ge \cdots \ge \lambda^+_n(\delta) \ge 0$ denote the eigenvalues of $\bm X^+(\delta)$, i.e., a permutation of elements in vector $\bm \lambda^+(\delta)$ by a descending order.

By letting the scalar  $\underline{\delta} \le  \frac{c^*}{(r-\k+1) \lambda_{\max}}$, after perturbation of eigenvalue vector $\bm \lambda^+(\delta)$ by any $\delta \in (0, \underline{\delta})$, we still have that $\bm x^* \succeq \bm \lambda^+(\delta)$ as shown follows.

\begin{enumerate}[(a)]
\item We show that $\lambda^+_i(\delta) = \lambda_i$ for any $i \in [j_0-1]$. If not, there must exist  $\lambda^+_{\ell}(\delta) > \lambda_{\ell}$ for some $\ell \in [j_0-1]$, which means that $\lambda^+_{\ell} \notin \{\lambda_i\}_{i\in [\k-1]}$  as $\{\lambda_i\}_{i\in [\k-1]}$ are eigenvalues of matrix $\bm X^+(\delta)$. 
According to the inequality \eqref{eq:eigen2},  $\lambda^+_{\ell}(\delta)$ must satisfy  $\lambda^+_{\ell}(\delta)\le \lambda_{\k} + \delta \lambda_{\max} \le  \lambda_{\k} + \lambda_{j_0-1} - \lambda_{j_0} \le \lambda_{j_0-1} \le \lambda_{\ell}$, where the second inequality is due to $\delta\lambda_{\max} \le c^*/(r-\k +1)$ and the third one is from $j_0 \le \k$. Therefore, $\lambda^+_i(\delta) = \lambda_i$ must hold for any $i \in [j_0-1]$  and we have
$ \sum_{ i \in [\ell]} \lambda^+_i(\delta) =  \sum_{ i \in [\ell]} \lambda_i \le \sum_{ i \in [\ell]} x_i^*$ for all $\ell \in [j_0-1]$. 

\item According to the inequality \eqref{eq:eigen2}, there are at most $(r-\k+1)$ entries in $\bm \lambda^+(\delta)$ that go beyond those of $\bm \lambda$ up to $\delta \lambda_{\max}$. Hence, we have
\[ \sum_{ i \in [\ell]} \lambda^+_i(\delta) \le  \sum_{ i \in [\ell]} \lambda_i  + (r-\k+1) \delta \lambda_{\max}  \le \sum_{ i \in [\ell]} \lambda_i + c^*  \le \sum_{ i \in [\ell]} x_i^*, \forall   \ell \in [j_0,\k], \]
where the second inequality is from the definition of constant $c^*$.

\item Since $\tr(\bm X^+(\delta)) =\tr(\bm X^*) = \sum_{i\in [\k]} \bm x^*$, we have
$\sum_{ i \in [\ell]} \lambda^+_i(\delta)  \le \tr(\bm X^+(\delta)) =\sum_{i\in [\k]} \bm x^*=\sum_{i\in [\ell]} \bm x^*$ for each $\ell \in [\k+1, n]$. 
\end{enumerate}
Similarly, there exists a scalar $\underline{\delta}^* >0$ such that $\bm x^* \succeq \bm \lambda^-(\delta)$ for any $\delta \in (0, \underline{\delta}^* )$. Letting $\underline{\delta}: =\min\{\underline{\delta}^*,\underline{\delta}\}$, we  complete the proof of the rank bound. \qedA
\end{proof}
\qed
\end{proof}

\subsection{Proof of Corollary \ref{cor:spca}}\label{proof:spca}
\corspca*
\begin{proof}
It is easy to check that the domain set $\X$ in \ref{app:fpca} satisfies $\k =k$.
According to \cite{li2022exactness}, at most $(m-1)$ linear  inequalities are binding at any extreme points in the feasible set of \ref{pc_fpca}.  Besides, the optimal value of \ref{pc_fpca}  must be finite due to the boundedness of set $\conv(\X)$.
Hence, in this example, the rank bound in \Cref{them:rank} reduces to $k +  \left\lfloor \sqrt{2 {m}+{1}/{4}}-{3}/{2}\right\rfloor$. When $m\le 2$, the rank bound reaches $k$, and thus  an  optimal extreme point of \ref{pc_fpca} coincides  with that of the original \ref{app:fpca}.
\qed
\end{proof}

\subsection{Proof of Corollary \ref{cor:qcqp}}\label{proof:qcqp}
\corqcqp*
\begin{proof}
Given a domain set $\X:=\{\bm X\in \S_{+}^n: \rank(\bm X)\le 1\}$,  we have that $F(\bm X)=0$ in \eqref{eq_set} and $\k=1$ by \Cref{def:k}.
According to the characterization of $\conv(\X)$ in \Cref{prop:conv},  we have
\begin{align*}
 \conv(\D)=\bigg\{\bm X \in \S_+^n: \exists \bm x \in   \Re_+^n,  x_1 \ge \cdots \ge x_n,  x_{k+1}=0,  \|\bm X\|_{(\ell)} \le \sum_{i\in [\ell]} x_{i}, \forall \ell \in [k], \tr(\bm X) = \sum_{i\in [k]}x_i\bigg\}.
\end{align*}
Note that the vector $\bm{x}$ can be unbounded.
Hence, in the proof of \Cref{prop:rank}, it suffices to construct a symmetric matrix $\bm \Delta \in \S^{r}$ to satisfy the first $d$ equations in \eqref{eq:new} since we can arbitrarily adjust the vector $\bm x^*$ therein. When $\V_{\rel}>-\infty$,  this observation leads to a better rank bound $r(r+1)/2\le  m $, where $r$ denotes the rank of an  extreme point to \ref{pc:qcqp}. When $\tilde m \le 2$, the equivalence between \ref{app:qcqp} and \ref{pc:qcqp} is directly from \Cref{cor:ch}.
\qed
\end{proof}

\subsection{Proof of Theorem \ref{them:nonface}}\label{proof:nonface}
Before proving  \Cref{them:nonface}, let us give an explicit characterization of the convex hull of domain set $\X$ in the non-symmetric matrix space.
\begin{proposition}\label{prop:nonsymconv}
	For a domain set $\X$ in non-symmetric matrix space, i.e., $\Q:=\Re^{n\times p}$ in \eqref{eq_set},  its convex hull $\conv(\X)$ is equal to
	\begin{equation*}
		\begin{aligned}
			\bigg\{\bm X\in  \Re^{n\times p}:  \exists  \bm x \in \Re_+^n, f(\bm x)\le 0,  x_1 \ge \cdots \ge x_n,  x_{k+1}=0,
			\|\bm X\|_{(\ell)} \le \sum_{i\in [\ell]} x_{i}, \forall \ell \in [k-1], \|\bm X\|_* \le \sum_{i\in [k]}x_i\bigg\}
		\end{aligned}
	\end{equation*}
	and is a closed set.
\end{proposition}
\begin{proof}
	The derivation of $\conv(\X)$ has been shown in \cite{kim2022convexification}[theorem 7]. The closeness proof  of set  $\conv(\X)$ follows that of \Cref{prop:conv}.   \qed
\end{proof}

\thmnonface*

\begin{proof}
 Suppose that there is a non-symmetric matrix $\bm X^* \in F^d$ with rank $r> \tilde k + \lfloor\sqrt{2d+9/4}-3/2\rfloor$. 
 % It is evident that any  non-symmetric matrix of size $n\times n$ can be mapped into a vector of length  $n^2$.  
 Then, we denote the singular value decomposition of matrix $\bm X^*$ by $\bm X^* = \bm Q_1 \bm \Lambda_1 \bm P^{\top}_1+\bm Q_2 \bm \Lambda_2 \bm P^{\top}_2$ and denote its singular values by $\lambda_1\ge \cdots \ge \lambda_r> 0 =\lambda_{r+1}=\cdots =\lambda_n$, where 
 \begin{align*}
&\bm \Lambda_1 := \Diag(\lambda_1,\cdots, \lambda_{\k-1}) \in \S_{++}^{\k-1}, \ \  \bm \Lambda_2 := \Diag(\lambda_{\k},\cdots, \lambda_r) \in \S_{++}^{r-\k+1},
\end{align*}
where $\bm Q_1, \bm Q_2$ and $\bm P_1, \bm P_2$ are corresponding left and right singular vectors.

Following the proof of \Cref{prop:rank} and using matrices $\{\bm X_i\}_{i\in [d+1]}$ therein, we can construct a nonzero symmetric matrix $\bm \Delta \in \S^{r-\k+1}$ 
satisfying the equations below
\[\langle  \bm P_2^{\top} (\bm X_{i} - \bm X_{d+1}) \bm Q_2, \bm \Delta \rangle = 0, \forall i \in [d], \ \  \tr(\bm \Delta) =0.\]

It should be noted that according to \Cref{prop:nonsymconv}, the description of convex hull set $\conv(\X)$ with $\Q=\Re^{n\times p}$ relies on the nuclear norm of a non-symmetric matrix. 

Next, we define two matrices $\bm X^{\pm}(\delta)=\bm X^*\pm \delta \bm Q_2 \bm \Delta \bm P_2^{\top}$,
where $\delta>0$ is a small scalar such that matrices $\begin{pmatrix}
    \bm \Lambda_1 & \bm 0 \\
    \bm 0 & \bm \Lambda_2 \pm \delta \bm \Delta
\end{pmatrix} \in \S_+^{r}$ are  positive semidefinite.
Hence, we can show that matrix $\bm X^+(\delta)$ has the same nonzero singular value vector as that of positive semidefinite matrix  
$\begin{pmatrix}
    \bm \Lambda_1 & \bm 0 \\
    \bm 0 & \bm \Lambda_2 + \delta \bm \Delta
\end{pmatrix}$,
which means that
$$\|\bm X^{+}(\delta) \|_* = \left\|\begin{pmatrix}
    \bm \Lambda_1 & \bm 0 \\
    \bm 0 & \bm \Lambda_2 + \delta \bm \Delta
\end{pmatrix}\right\|_*=  \tr( \bm \Lambda_1) + \tr( \bm \Lambda_2 + \delta \bm \Delta)=\tr( \bm \Lambda_1)+\tr(\bm \Lambda_2)=\|\bm X^*\|_*, $$
where the second equation is because for any positive semidefinite matrix, the eigenvalues meet singular values.
Analogously,  the equation  $\|\bm X^{-}(\delta) \|_* =\|\bm X^*\|_*$ holds for matrix $\bm X^{-} (\delta)$. Then, 
the existence of such a symmetric matrix $\bm \Delta$  can form a contradiction as shown in \Cref{prop:rank}, leading to the desired rank bound $d+1 \ge {(r-\k+1)(r-\k+2)}/{2}$.
\qed
\end{proof}
\subsection{Proof of Corollary \ref{cor:fsvd}}\label{proof:corfsvd}
\corfsvd*
\begin{proof}
	Analogous to   \ref{pc_fpca},  there are at most $(m-1)$  linear equations for all extreme points in the feasible set of \ref{pc_fsvd}. According to \Cref{them:nonsym}, any optimal extreme point of its \ref{pc_fsvd} has a rank at most $k+\left\lfloor \sqrt{2m+1/4} -3/2\right\rfloor$. When $m\le 2$, the rank bound becomes $k$ and thus the \ref{pc_fsvd}  can yield a rank-$k$ solution that is also optimal to the original \ref{app:fsvd}.  \qed
\end{proof}
\subsection{Proof of Corollary \ref{cor:lmc}}\label{proof:lmc}
\corlmc*
\begin{proof}
Let $(\hat{z}, \hat{\bm X})$ denote an optimal solution to the \ref{pc_lmc}. Then given an optimal solution $\hat{z}$, the compact set below can be viewed as an optimal set of variable $\bm X$ to  \ref{pc_lmc} 
\begin{align*}
\left\{\bm X\in \conv(\D):  X_{ij} =\hat{X}_{ij}, \forall (i,j)\in \Omega\right\},
\end{align*}
where now we have a domain set $\D:=\{\bm X\in \Re^{n\times p}: \rank(\bm X)\le k, \|\bm X\|_* \le \hat{z}\}.$

It is seen that the domain set $\X$ above  can be viewed as a special case of that in \Cref{eg1}; hence, we now have $\k=1 \le k$  by \Cref{def:k}. Since there are $m=|\Omega|$ linear inequalities, according to \Cref{them:nonsym}, any optimal extreme point $\bm X^*$ has a rank at most $1+\lfloor\sqrt{|\Omega|+9/4} -3/2\rfloor$. When $\lfloor\sqrt{|\Omega|+9/4} -1/2\rfloor \le k$,  the \ref{pc_lmc} can achieve the desired rank-$k$ solution  as the original \ref{app:lmc} based on \Cref{cor:nonsym}.
\qed
\end{proof}

\subsection{Proof of Theorem \ref{them:symsign} and Its Implication of \ref{eq_rel_rank} Exactness}\label{proof:symsign}
\thmsymsign*
\begin{proof}
Since the domain set $\X$ is sign-invariant, according to \Cref{prop:nonsymconv}, its convex hull $\conv(\X)$ is equal to
 \begin{equation*}
\begin{aligned}
 \bigg\{\bm X\in  \Q:  \exists  \bm x \in \Re_+^n, f(\bm x)\le 0,  x_1 \ge \cdots \ge x_n,  x_{k+1}=0, 
\|\bm X\|_{(\ell)} \le \sum_{i\in [\ell]} x_{i}, \forall \ell \in [k-1], \|\bm X\|_* \le \sum_{i\in [k]}x_i\bigg\}.
\end{aligned}
\end{equation*}
Then, the result in \Cref{prop:rank} can be readily extended to set $\conv(\D)$ with the symmetric indefinite matrix space, except considering the singular value decomposition rather than eigen-decomposition. That is, any $d$-dimensional face in set $\conv(\X)$ satisfies the rank-$\tilde k + \lfloor\sqrt{2 d+9/4}-3/2\rfloor$ constraint.
Using \Cref{lem:face}, the remaining proof is identical to that of \Cref{them:rank} and thus is omitted. \qed
\end{proof}

Notable, \Cref{them:symsign} implies that by letting $\tilde k + \lfloor\sqrt{2 d+9/4}-3/2\rfloor\le k$,  on can obtain a sufficient condition under which the \ref{eq_rel_rank} coincides with  the \ref{eq_rank}.
\begin{proposition}\label{cor:symch}
	Given   $\Q:=\S^n$ in \eqref{eq_set} and integer $\k\le k$ following \Cref{def:k}, suppose that the function $f(\cdot)$  in the domain set $\X$ is sign-invariant  and the \ref{eq_rel_rank}  admits  a line-free feasible set. Then,  if $\tilde{m} \le {(k-\k+2)(k-\k+3)}/{2} -2$ holds, we have that 
	\begin{enumerate}[(i)]
		\item Each feasible extreme point in \ref{eq_rel_rank} has a rank at most $k$; and
		\item The \ref{eq_rel_rank}  achieves the same optimal value as the original \ref{eq_rank}, i.e., $\V_{\opt} = \V_{\rel}$ if the \ref{eq_rel_rank}  yields a finite optimal value, i.e., $\V_{\rel}>-\infty$.
	\end{enumerate}
\end{proposition}
\begin{proof}
	Using \Cref{them:symsign}, the proof is identical to that of \Cref{cor:ch} and thus is omitted. \qed
\end{proof}

\subsection{Proof of Proposition \ref{prop:symconv}} \label{proof:symconv}
\propsymconv*
\begin{proof}
By \Cref{def:k}, we will replace $k$ by $\k$ in the below.
 According to  \cite{kim2022convexification}[theorem 4], we have that  $\conv(\X)= \Proj_{\bm X}(\T)$, where $\T:= \{(\bm X, \bm x )\in  \S^n\times \Re^n: \bm x \in \conv(\tilde{\X}\cap \{  x_1 \ge \cdots \ge  x_n\}),\bm x \succeq \bm \lambda(\bm X) \}$ and $\tilde{\X}:= \{\bm x \in \Re^n: \|\bm x\|_0 \le \k, f(\bm  x)\le 0 \}$. %and $\bm \lambda(\bm X)  \in \Re^n$ denotes the eigenvalues of symmetric matrix $\bm X$ that can be negative.

Next, let us consider the key inner convex hull   over $\bm x$ above, i.e., $\conv(\tilde{\X}\cap \{  x_1 \ge \cdots \ge  x_n\})$, where there may  exist both positive and negative elements in $\bm{x}$.
Given $x_1\geq x_2\geq \ldots \geq x_n$,  $x_{i_1}=\cdots = x_{i_2}=0$ must hold whenever $x_{i_1} =x_{i_2} =0$ with $1\le i_1\le i_2\le n$ and $i_2-i_1\geq n-\k-1$. Besides,  the constraint $\|\bm x\|_0 \le \k$ implies that there are at least $(n-\k)$ zero entries in the vector $\bm x$. Therefore, to remove the constraint $\|\bm x\|_0 \le \k$, we can split set $\tilde{\X}\cap\{x_1\ge \cdots \ge x_n \}$ into $(\k+1)$ subsets, i.e., $\cup_{s\in [\k+1]} \tilde{\X}^{s}$ depending on where zero entries of  $\bm x$  are located, where for each $s\in [\k+1]$, set $\tilde{\X}^{s}$ is defined by
\begin{align*}
\tilde{\X}^{s} := \left\{\bm x \in \Re^n: f(\bm x)\le 0, x_1 \ge \cdots \ge x_n,   x_{s}= x_{s+n-\k-1}=0 \right\}.
\end{align*}

Finally, plugging $\cup_{s\in [\k+1]} \tilde{\X}^{s} $ into set $\T$ and moving the convex hull operation outside, we have $\conv(\X) =\Proj_{\bm X}(\T) \supseteq  \conv(\cup_{s \in[\k+1]} \Y^{s})$. Note that the convex hull and projection operators can interchange. On the other hand, for a matrix $\bm{X}\in \X$, suppose its eigenvalue vector $\bm{\lambda}(\bm {X})$ is sorted in descending order. Then we see that $\bm{X}\in \cup_{s \in[\k+1]} \Y^{s}$. Thus,  we conclude that $\conv(\X) = \conv(\cup_{s \in[\k+1]} \Y^{s}) $.
Besides, we have $\X \subseteq \cup_{s \in[\k+1]} \Y^{s}$.
Following the same proof of \Cref{prop:conv}, each set $\Y^s$ must be closed for any $s\in [\k+1]$.
\qed
\end{proof}

\subsection{Proof of Theorem \ref{them:symface}} \label{proof:symface}
\thmsymface*
\begin{proof}
 We prove the rank bounds by contradiction. There are two cases to be discussed depending on the value of $s\in [\k+1]$. %Note that for the compact domain set $\X)$, analogous to  \Cref{them:rank}, we can directly replace $k$ by $\k$ in the following analysis.
\par
\noindent \textbf{Case I.} Suppose $s=1$ or $s=\k +1$.
Then for any matrix $\bm X^* \in \Y^s \subseteq \S^n$, its eigenvalue vector is either  all nonpositive  or all nonnegative as established in \Cref{prop:symconv}. 
Following the proof of \Cref{prop:rank},  we arrive at the same rank bound.

\noindent \textbf{Case II.} Suppose $2\le s \le \k$. Then for any  point $\bm X^* \in F^d $ with rank $r>\k$. 
Without loss of generality, we let $\bm X^* = \bm Q \Diag(\bm \lambda) \bm Q^{\top}$ be its eigen-decomposition and 
denote the  eigenvalues  by $\lambda_1  \ge \cdots \ge  \lambda_{d_1-1}>0= \lambda_{d_1}=\cdots = \lambda_{d_2-1} =0 > \lambda_{d_2}\ge \cdots\ge \lambda_{n}$ with $1\le d_1 \le d_2\le n$. 
When $d_1=1$, all eigenvalues of matrix $\bm X^*$ are nonpositive and the proof of \Cref{prop:rank} simply follows. Next, let us consider $2\le d_1 \le d_2 \le n$. Since there are only $r$ nonzero eigenvalues in matrix $\bm X^*\in \S^n$,  we have
\begin{align*}
d_2-d_1=n-r.
\end{align*}

 In addition, the eigenvalue vector of  matrix $\bm X^*$ corresponds to a vector $\bm x^*\in \Re^n$ such that $(\bm X^*, \bm x^*) \in \Y^s$, and  according to the definition of set $\Y^s$ \eqref{eq_ys} in  \Cref{prop:symconv}, we have
\begin{equation}\label{x_lam}
\begin{aligned}
&\bm x^*= [x^*_1, \cdots, x^*_{s-1}, 0, \cdots, 0, x_{s+n-\k}^*,\cdots, x^*_{n}]^\top, \ \  x^*_1 \ge \cdots \ge x^*_n,\\
&\bm \lambda= [\lambda_1, \cdots, \lambda_{d_1-1}, 0, \cdots, 0, \lambda_{d_2},\cdots, \lambda_{n}]^\top, \ \  \lambda_1\ge \cdots \ge \lambda_n,\\
&\sum_{i\in [j]}\lambda_i \le  \sum_{i\in [j]}x^*_i,\forall j\in [n-1], \ \ \sum_{i\in [n]}\lambda_i = \sum_{i\in [n]}x^*_i.
\end{aligned}
\end{equation}
Note that there must exist both positive and negative entries in vector $\bm x^*$.
Then, we claim that 
\begin{align}\label{ineqs}
\sum_{i\in [d_1-1]} \lambda_i  = \sum_{i\in [d_2-1]} \lambda_i\le  \sum_{i\in [s-1]} x^*_i.
\end{align}
Otherwise we have
$
\sum_{i\in [d_1-1]} x^*_i \ge  \sum_{i\in [d_1-1]} \lambda_i > \sum_{ i \in [s-1]} x^*_i
$,
which contradicts with the fact that $\max_{T\subseteq [n]} \sum_{i\in T} x_i^* = \sum_{ i \in [s-1]} x^*_i$.
Then there are three parts to be discussed depending on the relation among two pairs: $(s, d_1)$ and $(d_2, s+n-\k)$.

\begin{enumerate}[(i)]
\item Suppose  $2\le s<d_1 \le d_2 < s+n-\k \le n$. Then according to the inequality \eqref{ineqs}, we have
\begin{equation}\label{case1}
\begin{aligned}
&\sum_{i\in [\ell]} \lambda_i < \sum_{i\in [d_1-1]} \lambda_i \le \sum_{i\in [s-1]} x^*_i=  \sum_{i\in [\ell]} x^*_i, \ \  \forall \ell \in [s-1,d_1-2],\\
& 
\sum_{i\in [\ell]} \lambda_i < \sum_{i\in [d_2-1]} \lambda_i \le \sum_{i\in [s-1]} x^*_i=  \sum_{i\in [\ell]} x^*_i, \ \ \forall \ell \in [d_2,s+n-\k-1],
\end{aligned}
\end{equation}
where the strict inequalities are due to $\lambda_{d_1-1}>0$ and $\lambda_{d_2}<0$ and the equations result from  $x^*_{s} =\cdots = x_{s+n-\k-1}^*=0$ as detailed in \eqref{x_lam}.

Next,  to facilitate our analysis, we split $\bm x^*$ and $\bm \lambda$  in \eqref{x_lam} into three sign-definite subvectors
\begin{equation}\label{x_lam2}
\begin{aligned}
&(\bm x^*)^1 \in \Re_+^{d_1-1} :=\bm x^*_{[d_1-1]}, \ \ (\bm x^*)^2 \in \Re^{d_2-d_1}:= \bm x^*_{[d_1,d_2-1]}=\bm 0, \ \ (\bm x^*)^3 \in -\Re_+^{n-d_2+1} := \bm x^*_{[d_2, n]}, \\
&\bm\lambda^1 \in \Re_{++}^{d_1-1} :=\bm \lambda_{[d_1-1]}, \ \ \bm\lambda^2 \in \Re^{d_2-d_1}:= \bm \lambda_{[d_1, d_2-1]} =\bm 0, \ \ \bm \lambda^3\in -\Re_{++}^{n-d_2+1} := \bm \lambda_{[ d_2, n]}.
\end{aligned}
\end{equation}
% and  $\sum_{i\in [n]} \lambda_i = \sum_{ i \in [n]} x_i^*$, using  \Cref{lem:ineq} and following the proof in \Cref{them:rank}, any symmetric matrix $\bm \Delta_2 \in \S^{s+n-\k-d_2+1}$ with trace being zero satisfies
%\[ [(\bm x^*)^1, (\bm x^*)^1, (\bm x^*)^1] \succeq [(\bm \lambda)^1, (\bm x^*)^1, (\bm x^*)^1] 
%\bm\lambda^1 \pm \delta \bm\lambda(\bm \Delta_1),   \ \ \bm\lambda^1 \pm \delta \bm\lambda(\bm \Delta_1) \in \Re_+^{d_1-s+1},  \ \ \forall \delta \in (0, \underline{\delta}).\]
Suppose $\bm \Lambda_i =\Diag(\bm \lambda^i)$ for $i=1,2,3$, then $\bm \Lambda_1 \in \S_{++}^{d_1-1}$ and $\bm \Lambda_3 \in -\S_{++}^{n-d_2+1}$. Following the proof in \Cref{prop:rank}, given the strict inequalities from $s-1$ to $d_1-2$ on the first  line of \eqref{case1}, we can split $\bm \Lambda_1$ into $\bm \Lambda_1^1 \in \S_{++}^{s-2} := \Diag(\lambda_1, \cdots, \lambda_{s-2})$ and $\bm \Lambda_1^2 \in  \in \S_{++}^{d_1-s+1} := \Diag(\lambda_{s-1}, \cdots, \lambda_{d_1-1})$. Similarly, using the strict inequalities from $d_2$ to $s+n-\k-1$ on the second  line of \eqref{case1}, we split $\bm \Lambda_3$ into $\bm \Lambda_3^1 \in- \S_{++}^{s+n-\k-d_2+1}:= \Diag(\lambda_{d_2}, \cdots, \lambda_{s+n-\k})$ and $\bm \Lambda_3^2 \in -\S_{++}^{\k-s} := \Diag(\lambda_{s+n-\k+1}, \cdots, \lambda_{n})$.

Thus, we rewrite the matrix $\bm X^*\in \S^n$ as the eigen-decomposition below
\begin{align*}
\bm X^* &= \bm Q_1 \bm \Lambda_1 \bm Q_1^{\top} + \bm Q_2 \bm \Lambda_2\bm Q_2^{\top} + \bm Q_3\bm \Lambda_3\bm Q_3^{\top} = \bm Q_1 \bm \Lambda_1 \bm Q_1^{\top} + \bm Q_3\bm \Lambda_3\bm Q_3^{\top} \\
&= \begin{pmatrix}
\bm Q_{1}^1 & \bm Q_{1}^2
\end{pmatrix}
\begin{pmatrix}
\bm \Lambda_{1}^1 & \bm 0\\
\bm 0 & \bm \Lambda_{1}^2
\end{pmatrix} \begin{pmatrix}
\bm Q_{1}^1 & \bm Q_{1}^2
\end{pmatrix}^{\top} + \begin{pmatrix}
\bm Q_{3}^1 & \bm Q_{3}^2
\end{pmatrix}
\begin{pmatrix}
\bm \Lambda_{3}^1 & \bm 0\\
\bm 0 & \bm \Lambda_{3}^2
\end{pmatrix} \begin{pmatrix}
\bm Q_{3}^1 & \bm Q_{3}^2
\end{pmatrix}^{\top}
\end{align*}
%$$\bm X^*= \bm Q_1 \bm \Lambda_1 \bm Q_1^{\top} + \bm Q_1 \bm \Lambda_1 \bm Q_1^{\top} 
%\begin{pmatrix}
%\bm Q_1   & \bm Q_3
%\end{pmatrix} 
%\begin{pmatrix}
%\bm \Lambda_1  & \bm 0 \\
%  \bm 0 &\bm \Lambda_3
%\end{pmatrix} \begin{pmatrix}
%\bm Q_1   & \bm Q_3
%\end{pmatrix} ^{\top} ,$$
where the eigenvector matrix $\bm Q = \begin{pmatrix}
\bm Q_1 &\bm Q_2 &\bm Q_3
\end{pmatrix}$ can be decomposed accordingly and $\bm Q_2$ corresponds to zero eigenvalues and is thus omitted.

For any $d$-dimensional face $F^d$, there are ($d+1$)  points such that $F^d\subseteq \aff(\bm X_1, \cdots, \bm X_{d+1})$, where $\aff(\bm X_1, \cdots, \bm X_{d+1})=\{\sum_{i\in [d+1]} \alpha_i \bm X_i: \bm \alpha \in \Re^{d+1}, \sum_{i\in [d+1]} \alpha_i=1 \}$ denotes the affine hull of these points.
Suppose that the inequality $d+2 < 1/2((d_1-s+1)(d_1-s+2)+ (s+n-\k -d_2+1)(s+n-\k -d_2+2))$ holds, then there exists  a block symmetric matrix $\bm \Delta = \begin{pmatrix}
\bm \Delta_1 \in \S^{d_1-s+1} &\bm 0\\
\bm 0 & \bm \Delta_3 \in \S^{s+n-\k -d_2+1}
\end{pmatrix}$ such that
\begin{align} \label{eq:linear}
\langle  \begin{pmatrix}
\bm Q_1^2  & \bm Q_3^1
\end{pmatrix}^{\top} (\bm X_i - \bm X_{d+1}) \begin{pmatrix}
\bm Q_1^2  & \bm Q_3^1
\end{pmatrix}, \bm \Delta \rangle = 0, \forall i \in [d], \tr(\bm \Delta_1) =0, \tr(\bm \Delta_3)=0,
\end{align}
where the first $d$ equations imply
  \begin{align}\label{eq:perp}
    \begin{pmatrix}
      \bm Q_{1}^1 & \bm Q_{1}^2
\end{pmatrix}
\begin{pmatrix}
\bm 0 & \bm 0\\
\bm 0 &  \bm \Delta_1
\end{pmatrix} \begin{pmatrix}
\bm Q_{1}^1 & \bm Q_{1}^2
\end{pmatrix}^{\top} + \begin{pmatrix}
\bm \Delta_3 & \bm 0\\
\bm 0 & \bm 0
\end{pmatrix} \begin{pmatrix}
\bm Q_{3}^1 & \bm Q_{3}^2
\end{pmatrix}^{\top}  \perp F^d. 
  \end{align}
Then let us construct the two matrices $\bm X^+ \in \S^n$ and $\bm X^- \in \S^n$ as below
$$\bm X^{\pm}(\delta) = 
\begin{pmatrix}
\bm Q_{1}^1 & \bm Q_{1}^2
\end{pmatrix}
\begin{pmatrix}
\bm \Lambda_{1}^1 & \bm 0\\
\bm 0 & \bm \Lambda_{1}^2 \pm \delta \bm \Delta_1
\end{pmatrix} \begin{pmatrix}
\bm Q_{1}^1 & \bm Q_{1}^2
\end{pmatrix}^{\top} + \begin{pmatrix}
\bm Q_{3}^1 & \bm Q_{3}^2
\end{pmatrix}
\begin{pmatrix}
\bm \Lambda_{3}^1 \pm \delta \bm \Delta_3 & \bm 0\\
\bm 0 & \bm \Lambda_{3}^2
\end{pmatrix} \begin{pmatrix}
\bm Q_{3}^1 & \bm Q_{3}^2
\end{pmatrix}^{\top} ,$$
%$$\bm X^-(\delta) = \begin{pmatrix}
%\bm Q_1  & \bm Q_3
%\end{pmatrix} 
%\bigg[\begin{pmatrix}
%\bm \Lambda_1   & \bm 0\\
%\bm 0 & \bm \Lambda_3
%\end{pmatrix}  -
%\delta  
%\begin{pmatrix}
%\bm \Delta_1 &\bm 0\\
%\bm 0 & \bm \Delta_3
%\end{pmatrix}\bigg]
%\begin{pmatrix}
%\bm Q_1  & \bm Q_3
%\end{pmatrix} ^{\top} ,$$
where $\delta >0$ is small enough and the eigenvalues of $\bm X^+(\delta)$ and $\bm X^-(\delta)$ can be written as
\begin{align*}
\bm \lambda(\bm X^{\pm}(\delta))\in \Re^n := \begin{pmatrix}
\bm \lambda(\bm \Lambda_1^1), &\bm \lambda(\bm \Lambda_1^2 \pm \delta \bm \Delta_1), &\bm 0, & \bm \lambda(\bm \Lambda_3^1 \pm \delta \bm \Delta_3), & \bm \lambda(\bm \Lambda_3^2 )
\end{pmatrix}.
\end{align*}

First, let us focus on the nonnegative pair $((\bm x^*)^1, \bm\lambda^1)$. Given the first line in \eqref{case1} and the inequality \eqref{ineqs} that implies $\sum_{i\in [d_1-1]} \lambda_i \le  \sum_{ i \in [d_1-1]} x^*_i = \sum_{ i \in [s]} x^*_i$, according to  \Cref{claim1} in \Cref{prop:rank}, there exists a positive scalar $\underline{\delta}>0$ such that
\begin{align*}
(\bm x^*)^1 \succeq  
\begin{pmatrix}
\bm \lambda(\bm \Lambda_1^1), & \bm \lambda(\bm \Lambda_1^2 \pm \delta \bm \Delta_1)
\end{pmatrix},  \ \ \bm \Lambda_1^2 \pm \delta \bm \Delta_1 \in\S_{++}^{d_1-s+1}
,    \ \ \forall \delta \in (0, \underline{\delta}).
\end{align*}
Given $\tr(\bm \Delta_1)=0$ and  $\begin{pmatrix}
\bm \lambda(\bm \Lambda_1^1), & \bm \lambda(\bm \Lambda_1^2 \pm \delta \bm \Delta_1)
\end{pmatrix} \in \Re_{+}^{d_1-1}$, we can obtain
\begin{align*}
\begin{pmatrix}
(\bm x^*)^1, &(\bm x^*)^2
\end{pmatrix}
 \succeq  \begin{pmatrix}
\bm \lambda(\bm \Lambda_1^1), & \bm \lambda(\bm \Lambda_1^2 \pm \delta \bm \Delta_1), &\bm 0
\end{pmatrix}.
\end{align*}

In a similar vein, for the nonpositive pair $((\bm x^*)^3, \bm\lambda^3)$, motivated by the second line in \eqref{case1}, we can let  $j_0=j_1=s+n-\k-1 < j_2$ in \Cref{lem:ineq} and  $\lambda_{s+n-\k}=\cdots  = \lambda_{j_2}>\lambda_{j_2+1}$. Analogous to
the proof of \Cref{claim1} in \Cref{prop:rank}, along with the majorization result above, we can show that there exists a $\underline{\delta}>0$ such that for any $\delta \in (0, \underline{\delta})$,
\begin{align}\label{delta3}
\bm x^*:=
\begin{pmatrix}
(\bm x^*)^1, & (\bm x^*)^2, & (\bm x^*)^3
\end{pmatrix}
 \succeq  \begin{pmatrix}
 \bm \lambda(\bm \Lambda_1^1), &\bm \lambda(\bm \Lambda_1^2 \pm \delta \bm \Delta_1), &\bm 0, & \bm \lambda(\bm \Lambda_3^1 \pm \delta \bm \Delta_3), & \bm \lambda(\bm \Lambda_3^2 )
 \end{pmatrix} =: \bm \lambda (\bm X^{\pm} (\delta)).
\end{align}

Combining with the conditions on $\bm \Delta$ in \eqref{eq:linear} and the fact $\bm x^* \succeq \bm \lambda(\bm X^{\pm}(\delta))$ in \eqref{delta3}, we can conclude that $\bm X^+(\delta), \bm X^-(\delta)\in \Y^s$. Thus, the equation $\bm X^*=(\bm X^+(\delta) +\bm X^-(\delta))/2$  contradicts with the orthogonalization in \eqref{eq:perp}. We must have
\begin{align*}
&d+2 
\ge \frac{(d_1-s+1)(d_1-s+2)}{2} 
+ \frac{(s+n-\k-d_2+1)(s+n-\k-d_2+2)}{2}
\\&  =\frac{(d_1-s+1)(d_1-s+2)}{2} +
\frac{ (r-\k+s+1-d_1)(r-\k+s+2-d_1)}{2}  \ge \frac{(r-\k+2)(r-\k+4)}{4},
\end{align*}
where  the first equation is from $d_2-d_1=n-r$ and the last inequality is obtained by minimizing the convex function over $d_1$ that $d_1^* =s+(r-\k)/2$ at optimality. Therefore, the largest rank of all  points in $d$-dimensional face $F^d$ must not exceed $ \k + \sqrt{9+4d} -3$.

\item Suppose $d_1 \le s$, then $d_2-d_1=n-r$ implies $d_2 \le s+ n-r< s+n-\k$. Thus, the second line in \eqref{case1} still holds. 
Now we only focus on the nonpositive pair $((\bm x^*)^3, \bm \lambda^3) $ and construct a nonzero symmetric matrix $\bm \Delta_3 \in \S^{s+n-\k -d_2+1}$ as \eqref{eq:linear} with $\bm \Delta_1=\bm 0$. It follows that
\begin{align*}
d+1\ge 
\frac{(s+n-\k-d_2+1)(s+n-\k-d_2+2)}{2}
\ge  \frac{(r-\k+1)(r-\k+2)}{2}, 
\end{align*}
where the last  inequality is due to $d_2 \le s+n-r$. Hence, we conclude $r \le \k + \lfloor \sqrt{2 d+\frac{9}{4}}-\frac{3}{2}\rfloor$.

\item If $s+n-\k\le d_2 \le n$, then $d_2-d_1=n-r$ implies $d_1 \ge s+ r-\k>s$.
Thus, the first line in \eqref{case1} still holds and we can only focus on the nonnegative pair $((\bm x^*)^1, \bm \lambda^1)$. By constructing a nonzero symmetric matrix $\bm \Delta_1 \in \S^{d_1-s+1}$ as \eqref{eq:linear} with $\bm \Delta_3=\bm 0$, we have
\begin{align*}
d+1
&\ge \frac{(d_1-s+1)(d_1-s+2)}{2}
\ge  \frac{(r-\k+1)(r-\k+2)}{2}, 
\end{align*}
where the last inequality is due to $d_1\ge s+r-\k$. We thus obtain  $r \le \k + \lfloor \sqrt{2 {d}+{9}/{4}}-{3}/{2}\rfloor$.
\end{enumerate}
Taking the largest one among the rank bounds in parts (i), (ii), and (iii), we finally obtain a rank bound of $\k+\lfloor \sqrt{4d+9}\rfloor-3$.
   \qed
\end{proof}

\subsection{Proofs of Theorems \ref{them:sparsesign} and \ref{them:sparsenonsign}} \label{proof:sparseface}
The rank bound of \Cref{them:sparsesign,them:sparsenonsign} arise from describing the convex hull  $\conv(\X )$ and analyzing the rank of its faces, as shown below.

\begin{proposition}\label{prop:convsparse}
	Given a sparse domain set $\X $ in \eqref{set:sparse} and integer $\tilde k$ following \Cref{def:k}, we have
	\begin{enumerate}[(i)]
		\item If function $f(\cdot)$ in \eqref{set:sparse} is sign-invariant, set $\conv(\X )$ is closed and is equal to
		$\conv(\X ):= \Big\{\bm X\in  \S^n: \bm X =\Diag(\diag(\bm X)),  \exists  \bm x \in \Re_+^n, f(\bm x)\le 0,  x_1 \ge \cdots \ge x_n,  x_{\k+1}=0, 
		\|\bm X\|_{(\ell)} \le \sum_{i\in [\ell]} x_{i}, \forall \ell \in [\k-1], \|\bm X\|_* \le \sum_{i\in [\k]}x_i\Big\};
		$
		\item  Otherwise, $\conv(\X ) = \conv(\cup_{s \in[\k+1]} \hat{\Y}^s)$,
		where for each $s\in[\k+1]$,  set $\hat \Y^s$ is closed and equals
		$\hat\Y^{s}:= \big\{\bm X \in \S^n: \bm X =\Diag(\diag(\bm X)), \exists \bm x \in \Re^n,  f(\bm x)\le 0,  x_1\ge  \cdots \ge x_n, x_{s} = x_{s+n-\k-1}=0,
		\bm x \succeq \diag(\bm X)\big\}.
		$
	\end{enumerate} 
\end{proposition}
\begin{proof}
	When applied to the diagonal symmetric indefinite matrix space, Part (i) follows  \Cref{prop:nonsymconv}, and Part (ii) can be directly obtained by \Cref{prop:symconv}. \qed
\end{proof}

The set $\conv(\X)$ admits different expressions depending on whether the function $f(\cdot)$ in \eqref{set:sparse}  is
sign-invariant or not, as specified in \Cref{prop:convsparse}. Next, we derive the rank bounds of their faces for two expressions, respectively. 
\begin{restatable}{proposition}{thmsparseface}\label{them:sparseface}
	Given a sparse domain set $\X $ in \eqref{set:sparse},
	suppose  that  $F^d$ is a $d$-dimensional face in the set $\conv(\X )$ or in the set $\hat \Y^s$ for any $s\in [\k+1]$.
	Then all points in face $F^d$ always admit a rank at most $\tilde k + d$, where integer $\k \le k$ follows \Cref{def:k}.
\end{restatable}

\begin{proof}
The proof includes two parts.

\begin{enumerate}[(i)]
\item $F^d\subseteq \conv(\X )$.
Suppose that $f(\cdot)$ is sign-invariant. %Then the set  $\conv(\X )$  has a similar structure to the convex hull description in \Cref{prop:conv} and  allows the extension of analyses for \Cref{prop:rank}. 
Given that a matrix $\bm X^*\in F^d$ is diagonal and has a rank $r$ greater than $\k+d$, following the proof of \Cref{prop:rank}, we can construct a diagonal matrix $\bm \Delta \in \Re^{(r-\k+1)\times (r-\k+1)}$ satisfying the $(d+1)$ constraints \eqref{eq:new} whenever $d+1 < r-\k+1$ holds, as  any  diagonal matrix of size $(r-\k+1)\times (r-\k+1)$ can be mapped into a vector of length  $r-\k+1$.   The existence of such a diagonal matrix $\bm \Delta$  can form a contradiction as shown in \Cref{prop:rank}, implying that  $d+1 \ge r-\k+1$ must hold.

\item $F^d\subseteq \hat \Y^s$ for any $s\in [\k +1]$. Given that a matrix $\bm X^* \in F^d$ is  diagonal and is of rank $r$, following the similar proof of \Cref{them:symface}, we can construct a diagonal matrix $\bm \Delta$ satisfying $d+2$ equality constraints in \eqref{eq:linear}; that is,
\[d+2 \le d_1-s+1 + s+n-\k-d_2 +1 = r-\k +2, \]
where the second equality is due to $d_2-d_1=n-r$. This completes the proof. \qed 
\end{enumerate} 
\end{proof}

For a sparse domain set $\X$ in \eqref{set:sparse}, the two sets $\conv(\X)$, $\{\hat{\Y}^s\}_{s\in[\k+1]}$ provide us with the relaxations \ref{eq_rel_rank}, \ref{eq_rel_rank1}, respectively. Applying \Cref{lem:face} to either the set $\conv(\X)$ or the set $\hat{\Y}^s$ for each $s\in [\k+1]$ and using the results in \Cref{them:sparseface}, one can readily derive the rank bounds of either \ref{eq_rel_rank} or \ref{eq_rel_rank1}, as shown in the theorems below.

\themsparsesign*
\begin{proof}
Plugging the rank bound in \Cref{them:sparseface} of the set $\conv(\X)$ described in \Cref{prop:convsparse}, the rest of the proof is nearly
identical to that in \Cref{them:rank} and thus is omitted.  
\end{proof}

\themsparsenonsign*
\begin{proof}
According to  \Cref{them:sparseface}, any $d$-dimensional face of set $\hat \Y^s$ satisfies the rank-$(\k+d)$ for all $s\in [\k+1]$. Using \Cref{lem:face} and following the analysis of \Cref{them:rank}, we can show that for each inner minimization of  \ref{eq_rel_rank1}, all the extreme points in the intersection of  set $\hat \Y^s$ and $\tilde m$ linearly independent inequalities admit a rank-$(\k+\tilde m)$ bound. \qed 
\end{proof}
\subsection{Proof of Corollary \ref{cor:sparse}}\label{proof:sparse}
\corsparse*
\begin{proof}
Let $(\hat{z}, \hat{\bm y}, \hat{\bm X})$ denote an optimal solution to the \ref{pc_sparse}, where $\V_{\rel} \ge 0$ must hold.  Then given an optimal solution $\hat{z}$,  any matrix $\bm X$ in the following compact set $\T$ is optimal to \ref{pc_sparse} 
\begin{align*}
\T:=\left\{\bm X\in \conv(\X ):\bm A \diag(\bm X)=  \bm A \diag(\hat{\bm X}) \right\},
\end{align*}
where given $\hat{z}$, we let $ \X :=\{\bm X \in \S^n: \bm X =\Diag(\diag(\bm X)), \rank(\bm X) \le k, \|\bm X\|_F^2 \le \hat z \}$.

For each $i\in [m]$, let $\bm a_i \in \Re^n$ denote the $i$th row vector of matrix $\bm A$. In this way, the equation system $ \bm A \diag(\hat{\bm X}) =\bm A \diag(\bm X)$ can be reformulated by the standard form below
\begin{align*}
\left \langle \bm A_i, \hat{\bm X} \right \rangle = b_i, \forall i\in [m],
\end{align*}
where  for each $i\in [m]$, $\bm A_i = \Diag(\bm a_i)$ and $b_i = \bm a_i^{\top} \diag(\hat{\bm X})$. The rank of the matrix $\bm A$ represents the number of linearly independent row vectors in $\{\bm a_i\}_{i\in [m]}$ and so does for $\{\bm A_i\}_{i\in [m]}$. 

Since $\k \le k$ and there are $\rank(\bm A)$ linearly independent equations in the compact set $\T$, according to \Cref{them:sparsesign}, one can find an alternative  extreme point with rank at most $k+\rank(\bm A)$.
\qed
\end{proof}

\subsection{Proof of Proposition \ref{prop:conver}} \label{proof:conver} 
\propconver*
\begin{proof}
Notice that $\V_{\rel}\leq \langle \bm A_0, \bm  X^*\rangle$ holds, since the \ref{eq:rmp} only optimizes over a subset of $\conv(\X)$. To prove the other bound, we see that 
at termination, the optimal value of the \ref{eq_pp} satisfies $\V_{\rm P} \le \nu^*+\epsilon$. 
The Lagrangian dual of original \ref{eq_rel_rank} can be rewritten as
\[\V_{\rel} :=\min_{\bm{X}\in \conv(\X)}\Bigg\{ \left\langle \bm A_0,\bm X\right\rangle+\sup_{\bm \mu^l \in \Re_+^m, \bm \mu^u\in \Re_+^m} \bigg\{  \bigg \langle  \sum_{i\in [m]} (\mu^u_i - \mu^l_i) \bm A_i, \bm X\bigg \rangle  + (\bm b^l)^{\top} \bm \mu^l- (\bm b^u)^{\top} \bm \mu^u\bigg\}\Bigg\}.\]
By fixing the inner dual variables to be $((\bm \mu^l)^*, (\bm \mu^u)^*)$ which correspond to the \ref{eq:rmp}, we have
\begin{align*}
\V_{\rel}&\geq \min_{\bm{X}\in \conv(\X)}\Bigg\{ \left\langle \bm A_0,\bm X\right\rangle+ \bigg \langle  \sum_{i\in [m]} ((\mu^u_i )^*- (\mu^l_i)^*) \bm A_i, \bm X\bigg \rangle  + (\bm b^l)^{\top} (\bm \mu^l)^*- (\bm b^u)^{\top} \bm (\mu^u)^*\Bigg\}\\
&=(\bm b^l)^{\top} (\bm \mu^l)^*- (\bm b^u)^{\top} \bm (\mu^u)^*-\V_{\rm P} \ge (\bm b^l)^{\top} (\bm \mu^l)^*- (\bm b^u)^{\top} \bm (\mu^u)^*-\nu^*-\epsilon= \langle \bm A_0, \bm  X^* \rangle-\epsilon,
\end{align*}
where the second inequality is due to $\V_{\rm P} \le \nu^*+\epsilon$ and the last equality is from the strong duality of \ref{eq:rmp}. This completes the proof.
%According to the dual problem \eqref{eq:dual}, we have
%\[\langle \bm A, \hat{\bm  X}\rangle  -\epsilon = -\nu^* + (\bm b^l)^{\top} (\bm \mu^l)^* - (\bm b^u)^{\top} (\bm \mu^u)^* -\epsilon  < -\V_{\rm P} + (\bm b^l)^{\top} (\bm \mu^l)^* - (\bm b^u)^{\top} (\bm \mu^u)^* \le \V_{\rel}, \]
%where the last inequality is due to the weak duality and the fact $(\V_{\rm P}, (\bm \mu^l)^* , (\bm \mu^u)^* )$ forms a feasible solution to the dual problem of the original LSOP-R \eqref{eq_rel_dwr}. 
\qed
\end{proof}

 \subsection{Proof of Theorem \ref{them:ppsol}}\label{proof:ppsol}
 \thmppsol*
 \begin{proof}
The proof can be split into four parts. %depending on different settings of set $\Q$.
\\
\textbf{Part (i).} When $\Q:=\S_{+}^n$,  we represent the  matrix variable $\bm X \in \S_+^n$ by its eigen-decomposition $\bm X=\sum_{i\in [n]}\lambda_i \bm q_i \bm q_i^{\top}$ with eigenvalues  sorted in descending order. Hence,  the \ref{eq_sep1}  reduces to
\begin{align}\label{eq_sep3}
\text{(PP)} \quad \V_{\rm P}:=\max_{\bm \lambda \in \Re_+^n} \max_{\bm Q \in \Re^{n\times n}} \bigg\{\sum_{ i \in [n]}\lambda_i \bm q_i^{\top}\tilde{\bm A} \bm q_i: \bm Q^{\top} \bm Q= \bm I_n, \lambda_1 \ge\cdots \ge \lambda_n, \|\bm \lambda\|_0\le k,  f(\bm \lambda)\le 0\bigg\},
\end{align}
where $\|\bm \lambda\|_0\le k$ is equivalent to the rank constraint $\rank(\bm X) \le k$. Next we introduce a key claim. 
\begin{claim}\label{claim2}
$\max_{\bm Q \in \Re^{n\times \ell}, \bm Q^{\top} \bm Q =\bm I_{\ell}} \sum_{i\in [\ell]} \bm q_{i}^{\top} \tilde{\bm A} \bm q_i =\sum_{i\in [\ell]} \beta_i$ for each $\ell \in [n]$.
\end{claim}
\begin{proof}
Suppose $\bm Y = \bm Q \bm Q^{\top}$,  then we have 
\begin{align*}
\max_{\bm Q \in \Re^{n\times \ell}, \bm Q^{\top} \bm Q =\bm I_{\ell}} \sum_{i\in [\ell]} \bm q_{i}^{\top} \tilde{\bm A} \bm q_i &\le \max_{\bm Y\in \S_{+}^n }\left\{ \langle \tilde{\bm A}, \bm Y\rangle: \bm Y \preceq \bm I_n, \tr(\bm Y)= k \right\} \\
   & =  \max_{\bm Y\in \S_{+}^n }\left\{ \langle \tilde{\bm A}-\beta_n \bm I_n, \bm Y\rangle + k\beta_n : \bm Y \preceq \bm I_n, \tr(\bm Y)= k \right\}= \sum_{i\in [k]} \beta_i, 
\end{align*}
where given the positive semidefinite matrix $\tilde{\bm A}-\beta_n \bm I_n$, the last equation is due to lemma 2 in \cite{li2021beyond} and  the inequality is attainable.
\qedA
\end{proof}

Next we introduce the auxiliary variables $\bm y \in \Re^n$ where $y_i = \bm q_{i}^{\top} \tilde{\bm A} \bm q_i $  for each $i\in [n]$. Given $\lambda_1 \ge \cdots  \ge \lambda_n$, to maximize $\bm \lambda^{\top} \bm y$, we must have $y_1 \ge \cdots \ge y_n$ at optimality due to the rearrangement inequality \citep{hardy1952inequalities}.  Since \Cref{claim2} implies $\bm \beta \succeq \bm y$,  we propose a relaxation problem of the inner maximization over  matrix $\bm Q$ in problem \eqref{eq_sep3} that provides an upper bound
\begin{align*}
\max_{\bm y \in \Re^{n} }\bigg\{\sum_{ i \in [n]}\lambda_i y_i : \sum_{ i \in [\ell]} y_i \le \sum_{ i \in [\ell]}\beta_i, \forall \ell \in [n-1],\sum_{ i \in [n]} y_i = \sum_{ i \in [n]}\beta_i, y_1 \ge \cdots\ge y_{n}\bigg\},
\end{align*}
which admits an optimal solution $y^*_i=\beta_i$ for each $i\in [n]$ with optimal value $\sum_{ i \in [n]} \lambda_i \beta_i$ due to $\beta_1 \ge \cdots\ge \beta_n$ and $\lambda_1\ge \cdots \ge \lambda_n$. Besides,   the inner maximization over  matrix $\bm Q$ in problem \eqref{eq_sep3} can also achieve the objective value $\sum_{ i \in [n]} \lambda_i \beta_i$ when $\bm Q^* = \bm U$; thus, this solution is optimal. Plugging the optimal solution $\bm Q^* = \bm U$ into the problem \ref{eq_sep3}, we obtain
\begin{align}\label{eq_lambda}
\max_{\bm \lambda \in \Re_+^n} \left\{\bm \lambda^{\top} \bm \beta: \lambda_1 \ge\cdots \ge \lambda_n, \|\bm \lambda\|_0\le k,  f(\bm \lambda)\le 0\right\}.
\end{align}

Since $\beta_1 \ge \cdots \ge \beta_n$ and function $ f(\bm \lambda)$ is permutation-invariant with $\bm \lambda$, according to the rearrangement inequality \citep{hardy1952inequalities},  it is evident  that at optimality of problem \eqref{eq_lambda}, we must have $\lambda_1 \ge\cdots \ge \lambda_k \ge \lambda_{k+1}=\cdots  = \lambda_n=0$.
Therefore, problem \eqref{eq_lambda} can be reformulated by a convex optimization below
\begin{align*}
\max_{\bm \lambda \in \Re_+^n} \left\{\bm \lambda^{\top} \bm \beta:\lambda_i =0, \forall i\in [k+1, n],  f(\bm \lambda)\le 0 \right\},
\end{align*}
where any optimal solution to  problem above must satisfy $\lambda_1 \ge\cdots \ge \lambda_k \ge \lambda_{k+1}=\cdots  = \lambda_n=0$.

%{Introducing the Lagrangian multipliers $\bm \mu \in \Re^n$ and $\bm \beta \in \Re^t$ corresponding to the linear and spectral constraints in  problem above, respectively, the resulting Lagrangian dual problem  can be written as below
%\begin{align*}
%&\min_{\bm \mu\in \Re^n, \bm \theta \in \Re_+^n} \max_{\bm \lambda \in \Re_+^n} \bigg\{\bm \lambda^{\top} (\bm \beta+\bm \mu)-\sum_{j\in [t]}\theta_j f_j(\bm \lambda): \mu_i \ge 0, \forall i\in [k] \bigg\}.
%\end{align*}}
\noindent\textbf{Part (ii).} If matrix $\bm X\in \Q:=\Re^{n\times p}$ is non-symmetric and $n\le p$, then we introduce the singular value decomposition of $\bm X=\sum_{i\in [n]} \lambda_i \bm q_i \bm p_i^{\top}$ to reformulate the \ref{eq_sep1} as
\begin{align}\label{eq_nsym}
\max_{\bm \lambda \in \Re_+^{n}} \max_{\bm Q \in \Re^{n\times n}, \bm P\in \Re^{p\times n}} \bigg\{\sum_{ i \in [n]}\lambda_i \bm q_i^{\top}\tilde{\bm A}\bm p_i: \bm Q^{\top} \bm Q= \bm P^{\top} \bm P= \bm I_{n}, \lambda_1\ge \cdots \ge \lambda_n, \|\bm \lambda\|_0\le k \bigg\}.
\end{align}
Following the proof of \Cref{claim2} and lemma 1 in \cite{li2021beyond}, we obtain the following result:
\begin{claim} \label{claim3}%The following identity holds:
$\max_{\bm Q \in \Re^{n\times \ell}, \bm P \in \Re^{p\times \ell}, \bm P^{\top} \bm P = \bm Q^{\top} \bm Q =\bm I_{\ell}} \sum_{i\in [\ell]} \bm q_{i}^{\top} \tilde{\bm A} \bm p_i =\sum_{i\in [\ell]} \beta_i$ for each $\ell \in [n]$.
\end{claim}
Thus, according to \Cref{claim3},  we can equivalently reformulate the inner maximization by optimizing over the variables $(\bm Q, \bm P)$ of problem \eqref{eq_nsym} as
\begin{align*}
\max_{\bm y \in \Re^{n} }\bigg\{\sum_{ i \in [n]}\lambda_i y_i : \sum_{ i \in [\ell]} y_i \le \sum_{ i \in [\ell]}\beta_i, \forall \ell \in [n], y_1 \ge \cdots\ge y_{n}\bigg\}.
\end{align*}

Then the remaining proof simply follows from Part (i).
   
\noindent\textbf{Part (iii).} When $\Q:=\S^n$, 
following the analysis of Part (i), we can arrive at problem \eqref{eq_lambda} but now the eigenvalue variable $\bm \lambda\in \Re^n$ can be negative  under the symmetric case, i.e., $\bm X\in \S^n$. 
%Suppose that the objective coefficients in  problem \eqref{eq_lambda}  are sorted in a descending order, i.e., $\beta_1 \ge \cdots \ge \beta_n$, then we have
%\[ \bm \beta^{\top} \bm \lambda \le |\bm \beta|^{\top} |\bm \lambda| \]
Analogous to \Cref{prop:symconv}, the resulting feasible set   can be decomposed into $k+1$ subsets, i.e., 
\[\{\bm \lambda\in \Re^n: \lambda_1 \ge \cdots \ge \lambda_n, \|\bm\lambda\|_0 \le k, f(\bm \lambda)\le 0 \} = \cup_{s\in [k+1]} \T^s, \] 
where for each $s\in [k+1]$, $\T^s:= \{\bm \lambda\in \Re^n: \lambda_1 \ge \cdots \ge \lambda_n, \lambda_s =\lambda_{s+n- k-1}=0, f_j(\bm \lambda)\le 0, \forall j\in [t] \}$.  Replacing the feasible set of problem \eqref{eq_lambda}  by $\cup_{s\in [k+1]} \T^s$, we complete the proof.

\noindent\textbf{Part (iv).} For the case of $\Q:=\S^n$, suppose that $\bm X=\sum_{i\in [n]}\lambda_i \bm q_i \bm q_i^{\top}$ denote the eigen-decomposition of matrix variable $\bm X$. 
Since the rank and spectral functions are sign-invariant,  now changing the sign of variable $\bm \lambda$ allows us to reduce problem \eqref{eq_sep3} in Part (i) to
\begin{align*}
\text{(PP)} \quad \V_{\rm P}:=\max_{\bm \lambda \in \Re^n} \max_{\bm Q \in \Re^{n\times n}} \bigg\{\sum_{ i \in [n]}|\lambda_i| |\bm q_i^{\top}\tilde{\bm A}\bm q_i|: \bm Q^{\top} \bm Q= \bm I_n, |\lambda_1| \ge\cdots \ge |\lambda_n|, \|\bm \lambda\|_0\le k,  f(|\bm \lambda|)\le 0 \bigg\},
\end{align*}
Using the facts that $\max_{\bm Q \in \Re^{n\times \ell}, \bm Q^{\top} \bm Q =\bm I_{\ell}} \sum_{i\in [\ell]} |\bm q_{i}^{\top} \tilde{\bm A} \bm q_i| =\sum_{i\in [\ell]} |\beta_i|$ for each $\ell \in [n]$ and $|\beta_1|\ge \cdots |\beta_n|$,  the analysis of in Part (i) can be readily extended to derive $\bm q_i = \bm u_i$ for each $i\in [n]$ and obtain
\begin{align*}
\bm \lambda^*:=&\argmax_{\bm \lambda \in \Re^n} \left\{|\bm \lambda|^{\top} |\bm \beta|: |\lambda_i| =0, \forall i\in [k+1, n],  f(|\bm \lambda|)\le 0 \right\}.
\end{align*}
where the equality is due to the sign-invariant properties of spectral functions $f(\cdot)$. %and for each $i\in [k]$, we define $\sign(\beta_i)=1$ if $\beta_i\ge 0$ and $-1$, otherwise.
\qed
%First,  the feasible set over $\bm \lambda$  can be written as 
\end{proof}

\subsection{Proof of Corollary \ref{cor:closed}} \label{proof:closed}
\corclosed*
\begin{proof}
\textbf{Part (i).} Given $\Q:=\S_+^n$ and the $\ell$-norm function $f(\cdot)=\|\cdot \|_{\ell}$,  we let $r$ denote the largest integer of matrix $\tilde{\bm A}$ such that $\beta_r \ge 0$. Then, by letting $s:=\min\{k,r\}$,
the maximization problem \eqref{eq:ppc1} over $\bm \lambda$ in Part (i) of \Cref{them:ppsol} now reduces to
\begin{align*}
 \max_{\bm \lambda \in \Re_+^n} \left\{\bm \lambda^{\top} \bm \beta: \lambda_i=0,\forall i\in[s+1,n], \|\bm \lambda\|_{\ell}\le c \right\} \le c\sqrt[q]{{\sum_{j\in [s]} \beta_j^{q}}} = c\sqrt[q]{{\sum_{j\in [k]} (\beta_j)_+^{q}}}, 
\end{align*}
 where the inequality is due to Holder's inequality and the equation is from the definition $(\beta_j)_+=\max\{0,\beta_j\}$. The inequality can reach equation when there is $\alpha>0$ such that $\lambda_i^{\ell} =\alpha(\beta_i)_+^{q}$ for all $i\in [n]$, which enables us to construct optimal solution $\bm \lambda^*$ below
 \begin{align*}
     \lambda_i^*= \begin{cases}
      c \sqrt[\ell]{\frac{(\beta_i)_+^{q}}{\sum_{j\in [k]} (\beta_j)_+^{q}}}, & \forall i\in [k];\\
         0, & \forall i\in [k+1,n].
     \end{cases}.
 \end{align*}

\textbf{Part (ii).} Given $\Q:=\Re^{n\times p}$, the proof follows that of Part (i) except replacing the eigenvalues by singular values of matrix $\tilde{\bm A}$.

\textbf{Part (iii).} Given $\Q:=\S^n$ and sign-invariant $\ell$-norm function  $f(\cdot):=\|\cdot \|_{\ell}$, matrix $\tilde{\bm A}$ is symmetric and we let its eigenvalues satisfy $|\beta_1|\ge \cdots  \ge |\beta_n|$. Since the sign of variable $\bm \lambda$ can be arbitrary,
the maximization problem \eqref{eq:ppc4} over $\bm \lambda$ in Part (i) of \Cref{them:ppsol} now reduces to
\begin{align*}
\max_{\bm \lambda \in \Re^n} \left\{|\bm \lambda|^{\top} |\bm \beta|: |\lambda_i|=0,\forall i\in[k+1,n], \|\bm \lambda\|_{\ell}\le c \right\} \le c\sqrt[q]{{\sum_{j\in [k]} |\beta_j|^{q}}},
\end{align*}
 where similar to Part (i), the equality is attainable when $\lambda^*_i = \sign(\beta_i) c \sqrt[\ell]{\frac{|\beta_i|^{q}}{\sum_{j\in [k]} |\beta_j|^{q}}} $ for all $i\in [k]$ and $\lambda^*_i =0$ for all $i\in [k+1, n]$.
 \qed
\end{proof}

\subsection{Proof of Theorem \ref{them:reduce}}\label{proof:reduce}
\thmreduce*
\begin{proof} We split the proof into two parts.\par
\noindent \textbf{Part I.} we show that the  \Cref{algo_reduce} always terminates in different matrix spaces. %can iteratively reduce the rank of the solution $\bm X^*$ until converging to zero $\delta^*$.
\begin{enumerate}[(a)]
    \item Let us begin with the case of $\Q:=\S_+^n$. At each iteration of \Cref{algo_reduce}, for a pair of current and new solutions $(\bm X^*, \bm X(\delta^*))$, they satisfy  $\bm X(\delta^*)=\bm X^*+\delta^*\bm Y$, where the direction $\bm Y=\bm Q_2 \bm \Delta\bm Q_2^{\top}$ is defined in \eqref{eq_direct1}. The fact $\tr(\bm Y)=0$ implies that $\tr(\bm X^*)= \tr(\bm X(\delta^*))$.
    We let integer $r:=\rank(\bm X^*)$. Since  $\bm Y$  only perturbs $(r-\k+1)$ smallest positive  eigenvalues of matrix $\bm X^*$ (i.e., $\lambda_{\k}^*,\cdots,\lambda^*_r$) and the perturbed eigenvalues remain nonnegative, the first $(\k-1)$ largest eigenvalues of $\bm X^*$ are also eigenvalues of new matrix $\bm X(\delta^*)$. Therefore, we must have that  $\rank(\bm X(\delta^*)) \le r$ and $\|\bm X^*\|_{(k)} \le \|\bm X(\delta^*)\|_{(k)} $, implying that
\begin{equation}\label{eq_inequality}
\tr(\bm X^*)-\|\bm X^*\|_{(k)} = \sum_{i\in [\k, r]}\lambda_i^* \ge \tr(\bm X(\delta^*))-\|\bm X(\delta^*)\|_{(k)}.
\end{equation}
which means that the sum of $(r-\k+1)$ smallest positive eigenvalues in the new solution $\bm X(\delta^*)$ must not exceed that of $\bm X^*$. 

If the inequality \eqref{eq_inequality} reaches equality and $\rank(\bm X(\delta^*))=r$, then the $(r-\k+1)$ smallest positive eigenvalues of $\bm X^*(\delta^*)$ are exactly those of $\bm  \Lambda_2+\delta^* \bm \Delta \in \S_+^{r-\k+1}$, where $\tr(\bm \Lambda_2+\delta^* \bm \Delta)=\sum_{i\in [\k,r]}\lambda_i^*$. Then, following the proof of \Cref{prop:rank}, we can show that %$\bm X^*$ and $\bm X(\delta^*)$ share the same eigenvectors corresponding to the  $(r-\k+1)$ smallest  eigenvalues which determines the direction $\bm Y$ in \eqref{eq_direct1}. That is, 
there exists a positive $\hat{\delta}>0$ such that $\bm X(\delta^*)+ \hat{\delta}\bm Y$ is also feasible to Step 9 of \Cref{algo_reduce}, which contradicts to the optimality of $\delta^*$. %, which proves $\rank(\bm X(\delta^*)) \le r -1$ at equality. 
Therefore,   for the new solution $\bm X(\delta^*)$, either the sum of its $(r-\k+1)$ smallest nonzero eigenvalues strictly decreases (i.e., the inequality \eqref{eq_inequality} is strict) or the rank strictly reduces (i.e., $\rank(\bm X(\delta^*))<r$). Since both values $ \sum_{i\in [\k, r]}\lambda_i^*$ and $r$ are finite, 
either case ensures that the solution sequence is monotone and thus  \Cref{algo_reduce} always terminates with $\delta^*=0$ at Step 7.

% Thus, at the next iteration, for the new solution $\bm X^*:=\bm X(\delta^*)$, either its rank is no larger than $r-1$ or the sum of its $(r-\k+1)$ smallest eigenvalues strictly decreases. 

\item  Similar to Part (a),  we can show that \Cref{algo_reduce}  converges in $\Q=\Re^{n\times p}$ as the direction $\bm Y$ in \eqref{eq_direct2} has an impact on only $(r-\k+1)$ smallest positive singular values of matrix $\bm X^*$.

\item In contrast to the case of $\Q:=\S_+^n$ or $\Q:=\Re^{n\times p}$, the eigenvectors $ \begin{pmatrix}
\bm Q_1^2  & \bm Q_3^1
\end{pmatrix}$ that determine the direction $\bm Y$  in \eqref{eq_direct3} correspond to the $(d_1-s^*+1)$ smallest positive eigenvalues and $(s^*+n-\k-d_2+1)$ largest negative eigenvalues, and then the convergence analysis of Part (a) can be readily extended by leveraging those eigenvalues.
\end{enumerate}

\noindent \textbf{Part II.}  Since \Cref{algo_reduce} starts with an $\epsilon$-optimal solution, and we always find a direction $\bm Y$ along which  the objective value does not increase, the output solution is at least $\epsilon$-optimal.

 We show by contradiction that the output solution $\bm X^*$ of \Cref{algo_reduce} must satisfy the rank bounds.
We discuss different matrix spaces. 
\begin{enumerate}[(i)]
    \item $\Q:=\S_+^n$. At termination of \Cref{algo_reduce}, we have $\delta^*=0$ and obtain solution $\bm X^*$. Suppose $\rank(\bm X^*)>\k+ \lfloor\sqrt{2\tilde m + 9/4}-3/2\rfloor$, i.e., $\tilde m +1 < (r-\k+1)(r-\k+2)/2$. Then, there exists a nonzero matrix $\bm \Delta$ to form a nonzero direction $\bm Y$  in \eqref{eq_direct1} and  following \Cref{claim1} in the proof of \Cref{prop:rank},  we can find a nonzero $\underline \delta >0$ such that the eigenvalue vector $\bm X^* + \underline \delta \bm Y$ is majorized by $\bm x^*$ and $ \lambda_{\min}( \bm  \Lambda_2+  \underline \delta \bm \Delta ) \ge  0 $. This contradicts with the maximum value of $\delta$  being zero, i.e., $\delta^*=0$.  We thus complete the proof.

\item $\Q:=\Re^{n\times p}$. Similar to Part (i) with $\Q=\S_+^n$, when $\delta^*=0$, we can show that $\rank(\bm X^*) \le  \k+ \lfloor\sqrt{2\tilde m + 9/4}-3/2\rfloor$ using the proof of \Cref{them:nonface}.

\item $\Q:=\S^n$. According to the proof of \Cref{them:symface}, similar to Part (i), 
 whenever  $\rank(\bm X^*)>\k+ \lfloor\sqrt{4\tilde m + 9}-3\rfloor$, we can construct a nonzero matrix $\bm Y$ in \eqref{eq_direct3} to move the solution $\bm X^*$ by a nonzero distance $\delta^*$ along the direction $\bm Y$, which contradicts the condition $\delta^*=0$.\qed
\end{enumerate}
\end{proof}

\subsection{Proof of Corollary \ref{cor:converge}} \label{proof:converge}
\corconverge*
\begin{proof}
Given $\Q:=\S_+^n$, according to the convergence analysis of \Cref{algo_reduce} in the proof of \Cref{them:reduce},  at each iteration, we have that either the sum of the $(r-\k+1)$ smallest positive eigenvalues strictly decreases (i.e., the inequality \eqref{eq_inequality} is strict) or the rank strictly reduces (i.e., $\rank(\bm X(\delta^*))<r$).  If $\k=1$, then the sum of its $r$ smallest nonzero eigenvalues  is exactly the trace of new solution $\bm X(\delta^*)$ that always stays the same as matrix $\bm X^*$. Therefore,  the rank must strictly reduce at each iteration, i.e., $\rank(\bm X(\delta^*)) < \rank(\bm X^*)$. The similar analysis follows for the non-symmetric or symmetric indefinite matrix space. \qed
\end{proof}
\section{Tightness of Rank Bounds for \ref{eq_rel_rank}}\label{proof:tight}
This section presents the worst-case examples of \ref{eq_rel_rank} in which the rank bounds  in \Cref{them:symsign,them:nonsym,them:rank,them:sparsesign,them:sparsenonsign}  are tight, i.e., the rank bounds are attainable by these worst-case instances.

Let us begin with the following key lemma.
\begin{lemma}\label{lem:nuc}
For any matrix $\bm C\in \Re^{n\times n}$, the following results  hold for  the nuclear norm of $\bm C$
\begin{enumerate}[(i)]
    \item $\|\bm C\|_* \ge \sum_{i\in [n]} |C_{ii}|$; and
    \item $\|\bm C\|_* \ge \sum_{i\in [n]\setminus S} |C_{ii}| +  \|\bm C_{S,S}\|_*$ for all $S\subseteq [n]$ and $|S|=2$.
\end{enumerate}
Note that for any set $S\subseteq [n]$,   $\bm C_{S,S}$ denotes the submatrix of $\bm C$ with rows and columns in set $S$. 
\end{lemma}
\begin{proof}
    According to \cite{li2021beyond}[lemma 1], the nuclear norm of matrix $\bm C$ is equal to
    \begin{align}\label{eq:nuc}
       \|\bm C\|_* := \max_{\bm U \in \Re^{n\times n}} \left\{\langle\bm C, \bm U\rangle: \|\bm U\|_2 \le 1, \|\bm U\|_* \le n  \right\}.
    \end{align}
    % where $\|\bm U\|_*\leq n$ is a trivial constraint and is added for the purpose of presentation.
Since $\bm U =\diag(\sign(C_{11}), \cdots  \sign(C_{nn}))$ is a feasible solution to problem \eqref{eq:nuc} and leads to the objective value $\sum_{i\in [n]} |C_{ii}|$, we arrive at the inequality in Part (i).

In addition, for any subset $S\subseteq [n]$, the problem \eqref{eq:nuc} is lower bounded by
\begin{align*}
\|\bm C\|_* \ge \|\bm C_{S,S}\|_* +  \|\bm C_{[n]\setminus S,[n]\setminus S}\|_*:=
&\max_{\bm U \in \Re^{2\times 2}} \left\{\langle\bm C_{S,S}, \bm U\rangle: \|\bm U\|_2 \le 1, \|\bm U\|_* \le 2  \right\} \\ &+ \max_{\bm U \in \Re^{(n-2)\times (n-2)}} \left\{\langle\bm C_{[n]\setminus S,[n]\setminus S}, \bm U\rangle: \|\bm U\|_2 \le 1, \|\bm U\|_* \le n-2  \right\}, 
\end{align*}
where the inequality stems from the fact that the optimal solutions from both maximization problems above always compose a feasible solution to  problem \eqref{eq:nuc}. By applying Part (i) to the nuclear norm of $\bm C_{[n]\setminus S,[n]\setminus S}$,  we thus prove Part (ii). \qed
\end{proof}

Next, we are ready to show that the proposed rank bounds for \ref{eq_rel_rank} are tight.
\begin{lemma}
Given an integer $\k \le k$ following \Cref{def:k}, suppose $r:=\k + \lfloor \sqrt{2\tilde m + 9/4}-3/2\rfloor$. Then we have that
\begin{enumerate}[(i)]
\item   Given $\Q:=\S_+^n$ in \eqref{eq_set}, there exists an \ref{eq_rel_rank} example that contains a rank-$r$ extreme point;
    \item Given $\Q:=\Re^{n\times p}$ in \eqref{eq_set}, there exists  an \ref{eq_rel_rank} example that contains a rank-$r$ extreme point;
    \item Given $\Q:=\S^n$ in \eqref{eq_set}, suppose that function $f(\cdot)$ in the domain set $\X$ is sign-invariant. Then there exists an \ref{eq_rel_rank} example that contains a rank-$r$ extreme point.
\end{enumerate}

\end{lemma}
\begin{proof}
First, we have $\k\le r\le n$.
When $r=\k$, the results trivially hold. The following proof that focuses on the case of $r\ge \k+1$ splits into three parts.

\noindent\textbf{Part (i).} When $\Q:=\S_+^n$, let us consider the following example.
\begin{example}\label{eg:psdtight}
Suppose the domain set $\X:=\{\bm X \in \S_+^n: \|\bm X\|_2 \le 1, \rank(\bm X)\le k\}$ and $m =(r-k+1)(r-k+2)/2-1$ linear equations in \ref{eq_rank}: 
\begin{align*}
X_{ii} = \frac{1}{r-k+1}, \forall i \in [r-k],  \ \  X_{ij} =0, \forall i, j \in [r-k+1]\times [r-k+1], i< j.
\end{align*}
\end{example}

In \Cref{eg:psdtight}'s domain set $\X$, we have that $\k =k$, $m=\tilde m$, and  $r=\k+\lfloor\sqrt{2\tilde{m}+9/4}-3/2\rfloor$. It has been proven by \cite{tantipongpipat2019multi}[appendix B] that there exists a rank-$r$ extreme point of the feasible set of corresponding \ref{eq_rel_rank}.
% \begin{align*}
% \bm X^*:= \begin{pmatrix}
% \frac{1}{r-k+1} \bm I_{r-k+1}  & \bm 0_{r-k+1, k-1} & \bm 0_{r-k+1,n-r} \\
% \bm 0_{k-1, r-k+1} & \bm I_{k-1} & \bm 0_{k-1, n-r}\\
% \bm 0_{n-r, r-k+1} & \bm 0_{n-r, k-1} & \bm 0_{n-r, n-r}\\
% \end{pmatrix}.
% \end{align*}

\noindent\textbf{Part (ii).} When $\Q:=\Re^{n\times p}$,
let us consider the following example.  
\begin{example}\label{eg:tight}
Suppose the domain set $\X:=\{\bm X \in \Re^{n\times p}: \|\bm X\|_2 \le 1, \rank(\bm X)\le k\}$ with $n=p$ and $m =(r-k+1)(r-k+2)/2-1$ linear equations in \ref{eq_rank}: 
\begin{align*}
X_{ii} = \frac{1}{r-k+1}, \forall i \in [r-k],  \ \  X_{ij} =0, \forall i, j \in [r-k+1]\times [r-k+1], i< j.
\end{align*}
\end{example}

In \Cref{eg:tight}'s domain set $\X$, we also have $\k =k$, $m=\tilde m$, and  $r=\k+\lfloor\sqrt{2\tilde{m}+9/4}-3/2\rfloor$. It is recognized in \cite{li2022exactness}[lemma 5] that $\conv(\X)=\{\bm X \in \S_{+}^n:  \| \bm X\|_2 \le 1, \|\bm X\|_*\le k \}$. Thus, in this example, the resulting \ref{eq_rel_rank} admits the following feasible set 
\begin{align*}
\T=\left\{\bm X\in \Re^{n\times n}: \|\bm X\|_2 \le 1, \|\bm X\|_* \le k, X_{ii} = \frac{1}{r-k+1}, \forall i \in [r-k],   X_{ij} =0, \forall 1\le i< j \le r-k+1\right\}.
\end{align*}

Then,  it suffices to show a rank-$r$ extreme point in the feasible set $\T$ above. Specifically, let us construct a rank-$r$ matrix $\bm X^* \in \Re^{n\times n}$ below
\begin{align*}
\bm X^*:= \begin{pmatrix}
\frac{1}{r-k+1} \bm I_{r-k+1}  & \bm 0_{r-k+1, k-1} & \bm 0_{r-k+1,n-r} \\
\bm 0_{k-1, r-k+1} & \bm I_{k-1} & \bm 0_{k-1, n-r}\\
\bm 0_{n-r, r-k+1} & \bm 0_{n-r, k-1} & \bm 0_{n-r, n-r}\\
\end{pmatrix}.
\end{align*}

We prove the extremeness of matrix $\bm X^*$ by contradiction. Suppose that there exist two distinct points $\bm X_1, \bm X_2$ in set $\T$ such that $\bm X^*$ is equal to their convex combination, i.e., 
$$\exists 0<\alpha<1, \bm X^* = \alpha \bm X_1+(1-\alpha)\bm X_2.$$

First, we show that $\diag(\bm X^*)= \diag(\bm X_1) = \diag(\bm X_2)$. The inclusion  $\bm X_1, \bm X_2\in \T$ indicates $(X_1)_{ii} = (X_2)_{ii}=1/(r-k+1)$ for all $i \in [r-k]$.
Due to the constraint $\|\bm X\|_2 \le 1$ in set $\T$, we must have $(X_1)_{ii} = (X_2)_{ii}=1$ for all $i \in [r-k+2, r]$. In addition, according to Part (i) in \Cref{lem:nuc}, points $\bm X_1, \bm X_2$ must satisfy
\[ \sum_{i\in [n]} |(X_1)_{ii}|\le  \|\bm X_1\|_* \le k, \ \ \sum_{i\in [n]} |(X_2)_{ii}|\le  \|\bm X_2\|_* \le k,  \]
which further results in equations $(X_1)_{ii}=(X_2)_{ii}=X^*_{ii}$ for all $i\in \{r-k+1\}\cup [r+1,n]$.

Next, if there is a pair $(i^*,j^*) \in [r-k+1]\times [r-k+1]$ and $i^* > j^*$ such that $(X_1)_{i^*j^*}\neq 0$ and $(X_2)_{i^*j^*}\neq 0$, then we let $S:=\{i^*,j^*\}$. According to Part (ii) in \Cref{lem:nuc}, we have
\[\|\bm X_1\|_* \ge  \sum_{i\in [n]\setminus S} |(X_1)_{ii}| +  \|(\bm X_1)_{S,S}\|_* > \sum_{i\in [n]} |(X_1)_{ii}|  =k, \]
where the greater inequality is because for any $2\times 2$ submatrix $(\bm X_1)_{S,S}$ with only  one nonzero off-diagonal entry, we must have $ \|(\bm X_1)_{S,S}\|_* > \sum_{i\in S} |(X_1)_{ii}|$. Similarly, we can show $\|\bm X_2\|_* > k$. This thus forms a contraction with the nuclear norm constraint in set $\T$. 

If there is a pair $(i^*,j^*) \in [r-k+2, r]\times [r-k+2, r]$ and $i^* \neq j^*$ such that $(X_1)_{i^*j^*}\neq 0$ and $(X_2)_{i^*j^*}\neq 0$, then we let $S:=\{i^*,j^*\}$. Since $(X_1)_{i^*, i^*}= (X_1)_{j^*, j^*}=(X_2)_{i^*, i^*}= (X_2)_{j^*, j^*}=1$, we have
\[(\bm X_1)_{S,S}:= \begin{pmatrix}
    1 & (X_1)_{i^*j^*}\\
    (X_1)_{j^*i^*} & 1
\end{pmatrix}, \ \  (\bm X_2)_{S,S}:= \begin{pmatrix}
    1 & (X_2)_{i^*j^*}\\
    (X_2)_{j^*i^*} & 1
\end{pmatrix}. \]
Then, the simple calculation leads to  $ \|(\bm X_1)_{S,S}\|_2 >1$ and $ \|(\bm X_2)_{S,S}\|_2 >1$, which violates the largest singular value constraint in  set $\T$.

Finally, if there is a pair $(i^*,j^*) \in [r+1, n]\times [r+1, n]$ and $i^* \neq j^*$ such that $(X_1)_{i^*j^*}\neq 0$ and $(X_2)_{i^*j^*}\neq 0$, then we let $S:=\{i^*,j^*\}$. Since $(X_1)_{i^*, i^*}= (X_1)_{j^*, j^*}=(X_2)_{i^*, i^*}= (X_2)_{j^*, j^*}=0$, we have
\[(\bm X_1)_{S,S}:= \begin{pmatrix}
    0 & (X_1)_{i^*j^*}\\
    (X_1)_{j^*i^*} & 0
\end{pmatrix}, \ \  (\bm X_2)_{S,S}:= \begin{pmatrix}
    0 & (X_2)_{i^*j^*}\\
    (X_2)_{j^*i^*} & 0
\end{pmatrix}. \]
Similarly, using Part (ii) of \Cref{lem:nuc},  we can prove $ \|(\bm X_1)_{S,S}\|_* > \sum_{i\in S} |(X_1)_{ii}|$,  $\|\bm X_1\|_* > k$ and so does matrix $\bm X_2$, which is a contradiction.

Combining the above results together, we must have  $(X_1)_{i,j}= 0$ and $(X_2)_{i,j}= 0$ for all $i \neq j$, which means $\bm X^* = \bm X_1 =\bm X_2$. Thus, we complete the proof, provided that $\bm X^*$ is a rank-$r$ extreme point in the feasible set $\T$.

\noindent\textbf{Part (iii).} When function $f(\cdot)$ in the domain set $\X$ with $\Q:=\S^n$ is sign-invariant, we instead use the absolute eigenvalues, i.e., singular values; thus, the analysis of Part (ii) above can be readily extended, expecting replacing the non-symmetric matrix space with symmetric one. That is, we can construct a worst-case example below where a rank-$r$ extreme point exists. 
\begin{example}\label{eg:symtight}
Suppose the domain set $\X:=\{\bm X \in \S^n: \|\bm X\|_2 \le 1, \rank(\bm X)\le k\}$ and $m =(r-k+1)(r-k+2)/2-1$ linear equations in \ref{eq_rank}: 
\begin{align*}
X_{ii} = \frac{1}{r-k+1}, \forall i \in [r-k],  \ \  X_{ij} =0, \forall i, j \in [r-k+1]\times [r-k+1], i< j.
\end{align*} \qed
\end{example}
\end{proof}

\begin{theorem}
Given a sparse domain set $\X $ in \eqref{set:sparse} and an integer $\k\le k$ following \Cref{def:k}, we have  that
	\begin{enumerate}[(i)]
		\item If function $f(\cdot)$ in \eqref{set:sparse}  is sign-invariant, then there exists an \ref{eq_rel_rank} example that contains a rank-$(\k+\tilde m)$ extreme point; and
		\item There exists an \ref{eq_rel_rank1} example that contains a rank-$r$ extreme point.
	\end{enumerate}
\end{theorem}
\begin{proof}
Let us introduce an example below where function $f(\cdot)$ is sign-invariant.
\begin{example}\label{egsparse}
	Suppose a sparse domain set $\X:=\{\bm X \in \S^n: \bm X = \Diag(\diag(\bm X)), \rank(\bm X) \le k, \|\bm X\|_2 \le 1\}$ and $m=\tilde m$ linear equations in \ref{eq_rank}: 
	\begin{align*}
		X_{ii} = \frac{1}{m+1}, \forall i \in [m], 
	\end{align*} 
which leads to $\X:=\{\bm X \in \S^n: \bm X = \Diag(\diag(\bm X)), \|\bm X\|_* \le k, \|\bm X\|_2 \le 1\}$. 
\end{example}
Then, we can show that $\bm X^*=\diag(\bm x^*)$ is rank-$(k+\tilde m)$ and is an extreme point in the feasible set of \ref{eq_rel_rank}, where $x_i^*= 1/(m+1)$ for all $i\in [m+1]$, $x_i^*=1$ for all $i\in [m+2,m+k]$, and $x_i^*=0$ for all $i\in[m+k,n]$.  Thus, the rank bound in \Cref{them:sparsesign} is tight.

\Cref{them:sparsenonsign} provides an identical rank bound for any sparse domain set in \eqref{set:sparse}; hence,  \Cref{egsparse} also serves as a worst-case to show the tightness of \Cref{them:sparsenonsign}.
\qed
\end{proof}

\end{appendices}

\end{document}